\documentclass[10pt]{amsart}

\usepackage{amsmath,amsfonts,amsthm,mathrsfs,amssymb,verbatim,cite,fouridx}
\usepackage[left=3.5cm,right=3.5cm,top=3.5cm,bottom=3.5cm]{geometry}
\usepackage[usenames]{color}
\usepackage{tikz,graphicx,calligra}
\usepackage{esint}
\usepackage{todonotes}
\usepackage{hyperref}
\usepackage{cases}
\usepackage{bbm}

\usepackage[usenames]{color}

\allowdisplaybreaks[3]

\newtheorem{theorem}{Theorem}[section]
\newtheorem{lemma}[theorem]{Lemma}
\newtheorem{proposition}[theorem]{Proposition}
\newtheorem{corollary}[theorem]{Corollary}
\newtheorem{conjecture}[theorem]{Conjecture}

\theoremstyle{remark}
\newtheorem{remark}[theorem]{Remark}
\theoremstyle{definition}

\numberwithin{equation}{section}

\renewcommand{\sc}{\scriptstyle}
\renewcommand{\leq}{\leqslant}
\renewcommand{\geq}{\geqslant}

\DeclareMathOperator{\Tr}{Tr}
\DeclareMathOperator{\Vir}{Vir}
\DeclareMathOperator{\profile}{profile}

\DeclareMathOperator{\AG}{AG}
\DeclareMathOperator{\GK}{GK}
\DeclareMathOperator{\gk}{gk}
\DeclareMathOperator{\GKt}{\gk}

\DeclareMathOperator{\sgn}{sgn}
\DeclareMathOperator{\eup}{e}

\newcommand{\ceil}[1]{\lceil#1\rceil}
\newcommand{\floor}[1]{\lfloor#1\rfloor}
\newcommand{\qbin}[2]{\genfrac{[}{]}{0pt}{}{#1}{#2}}
\newcommand{\qbinbig}[2]{\Bigg[\genfrac{}{}{0pt}{}{#1}{#2}\Bigg]}
\newcommand{\qhyp}[2]{\fourIdx{}{#1}{}{#2}\phi}
\newcommand{\qHyp}[5]{\fourIdx{}{#1}{}{#2}\phi
                      \bigg[\genfrac{}{}{0pt}{}{#3}{#4};#5\bigg]}

\newcommand{\abs}[1]{\vert#1\vert}
\newcommand{\la}{\lambda}
\newcommand{\nub}{\boldsymbol{\nu}}
\newcommand{\La}{\Lambda}
\newcommand{\ip}[2]{\langle#1,#2\rangle}
\newcommand{\bil}[2]{(#1\vert#2)}

\DeclareMathOperator{\lev}{lev}
\DeclareMathOperator{\ch}{ch}
\DeclareMathOperator{\mult}{mult}

\tikzset{sq/.store in=\sq, sq=sqrt(3)/2}

\tikzset{bsquare/.pic={
\filldraw[gray,fill=white,very thin] (0,0)--({\sq},-1/2)--({\sq},1/2)--(0,1)--cycle;}}

\tikzset{rsquare/.pic={
\filldraw[gray,fill=white,very thin] (0,0)--({sqrt(3)/2},1/2)--({sqrt(3)/2},3/2)--(0,1)--cycle;}} 

\tikzset{gsquare/.pic={
\filldraw[gray,fill=white!30,very thin] (0,0)--({sqrt(3)/2},-1/2)--({sqrt(3)},0)--({sqrt(3)/2},1/2)--cycle;}}

\tikzset{bluesquare/.pic={
\filldraw[gray,fill=blue!30,very thin] (0,0)--({\sq},-1/2)--({\sq},1/2)--(0,1)--cycle;}}

\tikzset{bluesquarew/.pic={
\filldraw[gray,fill=blue!0,very thin] (0,0)--({\sq},-1/2)--({\sq},1/2)--(0,1)--cycle;}}

\tikzset{redsquare/.pic={
\filldraw[gray,fill=red!30,very thin] (0,0)--({sqrt(3)/2},1/2)--({sqrt(3)/2},3/2)--(0,1)--cycle;}} 

\tikzset{redsquarew/.pic={
\filldraw[gray,fill=red!0,very thin] (0,0)--({sqrt(3)/2},1/2)--({sqrt(3)/2},3/2)--(0,1)--cycle;}}

\tikzset{greensquare/.pic={
\filldraw[gray,fill=green!30,very thin] (0,0)--({sqrt(3)/2},-1/2)--({sqrt(3)},0)--({sqrt(3)/2},1/2)--cycle;}}

\tikzset{greensquarew/.pic={
\filldraw[gray,fill=green!0,very thin] (0,0)--({sqrt(3)/2},-1/2)--({sqrt(3)},0)--({sqrt(3)/2},1/2)--cycle;}}

\begin{document}

\title{The $\mathrm{A}_2$ Andrews--Gordon identities and cylindric partitions}

\author{S.\ Ole Warnaar}
\address{School of Mathematics and Physics,
The University of Queensland, Brisbane, QLD 4072, Australia}
\email{o.warnaar@maths.uq.edu.au}

\thanks{Work supported by the Australian Research Council.}

\subjclass[2020]{05A15, 05A19, 11P84, 17B65, 33D15, 81R10}

\begin{abstract}
Inspired by a number of recent papers by Corteel, Dousse, Foda, Uncu and 
Welsh on cylindric partitions and Rogers--Ramanujan-type identities, 
we obtain the $\mathrm{A}_2$ (or $\mathrm{A}_2^{(1)}$) 
analogues of the celebrated Andrews--Gordon identities.
We further prove $q$-series identities that correspond to the
infinite-level limit of the Andrews--Gordon identities
for $\mathrm{A}_{r-1}$ (or $\mathrm{A}_{r-1}^{(1)}$) for arbitrary rank $r$.
Our results for $\mathrm{A}_2$ also lead to conjectural, manifestly 
positive, combinatorial formulas for the $2$-variable generating function 
of cylindric partitions of rank $3$ and level $d$, such that $d$ is not a
multiple of $3$.

\smallskip

\noindent\textbf{Keywords:} $\mathrm{A}_{r-1}^{(1)}$ branching functions, 
character formulas, cylindric partitions, Rogers--Ramanujan identities. 
\end{abstract}

\maketitle

\section{Introduction}

The Rogers--Ramanujan identities \cite{Rogers94,Rogers17,RR19}
\begin{subequations}\label{Eq_RR}
\begin{align}\label{Eq_RR1}
\sum_{n=0}^{\infty} \frac{q^{n^2}}{(1-q)\cdots(1-q^n)}&=
\prod_{n=0}^{\infty}\frac{1}{(1-q^{5n+1})(1-q^{5n+4})}
\intertext{and}
\sum_{n=0}^{\infty} \frac{q^{n^2+n}}{(1-q)\cdots(1-q^n)}&=
\prod_{n=0}^{\infty}\frac{1}{(1-q^{5n+2})(1-q^{5n+3})}
\label{Eq_RR2}
\end{align}
\end{subequations}
are widely regarded as two of the deepest and most beautiful $q$-series 
identities in all of mathematics.
They play an important role in the theory of partitions 
\cite{Andrews76,Schur17,MacMahon04}, arise as characters in the
representation theory of infinite dimensional Lie algebras and vertex 
operator algebras
\cite{FF93,LM78a,LM78b,LW81b,LW82,LW84,MP87},
and have appeared in numerous other branches of mathematics.
The reader is referred to the recent book by Sills \cite{Sills18} 
for a comprehensive account of the Rogers--Ramanujan identities. 

A partition $\la$ of $n$ is a weakly decreasing 
sequence $\la=(\la_1,\la_2,\dots)$ of nonnegative integers such that
$\abs{\la}:=\la_1+\la_2+\cdots=n$.
If $m_i=m_i(\la):=\abs{\{j\geq 1: \la_j=i\}}$ denotes the multiplicity of 
the parts of $\la$ equal to $i$, then $\abs{\la}=\sum_{i\geq 1}i m_i$.
Schur and MacMahon \cite{Schur17,MacMahon04} independently observed that
the Rogers--Ramanujan identities are equivalent to the following 
combinatorial statement about partitions.
For fixed $s\in\{1,2\}$, the number of partitions of $n$ such that consecutive 
parts differ by at least two and such that $m_1\leq s-1$ is equal to the 
number of partitions of $n$ such that all parts are congruent to 
$\pm (3-s)$ modulo $5$.
Here $s=2$ corresponds to \eqref{Eq_RR1} and $s=1$ to \eqref{Eq_RR2}.
Gordon \cite{Gordon61} generalised the combinatorial form of the 
Rogers--Ramanujan identities to arbitrary odd moduli $2k+1$,
proving that for $1\leq s\leq k$ the number of partitions of $n$ 
such that $m_i+m_{i+1}\leq k-1$ (for all $i\geq 1$) and $m_1\leq s-1$ is
equal to the number of partitions of $n$ into parts not congruent to
$0,\pm s$ modulo $2k+1$.
Subsequently, Andrews \cite{Andrews74} discovered the analytic counterpart
of Gordon's partition theorem, proving that 
\begin{equation}\label{Eq_AG}
\sum_{n_1\geq n_2\geq\cdots\geq n_{k-1}\geq 0}
\frac{q^{n_1^2+\dots+n_{k-1}^2+n_s+\cdots+n_{k-1}}}
{(q)_{n_1-n_2}\cdots (q)_{n_{k-2}-n_{k-1}}(q)_{n_{k-1}}} 
=\frac{(q^{2k+1};q^{2k+1})_{\infty}}{(q)_{\infty}}\,\theta(q^s;q^{2k+1}).
\end{equation}
Here $(a)_n=(a;q)_n:=(1-a)(1-aq)\cdots(1-aq^{n-1})$ and
$(a)_{\infty}=(a;q)_{\infty}:=(1-a)(1-aq)\cdots$ are
$q$-shifted factorials, and $\theta(a;q):=(a;q)_{\infty}(q/a;q)_{\infty}$
is a modified Jacobi theta function.
The identities \eqref{Eq_AG} are now commonly referred to as the
Andrews--Gordon identities.

Let $\AG_{k,s}(q)$ denote the $q$-series in \eqref{Eq_AG}, so that
$\AG_{k,s}(q)$ may denote either the left- or the right-hand side of
the identity.
Apart from Gordon's two partition theoretic interpretations, the series
$\AG_{k,s}(q)$ for $0\leq s\leq k$ have also been identified as characters 
(or specialised characters) of affine Lie algebras.
Here we mention two such identifications. (For further examples, see e.g.,
\cite{Misra89,Misra90}.)
Let $\mathfrak{g}$ be an affine Lie algebra of rank $\ell$ with
simple roots $\alpha_0,\dots,\alpha_{\ell}$, simple coroots 
$\alpha_0^{\vee},\dots,\alpha_{\ell}^{\vee}$ and fundamental
weights $\La_0,\dots,\La_{\ell}$ (so that
$\langle\alpha_i^{\vee},\La_j\rangle=\delta_{ij}$),
where we adopt the labelling of the Dynkin diagrams as in
Tables 1 and 2 of \cite[Chapter 4]{Kac90}.
Let $P_{+}^d$ denote the set of level-$d$ dominant integral weights
of $\mathfrak{g}$ and let $\ch L(\la)$ be the character of the standard
(or integrable highest weight) module $L(\la)$ of highest weight $\la$.
Then, up to a simple infinite product 
$F:=(-q;q)_{\infty}=1/(q;q^2)_{\infty}$ corresponding
to the principally specialised character of the fundamental representation
$L(\La_a)$ ($a=0,1$),
$\AG_{k,s}(q)$ is equal to the principally specialised characters of
the affine Lie algebra $\mathrm{A}_1^{(1)}$ \cite[Theorem 5.16]{LM78b}:
\[
\eup^{-\la} \ch L(\la)\big|_
{(\eup^{-\alpha_0},\eup^{-\alpha_1})\mapsto (q,q)}=
F \cdot
\AG_{k,s}(q),
\] 
where $\la=(2k-s)\La_0+(s-1)\La_1\in P_{+}^{2k-1}$.
For $k=2$ this led Lepowsky and Wilson to the discovery of the principal
Heisenberg subalgebra $\mathfrak{s}$ of $\mathrm{A}_1^{(1)}$, culminating 
in the first purely Lie theoretic proof of the Rogers--Ramanujan identities 
\cite{LW81a,LW82}. 
In particular, Lepowsky and Wilson showed that $\AG_{2,s}(q)$ is exactly 
the character of the vacuum space of highest-weight vectors 
of $L((4-s)\La_0+(s-1)\La_1)$ with respect to $\mathfrak{s}$.
This was subsequently extended to higher-level standard modules of 
$\mathrm{A}_1^{(1)}$ by Meurman and Primc \cite{MP87} to yield a 
Lie theoretic proof of all of the Andrews--Gordon identities 
(as well as their even moduli analogues).\footnote{A distinction
between the results of \cite{LW81a,LW82} and \cite{MP87} is that
in the latter work, the combinatorial instead of the analytic
form of the identities is obtained.}

Along a different route, Griffin, Ono and the author showed that a
suitable non-principal specialisation of the characters of the
$\mathrm{A}_2^{(2)}$ standard modules of level $2k-2$ also leads to
the Rogers--Ramanujan and Andrews--Gordon $q$-series.
Specifically,
\[
\eup^{-\la} \ch L(\la)\big|_
{(\eup^{-\alpha_0},\eup^{-\alpha_1})\mapsto (-1,q)}
=\AG_{k,s}(q),
\]
where $\la$ is parametrised as $\la=(2k-2s)\La_0+(s-1)\La_1\in P_{+}^{2k-2}$.
This was then shown to generalise to $\mathrm{A}_{2r}^{(2)}$ for 
arbitrary $r\geq 1$, resulting in higher-rank generalisations of the 
Rogers--Ramanujan and Andrews--Gordon identities.
For example, if $k$ is a positive integer and $\kappa:=2k+2r-1$, 
then \cite[Theorem 1.1]{GOW16}
\begin{align}\label{Eq_GOW}
&\eup^{-(k-1)\La_r} \ch L\big((k-1)\La_r\big)\big|_
{(\eup^{-\alpha_0},\eup^{-\alpha_1},\dots,\eup^{-\alpha_r})
\mapsto (-1,q,\dots,q)} \\[2mm]
&\quad\qquad=\sum_{\substack{\la \text{ even}\\[1pt] \la_1\leq 2k-2}}
q^{\abs{\la}/2} P_{\la}\big(1,q,q^2,\dots;q^{2r-1}\big)  \notag \\
&\quad\qquad=\frac{(q^{\kappa};q^{\kappa})_{\infty}^r}{(q)_{\infty}^r}\,
\prod_{i=1}^r  \theta\big(q^{i+k-1};q^{\kappa}\big)
\prod_{1\leq i<j\leq r} 
\theta\big(q^{j-i},q^{i+j-1};q^{\kappa}\big), \notag
\end{align}
where $P_{\la}(x_1,x_2,\dots;t)$ is a Hall--Littlewood symmetric function
in infinitely many variables \cite{Macdonald95}.
For $r=1$ the identity \eqref{Eq_GOW} simplifies to the $s=k$ instance
of \eqref{Eq_AG}.

Surprisingly, finding the higher-rank generalisation of the Andrews--Gordon 
for the seemingly simpler affine Lie algebra $\mathrm{A}_{r-1}^{(1)}$ 
remains an open problem.
Parametrising the dominant integral weights of $\mathrm{A}_{r-1}^{(1)}$
as
\begin{equation}\label{Eq_la-mu}
\la=(d-\mu_1+\mu_r)\La_0+(\mu_1-\mu_2)\La_1+\dots+(\mu_{r-1}-\mu_r)\La_{r-1}
\in P_{+}^d,
\end{equation}
where $\mu=(\mu_1,\dots,\mu_r)$ is a partition such that $\mu_1-\mu_r\leq d$,
the challenge is to find generalisations $\AG_{\la;r}(q)$ of the sum-side
of \eqref{Eq_AG} such that
\begin{align}
\label{Eq_higher-rank_AG}
\eup^{-\la} & \ch L(\la)\big|_{(\eup^{-\alpha_0},\dots,\eup^{-\alpha_{r-1}})
\mapsto (q,\dots,q)} \\
&=F\cdot \frac{(q^{d+r};q^{d+r})_{\infty}^{r-1}}{(q)_{\infty}^{r-1}} 
\prod_{1\leq i<j\leq r} 
\theta\big(q^{\mu_i-\mu_j+j-i};q^{d+r}\big) \notag \\
&=F\cdot \AG_{\la;r}(q), \notag
\end{align}
where now, for arbitrary $0\leq a\leq r-1$,
\[
F:=\eup^{-\La_a} \ch L(\La_a)
\big|_{(\eup^{-\alpha_0},\dots,\eup^{-\alpha_{r-1}})
\mapsto (q,\dots,q)}=
\frac{(q^r;q^r)_{\infty}}{(q;q)_{\infty}}.
\]
In \cite{ASW99} it was shown that the infinite product in 
\eqref{Eq_higher-rank_AG} (without the factor $F$) is the character
of a non-unitary $W_r$-module (with $W_2$ the Virasoro algebra) as
well as an $\mathrm{A}_{r-1}^{(1)}$ branching function, see
Section~\ref{Sec_Characters} for details.

Arguably the most powerful method for discovering and proving identities
of the Rogers--Ramanujan type is that of Bailey chains and lattices, 
see \cite{AAB87,Andrews84,Paule85,W01}. 
In \cite{ASW99} Andrews, Schilling and the author developed an
$\mathrm{A}_2$-analogue of the classical Bailey chain and applied this to the 
problem of finding Rogers--Ramanujan and Andrews--Gordon identities for 
$\mathrm{A}_2^{(1)}$.
Although this resulted in several infinite families of 
Rogers--Ramanujan-type identities, the $q$-series of \cite{ASW99} correspond 
to $\AG_{\la;3}(q)$ multiplied by an unwanted factor $1/(q)_{\infty}$, 
obscuring the fact that the coefficients of the former are all nonnegative.
Only for $d=2$ and $d=4$ could this unwanted factor be eliminated,
resulting in, for example,
\begin{align}\label{Eq_A2-mod5}
\AG_{\La_0+\La_1;3}(q)&=
\sum_{n=0}^{\infty} \frac{q^{n^2}}{(q)_n}=
\prod_{n=0}^{\infty}\frac{1}{(1-q^{5n+1})(1-q^{5n+4})}
\intertext{and}
\label{Eq_A2-mod7}
\AG_{2\La_0+\La_1+\La_2;3}(q)
&=\sum_{n,m=0}^{\infty} \frac{q^{n^2-nm+m^2}}{(q)_n}\,
\qbin{2n}{m} \\
&=\prod_{n=0}^{\infty}\frac{1}{(1-q^{7n+1})^2(1-q^{7n+3})(1-q^{7n+4})
(1-q^{7n+6})^2}, \notag
\end{align}
where $\qbin{n}{m}$ is a $q$-binomial coefficient, see \eqref{Eq_qbinomial}
below.
The first of these results is essentially trivial since, by level-rank 
duality, $\AG_{\La_0+\La_1;3}(q)=\AG_{2\La_0+\La_1;2}(q)$.
The second Rogers--Ramanujan identity \eqref{Eq_RR2} follows in a similar 
manner from $\AG_{2\La_0;3}(q)=\AG_{3\La_0;2}(q)$.
The identity \eqref{Eq_A2-mod7} is one of five\footnote{There are only four
characters for modulus $7$, but one of these admits two distinct double-sum
expressions.} 
modulus-$7$ identities for $\mathrm{A}_2^{(1)}$ 
known collectively as the $\mathrm{A}_2$ 
Rogers--Ramanujan identities, see \cite{ASW99,CW19,FFW08,W06}.

In a series of recent papers \cite{Corteel17,CDU20,CW19,FW16}
a new approach to the problem of finding manifestly positive
multisum expressions for $\AG_{\la;r}(q)$ has emerged, based on
cylindric partitions.
Cylindric partitions, first introduced by Gessel and Krattenthaler
in \cite{GK97}, are an affine analogue of plane partitions.
Using notation and terminology as defined in Section~\ref{Sec_Cylindric},
let $\GK_c(q)$ be the size (or norm) generating function of 
cylindric partitions of rank $r$ and profile
$c=(c_0,\dots,c_{r-1})$:
\[
\GK_c(q):=\sum_{\substack{\pi \\[1pt] \profile(\pi)=c}} q^{\abs{\pi}},
\]
where $\abs{\pi}$ denotes the size of $\pi$.
The first key observation, due to Foda and Welsh \cite{FW16} (see also
\cite{Tingley08}), is that Borodin's product formula \cite{Borodin07}
for cylindric partitions implies that
\begin{equation}\label{Eq_AG-GK}
\AG_{\la;r}(q)=(q)_{\infty} \GK_c(q).
\end{equation}
Here the entries of the profile $c$ are fixed in terms of $\lambda$ as
$c_i=\langle\alpha_i^{\vee},\la\rangle$, i.e., if $\lambda$ is parametrised
as in \eqref{Eq_la-mu}, then $c_0=d-\mu_1+\mu_r$ and $c_i=\mu_i-\mu_{i+1}$
for $1\leq i\leq r-1$.
As an important consequence of \eqref{Eq_AG-GK}, the combinatorics of
cylindric partitions can be utilised to compute $\AG_{\la;r}(q)$.
The second key observation, due to Corteel and Welsh \cite{CW19}, is that
by considering the more general two-variable generating function
$\GK_c(z,q)$, 
\[
\GK_c(z,q):=\sum_{\substack{\pi \\[1pt] \profile(\pi)=c}} z^{\max(\pi)}
q^{\abs{\pi}},
\]
where $\max(\pi)$ is the largest largest part of the 
cylindric partition $\pi$, a system of functional equations arises that 
fully determines $\GK_c(z,q)$ for fixed level $c_0+\cdots+c_{r-1}=d$.
Solving this system for arbitrary rank $r$ and level $d$ appears extremely
challenging, but in a recent paper Corteel, Dousse and Uncu \cite{CDU20}
managed to solve the case $r=3$ and $d=5$, giving rise to manifestly
positive multisum expression for all five characters.
For example \cite[Theorem 3.3]{CDU20},
\[
\GK_{(2,2,1)}(z,q)=\frac{1}{(zq)_{\infty}}
\sum_{n_1,n_2,n_3,n_4=0}^{\infty}
\frac{z^{n_1} q^{n_1^2+n_2^2+n_3^2+n_4^2-n_1n_2+n_2n_4}}
{(q)_{n_1}}\,\qbin{n_1}{n_2}\qbin{n_1}{n_4}\qbin{n_2}{n_3}, 
\]
which by \eqref{Eq_higher-rank_AG} and \eqref{Eq_AG-GK} implies 
the beautiful \cite[Theorem 1.6]{CDU20}
\begin{align}\label{Eq_A2-mod8}
\AG_{2\La_0+2\La_1+\La_2;3}(q)
&=\sum_{n_1,n_2,n_3,n_4=0}^{\infty}
\frac{q^{n_1^2+n_2^2+n_3^2+n_4^2-n_1n_2+n_2n_4}}
{(q)_{n_1}}\,\qbin{n_1}{n_2}\qbin{n_1}{n_4}\qbin{n_2}{n_3} \\
&=\prod_{n=0}^{\infty}\frac{1}{(1-q^{8n+1})^2(1-q^{8n+2})(1-q^{8n+4})^2
(1-q^{8n+6})(1-q^{8n+7})^2}. \notag
\end{align}

Motivated by the results of Corteel et al., we succeeded in finding
the analogues of the Andrews--Gordon identities for
$\mathrm{A}_2^{(1)}$ for all moduli not congruent to $0$ modulo $3$.
The simplest examples for each fixed modulus are given below,
where $\theta(a_1,\dots,a_k;q):=\theta(a_1;q)\cdots\theta(a_k;q)$.

\begin{theorem}[$\mathrm{A}_2$ Andrews--Gordon identities, I]
\label{Thm_A2-vac-bothmoduli}
For $k$ a positive integer, the following identity for 
$\AG_{k\La_0+k\La_1+(k-1)\La_2;3}(q)$ holds:
\begin{align}\label{Eq_A2-vac}
\sum_{\substack{n_1,\dots,n_k\geq 0 \\[1pt] m_1,\dots,m_{k-1}\geq 0}} 
&\frac{q^{n_k^2}}{(q)_{n_1}} 
\prod_{i=1}^{k-1} q^{n_i^2-n_im_i+m_i^2} \qbin{n_i}{n_{i+1}} 
\qbin{n_i-n_{i+1}+m_{i+1}}{m_i} \\[2mm]
&=\frac{(q^{3k+2};q^{3k+2})_{\infty}^2}{(q)_{\infty}^2}\,
\theta(q^k,q^{k+1},q^{k+1};q^{3k+2}), \notag
\end{align}
where $m_k:=2n_k$.
Similarly, for $\AG_{k\La_0+(k-1)\La_1+(k-1)\La_2;3}(q)$
there holds $\AG_{\La_0;3}(q)=1$ and
\begin{align}
\label{Eq_A2-vac-b}
\sum_{\substack{n_1,\dots,n_{k-1}\geq 0 \\[1pt] 
m_1,\dots,m_{k-1}\geq 0}} 
&\frac{q^{\sum_{i=1}^{k-1} (n_i^2-n_im_i+m_i^2)}}{(q)_{n_1}}\,
\qbin{2n_{k-1}}{m_{k-1}} 
\prod_{i=1}^{k-2} \qbin{n_i}{n_{i+1}}\qbin{n_i-n_{i+1}+m_{i+1}}{m_i} 
\\[2mm]
&=\frac{(q^{3k+1};q^{3k+1})_{\infty}^2}{(q)_{\infty}^2}\,
\theta(q^k,q^k,q^{k+1};q^{3k+1}) \notag
\end{align}
for $k\geq 2$.
\end{theorem}

The $k=1$ case of \eqref{Eq_A2-vac} is \eqref{Eq_A2-mod5},
the $k=2$ case of \eqref{Eq_A2-vac-b} is \eqref{Eq_A2-mod7}, but
the $k=2$ case of \eqref{Eq_A2-vac} does not give \eqref{Eq_A2-mod8}!
Instead we obtain the triple-sum identity
\begin{align}\label{Eq_A2-mod8-triplesum}
\sum_{n_1,m_1,n_2=0}^{\infty} & \frac{q^{n_1^2-n_1m_1+m_1^2+n_2^2}}{(q)_{n_1}}
\, \qbin{n_1}{n_2}\qbin{n_1+n_2}{m_1} \\
&=\prod_{n=0}^{\infty}\frac{1}{(1-q^{8n+1})^2(1-q^{8n+2})(1-q^{8n+4})^2
(1-q^{8n+6})(1-q^{8n+7})^2}. \notag
\end{align}
The first two moduli that have not appeared before are $10$ and $11$, 
for which we get
\begin{align*}
\sum_{n_1,m_1,n_2,m_2=0}^{\infty}&
\frac{q^{n_1^2-n_1m_1+m_1^2+n_2^2-n_2m_2+m_2^2}}{(q)_{n_1}}\,
\qbin{n_1}{n_2}\qbin{n_1-n_2+m_2}{m_1}\qbin{2n_2}{m_2} \\[2mm]
&=\prod_{n=0}^{\infty}\bigg(\frac{1}{(1-q^{10n+1})^2(1-q^{10n+2})^2
(1-q^{10n+4})(1-q^{10n+5})^2} \\
&\quad\qquad\qquad\qquad\qquad\times
\frac{1}{(1-q^{10n+6})(1-q^{10n+8})^2(1-q^{10n+9})^2}\bigg)
\end{align*}
and
\begin{align*}
\sum_{n_1,m_1,n_2,m_2,n_3=0}^{\infty}&
\frac{q^{n_1^2-n_1m_1+m_1^2+n_2^2-n_2m_2+m_2^2+n_3^2}}{(q)_{n_1}}\,
\qbin{n_1}{n_2}\qbin{n_1-n_2+m_2}{m_1}\qbin{n_2}{n_3}
\qbin{n_2+n_3}{m_2} \\[2mm]
&=\prod_{n=0}^{\infty}\bigg(\frac{1}{(1-q^{11n+1})^2(1-q^{11n+2})^2
(1-q^{11n+3})(1-q^{11n+5})^2} \\
&\quad\qquad\times
\frac{1}{(1-q^{11n+6})^2(1-q^{11n+8})(1-q^{11n+9})^2(1-q^{11n+10})^2}\bigg).
\end{align*}

\bigskip

The remainder of this paper is organised as follows.
In the next section we state our main results on $\mathrm{A}_2$
(or $\mathrm{A}_2^{(1)}$) Rogers--Ramanujan and Andrews--Gordon identities.
This section also contains conjectural, manifestly-positive multi-sum 
expressions for the two-variable generating function of cylindric 
partition of rank $3$.
These conjectures, if true, show that the statistic $\max$ on cylindric
partitions is compatible with our $\mathrm{A}_2^{(1)}$ character formulas,
suggesting that $\max$ should have a natural representation-theoretic
interpretation.
In Section~\ref{Sec_Cylindric} we review some of the basics of cylindric
partitions needed in the remainder of the paper, and in
Section~\ref{Sec_Characters} we give representation theoretic
interpretations of $\AG_{\la;r}(q)$.
In Section~\ref{Sec_hyper} a number of identities for basic hypergeometric
series are proved, laying the groundwork for Section~\ref{Sec_Mod8}, which
focuses on the modulus-$8$ case and provides a connection with the work of
Corteel, Dousse and Uncu.
In Section~\ref{Sec_Proofs} we prove our main results, including the
$\mathrm{A}_2$ Andrews--Gordon identities of
Theorem~\ref{Thm_A2-vac-bothmoduli}.
Finally, in Theorem~\ref{Thm_infinite} of Section~\ref{Sec_Ar} we 
present an identity for $\mathrm{A}_{r-1}$ corresponding to the
infinite-level limit of the as-yet-to-be found Andrews--Gordon identities
for $\mathrm{A}_{r-1}^{(1)}$ for $r\geq 4$.

\section{Main results and conjectures}\label{Sec_Main}

In this section we give a summary of the $\mathrm{A}_2$ Andrews--Gordon 
identities for $\AG_{\la;3}(q)$ as well as some closely related, manifestly 
positive multisum expressions for the two-variable generating function of 
cylindric partitions $\GK_{(c_0,c_1,c_2)}(z,q)$ for
$c_0+c_1+c_2\not\equiv 0\pmod{3}$.

\subsection{The modulus-$(3k+2)$ case}
The simplest $\mathrm{A}_2$ Andrews--Gordon identity for the 
modulus $3k+2$ is \eqref{Eq_A2-vac} of the introduction.
This result is complemented with several further identities, the
first of which corresponds to $\AG_{(3k-s)\La_0+(s-1)\La_1;3}(q)$.

\begin{conjecture}[$\mathrm{A}_2$ Andrews--Gordon identities, II]
\label{Con_A2-two}
For integers $k,s$ such that $1\leq s\leq k+1$,
\begin{align}\label{Eq_ks-con}
\sum_{\substack{n_1,\dots,n_k\geq 0 \\[1pt] m_1,\dots,m_{k-1}\geq 0}}
& \frac{q^{n_k^2+\sum_{i=s}^k n_i}}{(q)_{n_1}} 
\prod_{i=1}^{k-1} q^{n_i^2-n_im_i+m_i^2+m_i} \qbin{n_i}{n_{i+1}} 
\qbin{n_i-n_{i+1}+m_{i+1}}{m_i} \\
&=\frac{(q^{3k+2};q^{3k+2})_{\infty}^2}{(q)_{\infty}^2}\,
\theta(q,q^s,q^{s+1};q^{3k+2}), \notag
\end{align}
where $m_k:=2n_k$.
\end{conjecture}

\begin{theorem}\label{Thm_threecases}
Conjecture~\ref{Con_A2-two} holds for $s\in\{1,k,k+1\}$.
\end{theorem}

\begin{remark}\label{Rem_alt-sums}
The proofs of the various $\mathrm{A}_2$ Andrews--Gordon identities all
use the Rogers--Ramanujan-type identities of \cite{ASW99} as a seed.
In these seeds, the summation variables $n_i$ and $m_i$ play a more
symmetric role than in the $\mathrm{A}_2$ Andrews--Gordon identities.
If a seed has a summand that is near-symmetric in the sense that it is
invariant when $n_i$ and $m_i$ are interchanged for all $i$,
except for a linear factor in the exponent of $q$, then this
near-symmetry can be exploited to yield two different forms for the 
sum side of the corresponding $\mathrm{A}_2$ Andrews--Gordon identity.
What is more, by conjecturing some near-symmetric seeds missing from 
\cite{ASW99} (see Conjecture~\ref{Con_missing}) this extends to 
Conjecture~\ref{Con_A2-two} in full.
Hence this conjecture admits the companion 
\begin{align}\label{Eq_ks-alt}
\sum_{\substack{n_1,\dots,n_k\geq 0 \\[1pt] m_1,\dots,m_{k-1}\geq 0}} &
\frac{q^{n_k^2+n_k+\sum_{i=s}^{k-1} m_i}}{(q)_{n_1}} 
\prod_{i=1}^{k-1} q^{n_i^2-n_im_i+m_i^2+n_i}
\qbin{n_i}{n_{i+1}}\qbin{n_i-n_{i+1}+m_{i+1}+\delta_{i,s-1}}{m_i} \\
&=\frac{(q^{3k+2};q^{3k+2})_{\infty}^2}{(q)_{\infty}^2}\,
\theta(q,q^s,q^{s+1};q^{3k+2}) \notag
\end{align}
for $1\leq s\leq k$, and
\begin{align}\label{Eq_kk1-alt}
\sum_{\substack{n_1,\dots,n_k\geq 0 \\[1pt] m_1,\dots, m_{k-1}\geq 0}} &
\frac{q^{n_k^2}}{(q)_{n_1}} \prod_{i=1}^{k-1} q^{n_i^2-n_im_i+m_i^2+n_i}
\qbin{n_i+\delta_{i,k-1}}{n_{i+1}}
\qbin{n_i-n_{i+1}+m_{i+1}}{m_i} \\
&=\frac{(q^{3k+2};q^{3k+2})_{\infty}^2}{(q)_{\infty}^2}\,
\theta(q,q^{k+1},q^{k+2};q^{3k+2}) \notag
\end{align}
when $s=k+1$.
In both cases $m_k:=2n_k$ as before and $\delta_{i,j}$ is a Kronecker
delta.
\end{remark}

The identity in the next theorem has a more complicated right-hand side,
corresponding to the linear combination
$\sum_{i=1}^s q^{s-i}\AG_{(3k-2i+1)\La_0+(i-1)\La_1+(i-1)\La_2;3}(q)$.

\begin{theorem}[$\mathrm{A}_2$ Andrews--Gordon identities, III]
\label{Thm_A2-three}
For integers $k,s$ such that $1\leq s\leq k$,
\begin{align*}
\sum_{\substack{n_1,\dots,n_k\geq 0 \\[1pt] m_1,\dots,m_{k-1}\geq 0}} &
\frac{q^{n_k^2+\sum_{i=s}^k n_i+\sum_{i=s}^{k-1} m_i}}{(q)_{n_1}} 
\prod_{i=1}^{k-1} q^{n_i^2-n_im_i+m_i^2} \qbin{n_i}{n_{i+1}} 
\qbin{n_i-n_{i+1}+m_{i+1}+\delta_{i,s-1}}{m_i} \\
&=\frac{(q^{3k+2};q^{3k+2})_{\infty}^2}{(q)_{\infty}^2}
\sum_{i=1}^s q^{s-i} \,\theta(q^i,q^i,q^{2i};q^{3k+2}), \notag
\end{align*}
where $m_k:=2n_k$.
\end{theorem}

\medskip

The above results give $3k+1$ distinct identities for modulus $3k+2$, except
when $k=1$, in which case the two Rogers--Ramanujan identities 
\eqref{Eq_RR1} and \eqref{Eq_RR2} arise.
In the modulus-$8$ case we have found another three results giving a total
of $11$ identities, including \eqref{Eq_A2-mod8-triplesum} 
of the introduction.
All of these will be discussed in detail in Section~\ref{Sec_Mod8}.

\medskip

Conjecturally, some of the above results can be extended to the two-variable
generating function for cylindric partitions of rank $3$.

\begin{conjecture}\label{Con_cylindric}
For a positive integer $k$, 
\[
\GK_{(k,k,k-1)}(z,q) 
=\frac{1}{(zq)_{\infty}}
\sum_{\substack{n_1,\dots,n_k\geq 0 \\[1pt] m_1,\dots,m_{k-1}\geq 0}}
\frac{z^{n_1} q^{n_k^2}}{(q)_{n_1}} 
\prod_{i=1}^{k-1} 
q^{n_i^2-n_im_i+m_i^2}\qbin{n_i}{n_{i+1}}\qbin{n_i-n_{i+1}+m_{i+1}}{m_i}
\]
and, for integers $k,s$ such that $1\leq s\leq k+1$,
\begin{align*}
&\GK_{(3k-s,s-1,0)}(z,q) \\
&\quad=\frac{1}{(zq)_{\infty}}
\sum_{\substack{n_1,\dots,n_k\geq 0 \\[1pt] m_1,\dots,m_{k-1}\geq 0}}
\frac{z^{n_1} q^{n_k^2+\sum_{i=s}^k n_i}}{(q)_{n_1}} 
\prod_{i=1}^{k-1} q^{n_i^2-n_im_i+m_i^2+m_i} \qbin{n_i}{n_{i+1}} 
\qbin{n_i-n_{i+1}+m_{i+1}}{m_i},
\end{align*}
where $m_k:=2n_k$ in both identities.
\end{conjecture}

\begin{theorem}\label{Thm_k12}
Conjecture~\ref{Con_cylindric} holds for $k=1$ and $k=2$. 
\end{theorem}

Conjecture~\ref{Con_cylindric} is consistent with the following
conjecture of Corteel et al., see \cite[Conjecture 4.2]{CDU20}.
For a nonnegative integer $n$ and $c=(c_0,\dots,c_{r-1})$, define
\begin{equation}\label{Eq_Qnc}
Q_{n,c}(q):=
(q^{\ell};q^{\ell})_n\,[z^n] \Big((zq)_{\infty}\GK_c(z,q)\Big)
\in\mathbb{Z}[[q]],
\end{equation}
where, for $d:=c_0+\dots+c_{r-1}$, $\ell:=\gcd(d,r)$.

\begin{conjecture}\label{Con_CDU-Q}
Let $c=(c_0,c_1,c_2)$ and $d:=c_0+c_1+c_2$ such that 
$d\not\equiv 0 \pmod{3}$.
Then $Q_{n,(c_0,c_1,c_2)}(q)$ is a polynomial in $q$ with 
nonnegative coefficients.
Moreover,
\begin{equation}\label{Eq_CDU-Q}
Q_{n,(c_0,c_1,c_2)}(1)=\Big(\tfrac{1}{6}(d+1)(d+2)-1\Big)^n.
\end{equation}
\end{conjecture}

The polynomiality and \eqref{Eq_CDU-Q} have both been proven by 
Welsh~\cite{Welsh21}, but the positivity part of the conjecture
remains wide open.
From Conjecture~\ref{Con_cylindric}, we immediately infer the following
manifestly positive representations for $d=3k-1$:
\[
Q_{n_1,(k,k,k-1)}(q)=
\sum_{\substack{n_2,\dots,n_k\geq 0 \\[1pt] m_1,\dots,m_{k-1}\geq 0}}
q^{n_k^2} \prod_{i=1}^{k-1} 
q^{n_i^2-n_im_i+m_i^2}\qbin{n_i}{n_{i+1}}\qbin{n_i-n_{i+1}+m_{i+1}}{m_i}
\]
and
\begin{align*}
&Q_{n_1,(3k-s,s-1,0)}(q) \\
&\qquad=\sum_{\substack{n_2,\dots,n_k\geq 0 \\[1pt] m_1,\dots,m_{k-1}\geq 0}}
q^{n_k^2+\sum_{i=s}^k n_i}
\prod_{i=1}^{k-1} q^{n_i^2-n_im_i+m_i^2+m_i} \qbin{n_i}{n_{i+1}} 
\qbin{n_i-n_{i+1}+m_{i+1}}{m_i},
\end{align*}
where $m_k:=2n_k$. 
For $q=1$ it is a standard exercise in binomial sums to show that this
implies
\[
Q_{n,(k,k,k-1)}(1)=Q_{n,(3k-s,s-1,0)}(1)=\Big(\tfrac{1}{2}k(3k+1)-1\Big)^n,
\]
in accordance with \eqref{Eq_CDU-Q}.

\subsection{The modulus-$(3k+1)$ case}

All the modulus-$(3k+2)$ identities have counterparts for modulus $3k+1$.
Our first result beyond \eqref{Eq_A2-vac-b} is a companion to
Conjecture~\ref{Con_A2-two} and gives Andrews--Gordon identities
for $\AG_{(3k-s-1)\La_0+(s-1)\La_1;3}(q)$.

\begin{conjecture}
[$\mathrm{A}_2$ Andrews--Gordon identities, II$'$]
\label{Con_A2-two-b}
For integers $k,s$ such that $k\geq 2$ and $1\leq s\leq k$,
\begin{align}\label{Eq_A2-two-b}
\sum_{\substack{n_1,\dots,n_{k-1}\geq 0 \\[1pt] m_1,\dots,m_{k-1}\geq 0}}
& \frac{q^{\sum_{i=1}^{k-1} (n_i^2-n_im_i+m_i^2+m_i)+
\sum_{i=s}^{k-1} n_i}}{(q)_{n_1}} \, \qbin{2n_{k-1}}{m_{k-1}}
\prod_{i=1}^{k-2} \qbin{n_i}{n_{i+1}} \qbin{n_i-n_{i+1}+m_{i+1}}{m_i} \\
&=\frac{(q^{3k+1};q^{3k+1})_{\infty}^2}{(q)_{\infty}^2}\,
\theta(q,q^s,q^{s+1};q^{3k+1}). \notag
\end{align}
\end{conjecture}

Again we have a proof for three (or two when $k=2$) values of $s$.

\begin{theorem}\label{Thm_threecases-b}
Conjecture~\ref{Con_A2-two-b} holds for $s\in\{1,k-1,k\}$.
\end{theorem}

If Conjecture~\ref{Con_A2-two-b} is true, then for $1\leq s\leq k$,
\begin{align}\label{Eq_ks-alt-b}
\sum_{\substack{n_1,\dots,n_{k-1}\geq 0 \\[1pt] m_1,\dots,m_{k-1}\geq 0}} 
\bigg(&\frac{
q^{\sum_{i=1}^{k-1} (n_i^2-n_im_i+m_i^2+n_i)+\sum_{i=s}^{k-1} m_i}}
{(q)_{n_1}} \, \qbin{2n_{k-1}+\delta_{s,k}}{m_{k-1}} 
\\[-1mm]
& \times 
\prod_{i=1}^{k-2} \qbin{n_i}{n_{i+1}} 
\qbin{n_i-n_{i+1}+m_{i+1}+\delta_{i,s-1}}{m_i}\bigg) \notag \\[2mm]
&\quad=\frac{(q^{3k+1};q^{3k+1})_{\infty}^2}{(q)_{\infty}^2}\,
\theta(q,q^s,q^{s+1};q^{3k+1}) \notag
\end{align}
in analogy with \eqref{Eq_ks-alt} and \eqref{Eq_kk1-alt}.
Equations \eqref{Eq_A2-vac-b}, \eqref{Eq_A2-two-b} and \eqref{Eq_ks-alt-b}
for $k=2$ give the four $\mathrm{A}_2$ Rogers--Ramanujan 
identities of \cite[Theorem 5.6]{ASW99}.
The fifth $\mathrm{A}_2$ Rogers--Ramanujan, which was conjectured in
\cite{FFW08} and proved in \cite[Theorem 1.2]{CW19}, is missing from
our modulus-$(3k+1)$ generalisations.

The next identity for character sum 
$\sum_{i=1}^s q^{s-i}\AG_{(3k-2i)\La_0+(i-1)\La_1+(i-1)\La_2;3}(q)$
is the analogue of Theorem~\ref{Thm_A2-three}.

\begin{theorem}
[$\mathrm{A}_2$ Andrews--Gordon identities, III$'$]
\label{Thm_A2-three-b}
For integers $k,s$ such that $1\leq s\leq k-1$,
\begin{align*}
&\sum_{\substack{n_1,\dots,n_{k-1}\geq 0 \\[1pt] m_1,\dots,m_{k-1}\geq 0}}
\frac{q^{\sum_{i=1}^{k-1}(n_i^2-n_im_i+m_i^2)+
\sum_{i=s}^{k-1}(n_i+m_i)}}{(q)_{n_1}}\,\qbin{2n_{k-1}}{m_{k-1}}
\prod_{i=1}^{k-2} \qbin{n_i}{n_{i+1}} 
\qbin{n_i-n_{i+1}+m_{i+1}+\delta_{i,s-1}}{m_i} \\
&\qquad\qquad\quad=
\frac{(q^{3k+1};q^{3k+1})_{\infty}^2}{(q)_{\infty}^2}
\sum_{i=1}^s q^{s-i}\,\theta(q^i,q^i,q^{2i};q^{3k+1}).
\end{align*}
\end{theorem}

The lifting of \eqref{Eq_A2-vac-b} and \eqref{Eq_A2-two-b} to the 
two-variable generating function for cylindric partitions results 
in our final conjecture.

\begin{conjecture}\label{Con_cylindric-b}
For $k$ a positive integer,
\begin{align*}
&\GK_{(k,k-1,k-1)}(z,q) \\
&\quad=\frac{1}{(zq)_{\infty}}
\sum_{\substack{n_1,\dots,n_{k-1}\geq 0 \\[1pt] m_1,\dots,m_{k-1}\geq 0}}
\frac{z^{n_1}q^{\sum_{i=1}^{k-1} (n_i^2-n_im_i+m_i^2)}}{(q)_{n_1}}\,
\qbin{2n_{k-1}}{m_{k-1}} 
\prod_{i=1}^{k-2} \qbin{n_i}{n_{i+1}}\qbin{n_i-n_{i+1}+m_{i+1}}{m_i} \\
\intertext{and, for integers integers $k,s$ such that $1\leq s\leq k$,}
&\GK_{(3k-s-1,s-1,0)}(z,q) \\
&\quad=\frac{1}{(zq)_{\infty}}
\sum_{\substack{n_1,\dots,n_{k-1}\geq 0 \\[1pt] m_1,\dots,m_{k-1}\geq 0}}
\bigg(\frac{z^{n_1} q^{\sum_{i=1}^{k-1} (n_i^2-n_im_i+m_i^2+m_i)+
\sum_{i=s}^{k-1} n_i}}{(q)_{n_1}} \\[-1mm]
& \qquad\qquad\qquad\qquad\qquad\quad\times \qbin{2n_{k-1}}{m_{k-1}}
\prod_{i=1}^{k-2} \qbin{n_i}{n_{i+1}} \qbin{n_i-n_{i+1}+m_{i+1}}{m_i}\bigg).
\end{align*}
\end{conjecture}

The $k=1$ case, corresponding to $\GK_{(1,0,0)}(z,q)=1/(zq)_{\infty}$,
is trivial.\footnote{Also the more general
$\GK_{(L^r)/(0^r)/1}(z,q)=1/(zq)_{rL}$ is trivially true,
where we refer to Section~\ref{Sec_Cylindric} for the definition
of $\GK_{\lambda/\mu/d}(z,q)$.
This implies that $\GK_{(1,0^{r-1})}(z,q)=1/(zq)_{\infty}$ for arbitrary $r$.} 
The $k=2$ case of Conjecture~\ref{Con_cylindric-b} is 
\cite[Theorem 3.2]{CW19} by Corteel and Welsh. 
Their theorem also includes manifestly positive double-sum expressions
for $\GK_{(3,0,1)}(z,q)$ and $\GK_{(2,2,0)}(z,q)$, which are not
included in the $k=2$ case of Conjecture~\ref{Con_cylindric-b}.

\section{Cylindric partitions}\label{Sec_Cylindric}

Let $\mathbb{N}$ and $\mathbb{N}_0$ denote the set of positive integers and 
nonnegative integers respectively.
Then an integer partition of size $n$ and length $r$ is a weakly decreasing
sequence $\la=(\la_1,\dots,\la_r)\in\mathbb{N}_0^r$ such that
$\abs{\la}:=\la_1+\dots+\la_r=n$.
Here the $\la_i$ are referred to as the parts of $\la$.
The above differs slightly from the standard definitions of part and
length \cite{Macdonald95} in that $0$ can be a part.
For example, in this paper the partitions $(7,6,6,4,0)$ and $(7,6,6,4)$ of 
$23$ will be viewed as distinct, having length $5$ and $4$ respectively.
We will alternatively use the multiplicities as exponent to represent 
a partition, leaving out the exponent $1$.
Hence $(7,6,6,4,0)=(7,6^2,4,0)$.

A Young diagram $Y$ is a configuration of (unit) squares or boxes that are 
arranged in left-justified rows such that row-lengths are weakly decreasing
from top to bottom, as in 

\medskip

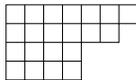
\begin{center}
\begin{tikzpicture}[scale=0.25,line width=0.3pt]
\draw (0,0)--(4,0)--(4,2)--(6,2)--(6,3)--(7,3)--(7,4)--(0,4)--cycle;
% horizontal lines
\draw (0,1)--(4,1);
\draw (0,2)--(4,2);
\draw (0,3)--(6,3);
% vertical lines
\draw (1,0)--(1,4);
\draw (2,0)--(2,4);
\draw (3,0)--(3,4);
\draw (4,2)--(4,4);
\draw (5,2)--(5,4);
\draw (6,3)--(6,4);
\end{tikzpicture}
\end{center}

\smallskip
If $\la=(\la_1,\dots,\la_r)$ is a partition and $Y$ a Young diagram 
of at most $r$ rows such that the $i$th row of $Y$ contains $\la_i$
squares, we say that $Y$ is the Young diagram of $\la$.
For example, the above diagram is the Young diagram of $(7,6^2,4,0^k)$
for arbitrary nonnegative integer $k$.
If $\nu$ is a partition such that its Young diagram is the transpose
of the Young diagram of $\la$, we say that $\nu$ is a conjugate of
$\la$ and write $\nu=\la'$. 
Of course, the conjugate of $\la$ is not unique and 
$(7,6^2,4,0^k)'=(4^3,2^2,1,0^{\ell})$ for $k,\ell$ nonnegative integers.
We note that the multiplicity $m_i=m_i(\la)$ of parts of size $i$ of the 
partition $\la=(\la_1,\dots,\la_r)$ is given by $m_0=r-\la'_1$ and
$m_i(\la)=\la'_i-\la'_{i+1}$ for $i\geq 1$.

\noindent
Given two partitions $\la,\mu$ of length $r$, write $\mu\subseteq\la$
if $\mu_i\leq\la_i$ for all $1\leq i\leq r$.
Then the skew (Young) diagram of $\la/\mu$ is obtained from the 
Young diagram of $\la$ by deleting all squares contained in the Young 
diagram of $\mu$.
Hence the skew diagram of $(7,6,4,4)/(3,1,1,0)$ is given by

\medskip

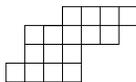
\begin{center}
\begin{tikzpicture}[scale=0.25,line width=0.3pt]
\draw (0,0)--(4,0)--(4,2)--(6,2)--(6,3)--(7,3)--(7,4)--(3,4)--(3,3)--(1,3)--(1,1)--(0,1)--cycle;
% horizontal lines
\draw (1,1)--(4,1);
\draw (1,2)--(4,2);
\draw (2,3)--(6,3);
% vertical lines
\draw (1,0)--(1,1);
\draw (2,0)--(2,3);
\draw (3,0)--(3,4);
\draw (4,2)--(4,4);
\draw (5,2)--(5,4);
\draw (6,3)--(6,4);
% filling
\end{tikzpicture}
\end{center}

\smallskip

\noindent
The skew diagram of $\la/\mu$ is connected if for all positive $\la_i$
(such that $i\neq 1$), $\mu_{i-1}<\la_i$.
Conversely, it is disconnected if there exists an $i$ such that
$\mu_i\geq\la_{i+1}>0$.

\medskip

Cylindric partitions were first introduced by Gessel and Krattenthaler in
\cite{GK97} as an affine analogue of (skew) plane partitions.
Fix a positive integer $r$, which we will refer to as the rank, and
let $\la,\mu$ be two partitions of length $r$ such that 
$\mu\subseteq\la$.\footnote{Gessel and Krattenthaler consider
more general integer sequences $\la$ and $\mu$ of length $r$, 
but we have no need for these here.}
Let $\mathscr{A}\subseteq \mathbb{N}_0$.
Then a plane partition $\pi$ on $\mathscr{A}$
of shape $\la/\mu$ and size $n$ is a filling of the diagram of 
$\la/\mu$ with elements from $\mathscr{A}$ such that the rows
and columns are both weakly decreasing and such that $\abs{\pi}$,
the sum of the entries of $\pi$, is equal to $n$.
For example,  

\smallskip

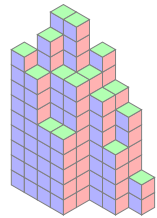
\usetikzlibrary{math}
\tikzmath{\s=0.5;}
\qquad
\begin{minipage}{0.5\textwidth}
\begin{center}
\begin{tikzpicture}[scale=0.45,line width=0.3pt]
\draw (0,0)--(4,0)--(4,2)--(6,2)--(6,3)--(7,3)--(7,4)--(3,4)--(3,3)--(1,3)--(1,1)--(0,1)--cycle;
% horizontal lines
\draw (1,1)--(4,1);
\draw (1,2)--(4,2);
\draw (2,3)--(6,3);
% vertical lines
\draw (1,0)--(1,1);
\draw (2,0)--(2,3);
\draw (3,0)--(3,4);
\draw (4,2)--(4,4);
\draw (5,2)--(5,4);
\draw (6,3)--(6,4);
% filling
\begin{scope}[color=red]
\draw (\s,\s) node {$9$};
\draw (1+\s,\s) node {$8$};
\draw (2+\s,\s) node {$5$};
\draw (3+\s,\s) node {$5$};
\draw (1+\s,1+\s) node {$10$};
\draw (2+\s,1+\s) node {$8$};
\draw (3+\s,1+\s) node {$8$};
\draw (1+\s,2+\s) node {$11$};
\draw (2+\s,2+\s) node {$11$};
\draw (3+\s,2+\s) node {$8$};
\draw (4+\s,2+\s) node {$7$};
\draw (5+\s,2+\s) node {$4$};
\draw (3+\s,3+\s) node {$9$};
\draw (4+\s,3+\s) node {$7$};
\draw (5+\s,3+\s) node {$6$};
\draw (6+\s,3+\s) node {$2$};
\end{scope}
\end{tikzpicture}
\end{center}
\end{minipage}
\begin{minipage}{0.3\textwidth}
\begin{center}
\begin{tikzpicture}[scale=0.20,transform shape]
\foreach \h in {3,5,...,19}{\pic at ({-1*\sq},\h/2) {bluesquare};}
\foreach \h in {1,2,...,8}{\pic at ({0*\sq},\h) {bluesquare};}
\foreach \h in {1,3,...,9,19,21}{\pic at ({1*\sq},\h/2) {bluesquare};}
\foreach \h in {0,1,...,4,6,7,8,12}{\pic at ({2*\sq},\h) {bluesquare};}
\foreach \h in {11,13,15,19,21,23}{\pic at ({3*\sq},\h/2) {bluesquare};}
\foreach \h in {1,3,...,13,19}{\pic at ({5*\sq},\h/2) {bluesquare};}
\foreach \h in {0,1,...,3}{\pic at ({6*\sq},\h) {bluesquare};}
\foreach \h in {9,11}{\pic at ({7*\sq},\h/2) {bluesquare};}
\foreach \h in {0,1}{\pic at ({8*\sq},\h) {bluesquare};}
\pic at ({0*\sq},9) {redsquare};
\foreach \h in {11,13,15}{\pic at ({1*\sq},\h/2) {redsquare};}
\foreach \h in {9,10}{\pic at ({2*\sq},\h) {redsquare};}
\foreach \h in {-1,1,...,7}{\pic at ({3*\sq},\h/2) {redsquare};}
\foreach \h in {0,1,...,7,9,10,11}{\pic at ({4*\sq},\h) {redsquare};}
\pic at ({5*\sq},15/2) {redsquare};
\foreach \h in {4,5,6,8,9}{\pic at ({6*\sq},\h) {redsquare};}
\foreach \h in {-1,1,...,5,13}{\pic at ({7*\sq},\h/2) {redsquare};}
\foreach \h in {2,3,...,5}{\pic at ({8*\sq},\h) {redsquare};}
\foreach \h in {-1,1}{\pic at ({9*\sq},\h/2) {redsquare};}
\pic at ({-1*\sq},21/2) {greensquare};
\pic at ({0*\sq},9) {greensquare};
\foreach \h in {11,23}{\pic at ({1*\sq},\h/2) {greensquare};}
\foreach \h in {5,9,13}{\pic at ({2*\sq},\h) {greensquare};}
\foreach \h in {17,25}{\pic at ({3*\sq},\h/2) {greensquare};}
\pic at ({4*\sq},9) {greensquare};
\foreach \h in {15,21}{\pic at ({5*\sq},\h/2) {greensquare};}
\foreach \h in {4,8}{\pic at ({6*\sq},\h) {greensquare};}
\pic at ({7*\sq},13/2) {greensquare};
\pic at ({8*\sq},2) {greensquare};
\end{tikzpicture}
\end{center}
\end{minipage}

\bigskip

\noindent
is a plane partition on $\mathbb{N}$ of shape $(7,6,4,4)/(3,1,1,0)$ 
and size $118$ (which we blatantly copied from \cite{GK97}).
The representation on the right corresponds to the usual stacking of
unit cubes so that $\abs{\pi}$ corresponds to the volume of $\pi$.
When $\mu=0$ and $\mathscr{A}=\mathbb{N}$ one obtains an ordinary plane 
partition of shape $\la$ in the sense of MacMahon \cite{MacMahon97}.
Each row and column of a plane partition $\pi$ is an ordinary integer 
partition (with parts in $\mathscr{A}$).
It will be convenient to encode the rows as a multipartition
\[
\nub=\big(\nu^{(1)},\nu^{(2)},\dots,\nu^{(r)}\big),
\]
where the partition 
$\nu^{(i)}=\big(\nu^{(i)}_1,\dots,\nu^{(i)}_{\la_i-\mu_i}\big)$
corresponds to the filling of the $i$th row of $\pi$.
The condition that each of the $\la_1-\mu_r$ columns of $\pi$ also form a 
partition then translates to
\begin{equation}\label{Eq_nu-ineq}
\nu^{(i)}_j\geq \nu^{(i+1)}_{j+\mu_i-\mu_{i+1}}\quad
\text{for $1\leq i\leq r-1$ and $1\leq j\leq \la_{i+1}-\mu_i$}.
\end{equation}
In the above example
\[
\nub=\big((9,7,6,2),(11,11,8,7,4),(10,8,8),(9,8,5,5)\big).
\]
Mostly one is interested in enumerating plane partitions with fixed shape,
in which case it is natural to consider connected shapes only.

Let $d$ be an integer, which we will call the level, such that
\[
d\geq \max\{\mu_1-\mu_r,\la_1-\la_r\}.
\]
Then a plane partition $\pi$ of shape $\la/\mu$ is said to be a cylindric 
partition of shape $\la/\mu/d$ if it is also a plane partition of shape
\begin{equation}\label{Eq_c-shape}
(d+\la_r,\la_1,\dots,\la_r)/(d+\mu_r,\mu_1,\dots,\mu_r)
\end{equation}
with multipartition given by
\[
\big(\nu^{(r)},\nu^{(1)},\nu^{(2)},\dots,\nu^{(r)}\big).
\]
In other words, on top of \eqref{Eq_nu-ineq} one also has the cyclic 
conditions
\begin{equation}\label{Eq_cyclic-ineq}
\nu^{(r)}_j\geq \nu^{(1)}_{j-\mu_1+\mu_r+d}\quad
\text{for $1\leq j\leq \la_1-\mu_r-d$}.
\end{equation}
If one is interested in cylindric partitions of fixed shape
$\la/\mu/d$ it is natural to further restrict the level to
\[
d<\la_1-\mu_r
\]
to avoid \eqref{Eq_cyclic-ineq} from trivialising.

The cylindric condition should really be viewed as the wrapping of the plane 
partition around a semi-infinite cylinder. 
Our earlier example of a plane partition is a non-trivial cylindric
partition of shape $(7,6,4,4)/(3,1,1,0)/d$ on $\mathbb{N}$ for
$d=5,6,7$, and a trivial one for $d\geq 8$.
For example, for $d=5$ either of the two diagrams below represents this
cylindric partition:

\smallskip \bigskip

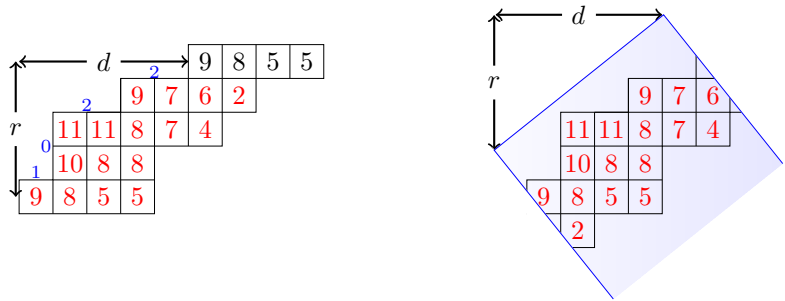
\tikzmath{\x=5;\y=4;\z=15;}
\begin{center}
\begin{tikzpicture}[scale=0.45,line width=0.3pt]

\shadedraw[left color=white!98!blue,right color=white!90!blue,draw=none] 
(7.56+\z,1.482)--(4.04+\z,5.882)--(-0.965+\z,1.878)--(2.5551+\z,-2.522)--cycle;

\draw[thick,->] (2.04+\z,5.882)--(4.00+\z,5.882);
\draw[thick,->] (1.04+\z,5.882)--(4.08+\z-5,5.882);
\draw (4.02-2.5+\z,5.882) node {$d$};

\draw[thick,->] (-0.965+\z,1.878+1.5)--(-0.965+\z,1.878);
\draw[thick,->] (-0.965+\z,1.878+2.5)--(-0.965+\z,1.87+4);
\draw (-0.965+\z,1.878+2) node {$r$};

\draw (\x-2.5,\y+\s) node {$d$};
\draw[thick,->] (\x-2,\y+\s)--(\x,\y+\s);
\draw[thick,->] (\x-3,\y+\s)--(0,\y+\s);

\draw (-0.1,\y+\s-2) node {$r$};
\draw[thick,->] (-0.1,\y+\s-1.5)--(-0.1,\y+\s);
\draw[thick,->] (-0.1,\y+\s-2.5)--(-0.1,\y+\s-4);

\draw (\x+2,\y)--(4+\x,\y)--(4+\x,1+\y)--(\x,1+\y)--(\x,\y);
\draw (1+\x,0+\y)--(1+\x,1+\y);
\draw (2+\x,0+\y)--(2+\x,1+\y);
\draw (3+\x,0+\y)--(3+\x,1+\y);

% profile
\draw[blue] (\s+\x-1.5,\s+3.7) node {$\scriptstyle{2}$};
\draw[blue] (\s+\x-3.5,\s+2.7) node {$\scriptstyle{2}$};
\draw[blue] (\s+\x-4.7,\s+1.5) node {$\scriptstyle{0}$};
\draw[blue] (\s+\x-5,\s+0.72) node {$\scriptstyle{1}$};

% filling
\draw (\s+\x,\s+\y) node {$9$};
\draw (1+\s+\x,\s+\y) node {$8$};
\draw (2+\s+\x,\s+\y) node {$5$};
\draw (3+\s+\x,\s+\y) node {$5$};

\draw (0,0)--(4,0)--(4,2)--(6,2)--(6,3)--(7,3)--(7,4)--(3,4)--(3,3)--(1,3)--(1,1)--(0,1)--cycle;
% horizontal lines
\draw (1,1)--(4,1);
\draw (1,2)--(4,2);
\draw (2,3)--(6,3);
% vertical lines
\draw (1,0)--(1,1);
\draw (2,0)--(2,3);
\draw (3,0)--(3,4);
\draw (4,2)--(4,4);
\draw (5,2)--(5,4);
\draw (6,3)--(6,4);
% filling
\begin{scope}[color=red]
\draw (\s,\s) node {$9$};
\draw (1+\s,\s) node {$8$};
\draw (2+\s,\s) node {$5$};
\draw (3+\s,\s) node {$5$};
\draw (1+\s,1+\s) node {$10$};
\draw (2+\s,1+\s) node {$8$};
\draw (3+\s,1+\s) node {$8$};
\draw (1+\s,2+\s) node {$11$};
\draw (2+\s,2+\s) node {$11$};
\draw (3+\s,2+\s) node {$8$};
\draw (4+\s,2+\s) node {$7$};
\draw (5+\s,2+\s) node {$4$};
\draw (3+\s,3+\s) node {$9$};
\draw (4+\s,3+\s) node {$7$};
\draw (5+\s,3+\s) node {$6$};
\draw (6+\s,3+\s) node {$2$};
\end{scope}

\shadedraw[left color=white!98!blue,right color=white!90!blue,draw=none] (7.56+\z,1.482)--(4.04+\z,5.882)--(-0.965+\z,1.878)--(2.5551+\z,-2.522)--cycle;
\draw (1+\z,-1)--(2+\z,-1)--(2+\z,0);
\draw (1+\z,0)--(1+\z,-1);
\draw[red] (\s+\z+1,\s-1) node {$2$};

\draw (\x+\z,\y)--(0.6+\x+\z,\y);
\draw (\x+\z,\y)--(\x+\z,\y+0.9);

\draw (\z,0)--(4+\z,0)--(4+\z,2)--(6+\z,2)--(6+\z,3)--(7+\z,3)--(7+\z,4)--(3+\z,4)--(3+\z,3)--(1+\z,3)--(1+\z,1)--(0+\z,1)--cycle;
% horizontal lines
\draw (1+\z,1)--(4+\z,1);
\draw (1+\z,2)--(4+\z,2);
\draw (2+\z,3)--(6+\z,3);
% vertical lines
\draw (1+\z,0)--(1+\z,1);
\draw (2+\z,0)--(2+\z,3);
\draw (3+\z,0)--(3+\z,4);
\draw (4+\z,2)--(4+\z,4);
\draw (5+\z,2)--(5+\z,4);
\draw (6+\z,3)--(6+\z,4);
% filling
\begin{scope}[color=red]
\draw (\s+\z,\s) node {$9$};
\draw (1+\s+\z,\s) node {$8$};
\draw (2+\s+\z,\s) node {$5$};
\draw (3+\s+\z,\s) node {$5$};
\draw (1+\s+\z,1+\s) node {$10$};
\draw (2+\s+\z,1+\s) node {$8$};
\draw (3+\s+\z,1+\s) node {$8$};
\draw (1+\s+\z,2+\s) node {$11$};
\draw (2+\s+\z,2+\s) node {$11$};
\draw (3+\s+\z,2+\s) node {$8$};
\draw (4+\s+\z,2+\s) node {$7$};
\draw (5+\s+\z,2+\s) node {$4$};
\draw (3+\s+\z,3+\s) node {$9$};
\draw (4+\s+\z,3+\s) node {$7$};
\draw (5+\s+\z,3+\s) node {$6$};
\end{scope}

\filldraw[white] (-0.965+\z,1.878)--(2.5551+\z,-2.522)--(1+\z,-2)--cycle;
\filldraw[white] (4.04+\z,5.882)--(7.56+\z,1.482)--(7+\z,5)--cycle;
\draw[blue] (7.56+\z,1.482)--(4.04+\z,5.882)--(-0.965+\z,1.878)--(2.5551+\z,-2.522);

\end{tikzpicture}
\end{center}

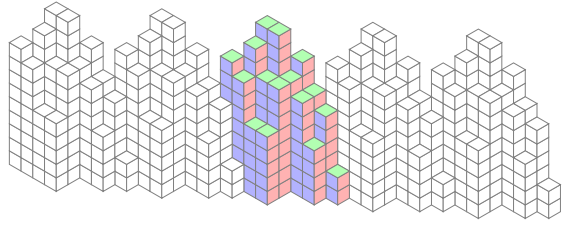
\begin{figure}[t]
\begin{center}
\begin{tikzpicture}[scale=0.18,transform shape]
% Blue tiles
\foreach \h in {5,7,...,19,21}{\pic at ({-19*\sq},\h/2) {bluesquarew};}
\foreach \h in {2,3,...,9}{\pic at ({-18*\sq},\h) {bluesquarew};}
\foreach \h in {3,5,...,11,21,23}{\pic at ({-17*\sq},\h/2) {bluesquarew};}
\foreach \h in {1,2,...,5,7,8,9,13}{\pic at ({-16*\sq},\h) {bluesquarew};}
\foreach \h in {13,15,17,21,23,25}{\pic at ({-15*\sq},\h/2) {bluesquarew};}
\foreach \h in {3,5,...,15,21}{\pic at ({-13*\sq},\h/2) {bluesquarew};}
\foreach \h in {1,2,...,4}{\pic at ({-12*\sq},\h) {bluesquarew};}
\foreach \h in {11,13}{\pic at ({-11*\sq},\h/2) {bluesquarew};}
\foreach \h in {1,2}{\pic at ({-10*\sq},\h) {bluesquarew};}
\foreach \h in {8,9,10}{\pic at ({-10*\sq},\h) {bluesquarew};}
\foreach \h in {7,9,...,17}{\pic at ({-9*\sq},\h/2) {bluesquarew};}
\foreach \h in {1,2,...,5,10,11}{\pic at ({-8*\sq},\h) {bluesquarew};}
\foreach \h in {1,3,...,9,13,15,17,25}{\pic at ({-7*\sq},\h/2) {bluesquarew};}
\foreach \h in {6,7,8,10,11,12}{\pic at ({-6*\sq},\h) {bluesquarew};}
\foreach \h in {1,2,...,7,10}{\pic at ({-4*\sq},\h) {bluesquarew};}
\foreach \h in {1,3,...,7}{\pic at ({-3*\sq},\h/2) {bluesquarew};}
\foreach \h in {5,6}{\pic at ({-2*\sq},\h) {bluesquarew};}
\foreach \h in {1,3}{\pic at ({-1*\sq},\h/2) {bluesquarew};}
\foreach \h in {15,17,19}{\pic at ({-1*\sq},\h/2) {bluesquare};}
\foreach \h in {3,4,...,8}{\pic at ({0*\sq},\h) {bluesquare};}
\foreach \h in {1,3,...,9,19,21}{\pic at ({1*\sq},\h/2) {bluesquare};}
\foreach \h in {0,1,...,4,6,7,8,12}{\pic at ({2*\sq},\h) {bluesquare};}
\foreach \h in {11,13,15,19,21,23}{\pic at ({3*\sq},\h/2) {bluesquare};}
\foreach \h in {1,3,...,13,19}{\pic at ({5*\sq},\h/2) {bluesquare};}
\foreach \h in {0,1,...,3}{\pic at ({6*\sq},\h) {bluesquare};}
\foreach \h in {9,11}{\pic at ({7*\sq},\h/2) {bluesquare};}
\foreach \h in {0,1}{\pic at ({8*\sq},\h) {bluesquare};}
\foreach \h in {7,8,9}{\pic at ({8*\sq},\h) {bluesquarew};}
\foreach \h in {5,7,...,15}{\pic at ({9*\sq},\h/2) {bluesquarew};}
\foreach \h in {0,1,...,4,9,10}{\pic at ({10*\sq},\h) {bluesquarew};}
\foreach \h in {-1,1,...,7,11,13,15,23}{\pic at ({11*\sq},\h/2) {bluesquarew};}
\foreach \h in {5,6,7,9,10,11}{\pic at ({12*\sq},\h) {bluesquarew};}
\foreach \h in {0,1,...,6,9}{\pic at ({14*\sq},\h) {bluesquarew};}
\foreach \h in {-1,1,...,5}{\pic at ({15*\sq},\h/2) {bluesquarew};}
\foreach \h in {4,5}{\pic at ({16*\sq},\h) {bluesquarew};}
\foreach \h in {-1,1}{\pic at ({17*\sq},\h/2) {bluesquarew};}
\foreach \h in {13,15,17}{\pic at ({17*\sq},\h/2) {bluesquarew};}
\foreach \h in {2,3,...,7}{\pic at ({18*\sq},\h) {bluesquarew};}
\foreach \h in {-1,1,...,7,17,19}{\pic at ({19*\sq},\h/2) {bluesquarew};}
\foreach \h in {-1,0,...,3,5,6,7,11}{\pic at ({20*\sq},\h) {bluesquarew};}
\foreach \h in {9,11,13,17,19,21}{\pic at ({21*\sq},\h/2) {bluesquarew};}
\foreach \h in {-1,1,...,11,17}{\pic at ({23*\sq},\h/2) {bluesquarew};}
\foreach \h in {-1,0,...,2}{\pic at ({24*\sq},\h) {bluesquarew};}
\foreach \h in {7,9}{\pic at ({25*\sq},\h/2) {bluesquarew};}
\foreach \h in {-1,0}{\pic at ({26*\sq},\h) {bluesquarew};}
% Red tiles
\pic at ({-18*\sq},10) {redsquarew};
\foreach \h in {13,15,17}{\pic at ({-17*\sq},\h/2) {redsquarew};}
\foreach \h in {10,11}{\pic at ({-16*\sq},\h) {redsquarew};}
\foreach \h in {1,3,...,9}{\pic at ({-15*\sq},\h/2) {redsquarew};}
\foreach \h in {1,2,...,8,10,11,12}{\pic at ({-14*\sq},\h) {redsquarew};}
\pic at ({-13*\sq},17/2) {redsquarew};
\foreach \h in {5,6,7,9,10}{\pic at ({-12*\sq},\h) {redsquarew};}
\foreach \h in {1,3,...,7,15}{\pic at ({-11*\sq},\h/2) {redsquarew};}
\foreach \h in {3,4,...,6}{\pic at ({-10*\sq},\h) {redsquarew};}
\foreach \h in {1,3}{\pic at ({-9*\sq},\h/2) {redsquarew};}
\pic at ({-9*\sq},19/2) {redsquarew};
\foreach \h in {6,7,8}{\pic at ({-8*\sq},\h) {redsquarew};}
\foreach \h in {19,21}{\pic at ({-7*\sq},\h/2) {redsquarew};}
\foreach \h in {0,1,...,4}{\pic at ({-6*\sq},\h) {redsquarew};}
\foreach \h in {1,3,...,15,19,21,23}{\pic at ({-5*\sq},\h/2) {redsquarew};}
\pic at ({-4*\sq},8) {redsquarew};
\foreach \h in {9,11,13,17,19}{\pic at ({-3*\sq},\h/2) {redsquarew};}
\foreach \h in {0,1,...,3,7}{\pic at ({-2*\sq},\h) {redsquarew};}
\foreach \h in {5,7,...,11}{\pic at ({-1*\sq},\h/2) {redsquarew};}
\foreach \h in {0,1}{\pic at ({0*\sq},\h) {redsquarew};}
\pic at ({0*\sq},9) {redsquare};
\foreach \h in {11,13,15}{\pic at ({1*\sq},\h/2) {redsquare};}
\foreach \h in {9,10}{\pic at ({2*\sq},\h) {redsquare};}
\foreach \h in {-1,1,...,7}{\pic at ({3*\sq},\h/2) {redsquare};}
\foreach \h in {0,1,...,7,9,10,11}{\pic at ({4*\sq},\h) {redsquare};}
\pic at ({5*\sq},15/2) {redsquare};
\foreach \h in {4,5,6,8,9}{\pic at ({6*\sq},\h) {redsquare};}
\foreach \h in {-1,1,...,5,13}{\pic at ({7*\sq},\h/2) {redsquare};}
\foreach \h in {2,3,...,5}{\pic at ({8*\sq},\h) {redsquare};}
\foreach \h in {-1,1}{\pic at ({9*\sq},\h/2) {redsquare};}
\pic at ({9*\sq},17/2) {redsquarew};
\foreach \h in {5,6,7}{\pic at ({10*\sq},\h) {redsquarew};}
\foreach \h in {17,19}{\pic at ({11*\sq},\h/2) {redsquarew};}
\foreach \h in {-1,0,...,3}{\pic at ({12*\sq},\h) {redsquarew};}
\foreach \h in {-1,1,...,13,17,19,21}{\pic at ({13*\sq},\h/2) {redsquarew};}
\pic at ({14*\sq},7) {redsquarew};
\foreach \h in {7,9,11,15,17}{\pic at ({15*\sq},\h/2) {redsquarew};}
\foreach \h in {-1,0,...,2,6}{\pic at ({16*\sq},\h) {redsquarew};}
\foreach \h in {3,5,...,9}{\pic at ({17*\sq},\h/2) {redsquarew};}
\foreach \h in {-1,0}{\pic at ({18*\sq},\h) {redsquarew};}
\pic at ({18*\sq},8) {redsquarew};
\foreach \h in {9,11,13}{\pic at ({19*\sq},\h/2) {redsquarew};}
\foreach \h in {8,9}{\pic at ({20*\sq},\h) {redsquarew};}
\foreach \h in {-3,-1,...,5}{\pic at ({21*\sq},\h/2) {redsquarew};}
\foreach \h in {-1,0,...,6,8,9,10}{\pic at ({22*\sq},\h) {redsquarew};}
\pic at ({23*\sq},13/2) {redsquarew};
\foreach \h in {3,4,5,7,8}{\pic at ({24*\sq},\h) {redsquarew};}
\foreach \h in {-3,-1,...,3,11}{\pic at ({25*\sq},\h/2) {redsquarew};}
\foreach \h in {1,2,...,4}{\pic at ({26*\sq},\h) {redsquarew};}
\foreach \h in {-3,-1}{\pic at ({27*\sq},\h/2) {redsquarew};}
% Green tiles
\pic at ({-19*\sq},23/2) {greensquarew};
\pic at ({-18*\sq},10) {greensquarew};
\foreach \h in {13,25}{\pic at ({-17*\sq},\h/2) {greensquarew};}
\foreach \h in {6,10,14}{\pic at ({-16*\sq},\h) {greensquarew};}
\foreach \h in {19,27}{\pic at ({-15*\sq},\h/2) {greensquarew};}
\pic at ({-14*\sq},10) {greensquarew};
\foreach \h in {17,23}{\pic at ({-13*\sq},\h/2) {greensquarew};}
\foreach \h in {5,9}{\pic at ({-12*\sq},\h) {greensquarew};}
\pic at ({-11*\sq},15/2) {greensquarew};
\pic at ({-10*\sq},3) {greensquarew};
\pic at ({-10*\sq},11) {greensquarew};
\pic at ({-9*\sq},19/2) {greensquarew};
\foreach \h in {6,12}{\pic at ({-8*\sq},\h) {greensquarew};}
\foreach \h in {11,19,27}{\pic at ({-7*\sq},\h/2) {greensquarew};}
\foreach \h in {9,13}{\pic at ({-6*\sq},\h) {greensquarew};}
\pic at ({-5*\sq},19/2) {greensquarew};
\foreach \h in {8,11}{\pic at ({-4*\sq},\h) {greensquarew};}
\foreach \h in {9,17}{\pic at ({-3*\sq},\h/2) {greensquarew};}
\pic at ({-2*\sq},7) {greensquarew};
\pic at ({-1*\sq},5/2) {greensquarew};
\pic at ({-1*\sq},21/2) {greensquare};
\pic at ({0*\sq},9) {greensquare};
\foreach \h in {11,23}{\pic at ({1*\sq},\h/2) {greensquare};}
\foreach \h in {5,9,13}{\pic at ({2*\sq},\h) {greensquare};}
\foreach \h in {17,25}{\pic at ({3*\sq},\h/2) {greensquare};}
\pic at ({4*\sq},9) {greensquare};
\foreach \h in {15,21}{\pic at ({5*\sq},\h/2) {greensquare};}
\foreach \h in {4,8}{\pic at ({6*\sq},\h) {greensquare};}
\pic at ({7*\sq},13/2) {greensquare};
\pic at ({8*\sq},2) {greensquare};
\pic at ({8*\sq},10) {greensquarew};
\pic at ({9*\sq},17/2) {greensquarew};
\foreach \h in {5,11}{\pic at ({10*\sq},\h) {greensquarew};}
\foreach \h in {9,17,25}{\pic at ({11*\sq},\h/2) {greensquarew};}
\foreach \h in {8,12}{\pic at ({12*\sq},\h) {greensquarew};}
\pic at ({13*\sq},17/2) {greensquarew};
\foreach \h in {7,10}{\pic at ({14*\sq},\h) {greensquarew};}
\foreach \h in {7,15}{\pic at ({15*\sq},\h/2) {greensquarew};}
\pic at ({16*\sq},6) {greensquarew};
\pic at ({17*\sq},3/2) {greensquarew};
\pic at ({17*\sq},19/2) {greensquarew};
\pic at ({18*\sq},8) {greensquarew};
\foreach \h in {9,21}{\pic at ({19*\sq},\h/2) {greensquarew};}
\foreach \h in {4,8,12}{\pic at ({20*\sq},\h) {greensquarew};}
\foreach \h in {15,23}{\pic at ({21*\sq},\h/2) {greensquarew};}
\pic at ({22*\sq},8) {greensquarew};
\foreach \h in {13,19}{\pic at ({23*\sq},\h/2) {greensquarew};}
\foreach \h in {3,7}{\pic at ({24*\sq},\h) {greensquarew};}
\pic at ({25*\sq},11/2) {greensquarew};
\pic at ({26*\sq},1) {greensquarew};
\end{tikzpicture}
\end{center}
\caption{Geometric representation of a cylindric partition, showing
five of the infinitely many copies of a fundamental domain.
The drift is due to the fact that $d\neq r$.}
\label{figure}
\end{figure}

\smallskip

\noindent
where the blue labels on the left correspond to the profile of the 
cylindric partition as explained below.
For this same cylindric partition represented in terms of stacked 
unit cubes, see Figure~\ref{figure}.

Let $\pi$ be cylindric partition of shape $\la/\mu/d$ and rank $r$.
Then the profile $c=(c_0,\dots,c_{r-1})$ of $\pi$ is a sequence 
of $r$ nonnegative integers that sums to $d$, given by
\[
c_0=d-\mu_1+\mu_r \quad \text{and} \quad c_i=\mu_i-\mu_{i+1} \quad
\text{for $1\leq i\leq r-1$}.
\]
The profile of the cylindric partition in our example is $(2,2,0,1)$.

Given a cylindric partition $\pi$ of rank $r$, 
its size $\abs{\pi}$ (Gessel and Krattenthaler use the term norm)
is once again defined as the sum of its entries, i.e.,
\[
\abs{\pi}=\abs{\nub}:=\sum_{i=1}^r \abs{\nu^{(i)}},
\]
where $\nub$ is the multipartition corresponding to $\pi$.
Note that this corresponds to the volume of a fundamental 
domain as shown in Figure~\ref{figure}.
Also,
\[
\max(\pi):=\max\big\{\nu^{(i)}_j: 1\leq i\leq r,
1\leq j\leq \la_i-\mu_i\big\}
\]
denotes the value of the maximal entry (or entries) of $\pi$
(or the height of $\pi$ in its geometric representation).
The above example of a cylindric partition of shape
$(7,6,4,4)/(3,1,1,0)/5$ has size $118$, rank $4$, level $5$ and maximal 
entry $11$.

Let $\mathscr{C}_{\la/\mu/d}(\mathscr{A})$ denote the set of cylindric
partitions of shape $\la/\mu/d$ on $\mathscr{A}$. 
We are interested in the generating functions
\begin{subequations}
\begin{align}
\GK_{\la/\mu/d}(z,q;\mathscr{A})
&:=\sum_{\pi\in \mathscr{C}_{\la/\mu/d}(\mathscr{A})}
z^{\max(\pi)} q^{\abs{\pi}}
\intertext{and}
\GK_{\la/\mu/d;n}(q;\mathscr{A})
&:=\sum_{\substack{\pi\in \mathscr{C}_{\la/\mu/d}(\mathscr{A}) \\[1pt]
\max(\pi)\leq n}} q^{\abs{\pi}}
=\sum_{m=0}^n [z^m] \GK_{\la/\mu/d}(z,q;\mathscr{A}),
\label{Eq_GF-max}
\end{align}
where $[z^m] f(z)$ stands for the coefficient of $z^m$ in the
polynomial or formal power series $f(z)$.
\end{subequations}

For later convenience we also set
\begin{equation}\label{Eq_zero}
\GK_{\la/\mu/d}(z,q;\mathscr{A})=0
\end{equation}
if $\mu\not\subseteq\la$ or if $\la$ and $\mu$ are not both partitions.

The two-variable generating function for cylindric partitions has
translation symmetry
\begin{equation}\label{Eq_trans}
\GK_{(\la_1,\dots,\la_r)/(\mu_1,\dots,\mu_r)/d}(z,q;\mathscr{A})=
\GK_{(\la_1-k,\dots,\la_r-k)/
(\mu_1-k,\dots,\mu_r-k)/d}(z,q;\mathscr{A})
\end{equation}
for $0\leq k\leq\mu_r$, cyclic symmetry 
\begin{equation}
\label{Eq_cyc}
\GK_{(\la_1,\dots,\la_r)/(\mu_1,\dots,\mu_r)/d}(z,q;\mathscr{A})
=\GK_{(\la_r+d,\la_1,\dots,\la_{r-1})/
(\mu_r+d,\mu_1,\dots,\mu_{r-1})/d}(z,q;\mathscr{A}),
\end{equation}
and conjugation symmetry \cite[page 462]{GK97}
\begin{equation}\label{Eq_levelrank}
\GK_{(\la_1+dL,\dots,\la_r+dL)/\mu/d}(z,q;\mathscr{A})=
\GK_{(\la_1'+rL,\dots,\la_d'+rL)/\mu'/r}(z,q;\mathscr{A})
\end{equation}
for $\la,\mu\subseteq (d^r)$, $\mu',\la'\subseteq (r^d)$ 
and $L$ a nonnegative integer.
Following \cite{FW16}, we will refer to \eqref{Eq_levelrank}
as level-rank duality for cylindric partitions, in analogy with
level-rank duality in representation theory.

By a special case of \cite[Theorem 2]{GK97} due to Gessel and Krattenthaler,
\begin{multline}\label{Eq_GK-thm3}
\GK_{\la/\mu/d;n}(q;\mathbb{N}_0) =
\sum_{\substack{k_1,\dots,k_r\in\mathbb{Z} \\[1pt] k_1+\dots+k_r=0}}
\det_{1\leq i,j\leq r}\bigg(
q^{r(d+r)\binom{k_i}{2}+(d+r)ik_i+(\mu_j-j)(rk_i+i-j)} \\[-2mm]
\times \qbin{n+\la_i-\mu_j-dk_i}{\la_i-\mu_j-(d+r)k_i+j-i}\bigg).
\end{multline}
Taking the large-$n$ limit and then specialising
$\la_1=\dots=\la_r=L\geq \mu_1$, this simplifies to \cite[Theorem 5]{GK97}
\begin{align}\label{Eq_Lr}
&\GK_{(L^r)/\mu/d}(1,q;\mathbb{N}_0) \\ 
&\qquad=\sum_{\substack{k_1,\dots,k_r\in\mathbb{Z} \\[1pt] k_1+\dots+k_r=0}}
\prod_{i=1}^r \frac{q^{r(d+r)\binom{k_i}{2}+(di+r\mu_i)k_i}}
{(q)_{L-\mu_i-(d+r)k_i+i-1}}
\prod_{1\leq i<j\leq r} \Big(1-q^{(d+r)(k_i-k_j)+\mu_i-\mu_j+j-i}\Big).
\notag
\end{align}

Let $\nu$ and $\la$ be partitions of length $r$ such that $\nu\subseteq\la$.
Any $\bar{\pi}\in\mathscr{C}_{\nu/\mu/d}(\mathbb{N})$ maps to a unique
$\pi\in\mathscr{C}_{\la/\mu/d}(\mathbb{N}_0)$ such that 
$\abs{\pi}=\abs{\bar{\pi}}$ and $\max(\pi)=\max(\bar{\pi})$,
by simply filling those squares of $\pi$ not contained in $\bar{\pi}$ by
$0$ and leaving the other squares unchanged.
Conversely, by deleting all squares with filling $0$, any
$\pi\in \mathscr{C}_{\la/\mu/d}(\mathbb{N}_0)$ maps to a
unique $\bar{\pi}\in\mathscr{C}_{\nu/\mu/d}(\mathbb{N})$ for
some fixed $\nu\subseteq\la$.
Hence, if for a fixed profile $c\in\mathbb{N}_0^r$ such that
$c_0+\dots+c_{r-1}=d$, we set
\[
\mu(c):=(c_1+\dots+c_{r-1},c_2+\dots+c_{r-1},\dots,c_{r-1},0),
\]
then the generating function $\GK_c(z,q)$ of all cylindric partitions 
on $\mathbb{N}$ of rank $r$, level $d$ and profile $c$ is given by
\begin{equation}\label{Eq_Gc}
\GK_c(z,q):=\sum_{\substack{\la\supseteq\mu(c)\\[1pt] l(\la)=r}} 
\GK_{\la/\mu(c)/d}(z,q;\mathbb{N})=
\lim_{L\to\infty} \GK_{(L^r)/\mu(c)/d}(z,q;\mathbb{N}_0).
\end{equation}
From the cyclic symmetry \eqref{Eq_cyc},
\begin{equation}\label{Eq_cyc-profiles}
\GK_{(c_0,c_1,\dots,c_{r-1})}(z,q)=
\GK_{(c_{r-1},c_0,c_1,\dots,c_{r-2})}(z,q).
\end{equation}
Hence the number of inequivalent profiles $c=(c_0,\dots,c_{r-1})$ 
such that $c_0+\dots+c_{r-1}=d$ is given by
\begin{equation}\label{Eq_number-profiles}
\frac{1}{r}\, \big[x^d\big] \, \sum_{k\mid r} \phi(k) 
\bigg(\frac{1}{1-x^k}\bigg)^{r/k},
\end{equation}
where $\phi$ is Euler's totient function.

When $z=1$ it follows from \eqref{Eq_Lr} and \eqref{Eq_Gc} that
\[
\GK_c(1,q)=\frac{1}{(q)_{\infty}^r}
\sum_{\substack{k_1,\dots,k_r\in\mathbb{Z} \\[1pt] k_1+\dots+k_r=0}}
\prod_{i=1}^r q^{r(d+r)\binom{k_i}{2}+(di+r\mu_i)k_i}
\prod_{1\leq i<j\leq r} \Big(1-q^{(d+r)(k_i-k_j)+\mu_i-\mu_j+j-i}\Big),
\]
where $\mu=\mu(c)$ and $d=c_0+\dots+c_{r-1}$.
By the $\mathrm{A}_{r-1}^{(1)}$ Macdonald identity \cite{Macdonald72}
\begin{equation}\label{Eq_Macdonald}
\sum_{\substack{k_1,\dots,k_r\in\mathbb{Z} \\[1pt] k_1+\dots+k_r=0}}
\prod_{i=1}^r x_i^{rk_i} q^{r\binom{k_i}{2}+ik_i}
\prod_{1\leq i<j\leq r} \big(1-q^{k_i-k_j}x_i/x_j\big) =
(q)_{\infty}^{r-1} \prod_{1\leq i<j\leq r} \theta(x_i/x_j;q)
\end{equation}
with $(q,x_i)\mapsto (q^{d+r},q^{\mu_i-i})$,
this yields Borodin's product formula \cite{Borodin07} 
\begin{equation}\label{Eq_Bor}
\GK_c(1,q)
=\frac{(q^{d+r};q^{d+r})_{\infty}^{r-1}}{(q)_{\infty}^r}
\prod_{1\leq i<j\leq r} \theta(q^{\mu_i-\mu_j+j-i};q^{d+r}).
\end{equation}
(The reader is referred to 
\cite{CSV11,Koshida20,Krattenthaler08,Langer12a,Langer12b,Tingley08} for 
alternative derivations and/or generalisations.)
Comparing this with the definition of $\AG_{\la;r}(q)$ given in the
introduction, we thus have
\begin{equation}\label{Eq_GK-AG}
\GK_{(c_0,\dots,c_{r-1})}(1,q)=\frac{1}{(q)_{\infty}}\,
\AG_{c_0\La_0+\dots+c_{r-1}\La_{r-1};r}(q),
\end{equation}
which is \eqref{Eq_AG-GK}.

Due to the additional reversal symmetry
\begin{equation}\label{Eq_reversal}
\GK_{(c_0,c_1,\dots,c_{r-1})}(1,q)=\GK_{(c_{r-1},\dots,c_1,c_0)}(1,q),
\end{equation}
\eqref{Eq_Bor} yields a smaller number of distinct infinite products than 
predicted by \eqref{Eq_number-profiles}.
For rank $3$ and $d=3k+i-3$ with $i=0,1,2$ and $k\geq 1$ (such that $d\geq 1$),
the cyclic symmetry plus \eqref{Eq_reversal} leads to a total of
\begin{equation}\label{Eq_rank3-count}
\binom{k+1}{2}+\Big\lfloor\frac{1}{4}(k+i-1)^2\Big\rfloor
\end{equation} 
distinct infinite products for the modulus $3k+i$.

For later reference we also note that the level-rank duality 
\eqref{Eq_levelrank} implies that
\begin{equation}\label{Eq_levelrank-c}
\GK_c(z,q)=\GK_{c'}(z,q),
\end{equation}
where $c=(c_0,\dots,c_{r-1})$ such that $c_0+\dots+c_{r-1}=d$ and 
$c'=(c'_0,\dots,c'_{d-1})$
such that $c'_1+\dots+c'_{d-1}=r$ are related by $\mu(c)=(\mu(c'))'$,
where $l(\mu(c))=r$ and $l(\mu(c'))=d$.
For example,
\begin{subequations}\label{Eq_RR-rank3}
\begin{equation}
\GK_{(1,1,0)}(z,q)=\GK_{(2,1)}(z,q)
\end{equation}
since $\mu(1,1,0)=(1,0,0)$ and $\mu(2,1)=(1,0)$, and
\begin{equation}
\GK_{(2,0,0)}(z,q)=\GK_{(3,0)}(z,q)
\end{equation}
\end{subequations}
since $\mu(2,0,0)=(0,0,0)$ and $\mu(3,0)=(0,0)$.
In particular, the expression \eqref{Eq_number-profiles} is symmetric in 
$d$ and $r$.

\section{Connections to representation theory}\label{Sec_Characters}

Let $\mathfrak{h}$ and $\mathfrak{h}^{\ast}$ be the Cartan subalgebra
and its dual of the affine Lie algebra $\mathrm{A}_{r-1}^{(1)}$, with
pairing $\langle \cdot,\cdot\rangle$, see \cite{Kac90,Wakimoto01}.
As usual we use the non-degenerate bilinear form $\bil{\cdot}{\cdot}$ on 
$\mathfrak{h}$ to identify $\mathfrak{h}$ and $\mathfrak{h}^{\ast}$.
For $I:=\{0,1,\dots,r-1\}$,
let $\{\alpha_i: i\in I\}$ be the set of simple roots, 
$\{\La_i: i\in I\}$ the set of fundamental weights
and $\delta=\sum_{i\in I}\alpha_i$ the null root.
Then $\mathfrak{h}^{\ast}=
\mathrm{Span}_{\mathbb{C}}\{\Lambda_0,\alpha_0,\dots,\alpha_{r-1}\}$.
By the above identification, and since we are considering 
$\mathrm{A}_{r-1}^{(1)}$ only, we do not need to distinguish between
roots and coroots, i.e., 
$\langle\alpha_i,\La_j\rangle=\bil{\alpha_i}{\La_j}=\delta_{ij}$ for all
$i,j\in I$.
We further denote the finite part of $\mathfrak{h}^{\ast}$
by $\bar{\mathfrak{h}}^{\ast}$, i.e.,
$\bar{\mathfrak{h}}^{\ast}=\mathrm{Span}_{\mathbb{C}}
\{\alpha_i: i\in\bar{I}\}$, where $\bar{I}:=\{1,2,\dots,r-1\}$, 
The level of $\la\in\mathfrak{h}^{\ast}$ is defined as 
$\lev(\la):=\langle \la,\delta\rangle$.
In particular, $\lev(\La_i)=1$ for all $i\in I$.
Let $P_{+}$ and $P_{+}^m$ denote the set of dominant integral weights,
and level-$m$ dominant integral weights, respectively:
\begin{align*}
P_{+}&:=\big\{ \la\in\mathfrak{h}^{\ast}: 
\langle\la,\alpha_i^{\vee}\rangle\in\mathbb{N}_0 \text{ for $i\in I$}\big\}
=\sum_{i\in I}\mathbb{N}_0\La_i+\mathbb{C}\delta, \\
P_{+}^m&:=\big\{\la\in P_{+}: \lev(\la)=m\big\}, 
\end{align*}
where $m$ is a nonnegative integer.
Since the $\delta$-part of $P_{+}^m$ does not play a role in what follows,
we will typically ignore it and parametrise weights in $P_{+}^m$ as
\begin{equation}\label{Eq_la-para}
\la=(m-\mu_1+\mu_r)\La_0+(\mu_1-\mu_2)\La_i+\dots+(\mu_{r-1}-\mu_r)\La_{r-1},
\end{equation}
where $\mu=(\mu_1,\dots,\mu_r)$ is a partition such that $\mu_1-\mu_r\leq m$.
This parametrisation only depends on the differences $\mu_i-\mu_r$ for 
$i\in\bar{I}$, and without loss of generality we may assume that $\mu_r=0$.
The Weyl vector $\rho$ is defined by $\langle \rho,\alpha_i\rangle=1$
for all $i\in I$. 
Although this fixes $\rho$ modulo $\mathbb{C}\delta$, we once again
ignore the $\delta$-part and simply take $\rho=\sum_{i\in I}\La_i$.
Note that $\lev(\rho)=r$. 
Similarly, we fix the fundamental weights $\La_i$ for $i\in\bar{I}$ as
$\La_i=\La_0+\sum_{j\in\bar{I}}(\min\{i,j\}-ij/r)\alpha_j$.

For $\la\in P_{+}$, let $L(\la)$ be the standard module of 
$\mathrm{A}_{r-1}^{(1)}$ of highest weight $\la$ with $V_{\mu}$ the
weight-space indexed by $\mu$ in the weight-space decomposition of $L(\la)$.
Then the character of $L(\la)$ is defined as
\[
\ch L(\la):=\sum_{\mu\in\mathfrak{h}^{\ast}} \dim(V_{\mu}) \eup^{\mu}.
\]
Since $\dim(V_{\mu})=0$ if $\la-\mu\not\in\sum_{i\in I}\mathbb{N}_0\alpha_i$,
\[
\eup^{-\la} \ch L(\la)\in 
\mathbb{Z}[[\eup^{-\alpha_0},\dots,\eup^{-\alpha_{r-1}}]].
\]
Let $W=\overline{W}\ltimes \overline{Q}$ be the Weyl group of 
$\mathrm{A}_{r-1}^{(1)}$, with $\overline{W}\cong S_r$ the classical
part of $W$ and
$\overline{Q}:=\sum_{i\in\bar{I}}\mathbb{Z}\alpha_i$.
Then, by the Weyl--Kac character formula \cite{Kac74,Kac90},
\begin{equation}\label{Eq_WK}
\ch L(\la)=
\frac{\sum_{w\in W} \sgn(w) \eup^{w(\la+\rho)-\rho}}
{\prod_{\alpha>0} (1-\eup^{-\alpha})^{\mult(\alpha)}},
\end{equation}
where the product in the denominator is over the positive roots in the
root system of $\mathrm{A}_{r-1}^{(1)}$ and $\mult(\alpha)$ is the
multiplicity of the root $\alpha$.
If $\la\in P_{+}^1$, we have the alternative simpler expression
(see e.g., \cite{Kac78,KP84} or \cite[Eq.~(12.13.6)]{Kac90}),
\begin{equation}\label{Eq_level1}
\eup^{-\La_{\ell}} \ch L(\La_{\ell})
=\frac{1}{\prod_{n=1}^{\infty}(1-\eup^{-n\delta})^{r-1}}
\sum_{\alpha\in\overline{Q}} 
\eup^{\alpha-\frac{1}{2}\|\alpha\|^2\delta-\bil{\alpha}{\La_{\ell}}\delta},
\end{equation}
where $\overline{Q}:=\sum_{i\in \bar{I}}\mathbb{Z}\alpha_i$.

Define the graded or $q$-dimension of $L(\la)$ as
\[
\dim_q L(\la):=F_{\mathbbm{1}}\big(\eup^{-\la} \ch L(\la)\big),
\]
where $F_{\mathbbm{1}}$ is the principal specialisation \cite{Lepowsky79}:
\begin{gather*}
F_{\mathbbm{1}}:~\mathbb{C}[[\eup^{-\alpha_0},\dots,\eup^{-\alpha_{r-1}}]]
\to \mathbb{C}[[q]] \\
\eup^{-\alpha_i}\mapsto q\quad\text{for all $i\in I$}.
\end{gather*}
Then \cite{Kac78,Lepowsky82},
\begin{align*}
\dim_q L(\la)&=\prod_{\alpha>0} 
\bigg(\frac{1-q^{\ip{\la+\rho}{\alpha}}}
{1-q^{\ip{\rho}{\alpha}}}\bigg)^{\mult(\alpha)} \\[2mm]
&=\frac{(q^{d+r};q^{d+r})_{\infty}^{r-1}(q^r;q^r)_{\infty}}
{(q)_{\infty}^r} \prod_{1\leq i<j\leq r} 
\theta\big(q^{\mu_i-\mu_j+j-i};q^{d+r}\big),
\end{align*}
where in the second expression on the right it is assumed that
$\la\in P_{+}^d$ with $\la$ parametrised as in 
\eqref{Eq_la-para} with $m$ replaced by $d$.
For $\la=\La_{\ell}$ the above simplifies to 
$(q^r;q^r)_{\infty}/(q)_{\infty}$ for all $\ell\in I$.
By \eqref{Eq_Bor} and \eqref{Eq_GK-AG} we thus have
\begin{align*}
\GK_{(c_0,\dots,c_{r-1})}(1,q)&=
\frac{1}{(q)_{\infty}}\,\AG_{c_0\La_0+\dots+c_{r-1}\La_{r-1};r}(q) \\[1mm]
&=\frac{1}{(q)_{\infty}}\cdot\frac{\dim_q L(c_0\La_0+\cdots+c_{r-1}\La_{r-1})}
{\dim_q L(\La_{\ell})},
\end{align*}
see also \cite{FW16,Tingley08}.

Alternatively, we can identify $\GK_{(c_0,\dots,c_{r-1})}(1,q)$ as an
$\mathrm{A}_{r-1}^{(1)}$ branching function.
To this end we require the characters of certain admissible
representations of $\mathrm{A}_{r-1}^{(1)}$~\cite{KW88,KW89}.
Let $d,m$ be a pair of relatively prime integers such that $d\geq 1$ 
and $m+(d-1)r\geq 0$.
For $\mu=(\mu_1,\dots,\mu_r)$ a partition such that
$\mu_1-\mu_r\leq m+(d-1)r$, let
\begin{equation}\label{Eq_Lam}
\la(\mu;m/d):=(m/d-\mu_1+\mu_r)\La_0+(\mu_1-\mu_2)\La_1+\dots+
(\mu_{r-1}-\mu_r)\La_{r-1},
\end{equation}
(so that $\lev(\la(\mu;m/d))=m/d$), and let
$P^{m/d}$ denote the set of all such weights.
For $d=1$ this is just $P_{+}^m$ but for $d\geq 2$ the weights in 
$P^{m/d}$ are not integral.
The set $P^{m/d}$ is a subset of the set of (principal) admissible
weights of $\mathrm{A}_{r-1}^{(1)}$ defined by Kac and Wakimoto
in \cite{KW88,KW89,KW90}.
As follows from their work (see e.g., \cite[Proposition 3]{KW88} or 
\cite[Theorem 3.1]{KW90}),
the generating function for cylindric partitions of rank $r$ and level
$d$ (such that $d$ and $r$ are relatively prime)
arises as a branching function $b^{\la\otimes\la'}_{\la''}(q)$ corresponding
to the decomposition of $\ch L(\la) \ch L(\la')$ in terms of $\ch L(\la'')$,
where $\la\in P^{r/d-r}$, $\la'\in P_{+}^1$ and $\la''\in P^{r/d-r+1}$.
For the benefit of those readers not familiar with the theory of admissible
representations of affine Kac--Moody algebras, we will translate 
the details of \cite[Proposition 3]{KW88} pertaining to the case of 
cylindric partitions and $\mathrm{A}_{r-1}^{(1)}$ Andrews--Gordon $q$-series
into the language of formal power series.
First of all we note that by setting $\eup^{-\alpha_i}=x_i/x_{i+1}$
for $i\in\bar{I}$ and $\eup^{-\delta}=q$, and by using the semi-direct
product structure of $W$, we can write \eqref{Eq_WK} as
\begin{align}\label{Eq_WK2}
\eup^{-\la} \ch L(\la)&=\frac{1}{(q)_{\infty}^{r-1}
\prod_{1\leq i<j\leq r} (-x_j) \theta(x_i/x_j;q)} \\
&\qquad\times
\sum_{\substack{k_1,\dots,k_r\in\mathbb{Z}\\[1pt] k_1+\cdots+k_r=0}}
\prod_{i=1}^r x_i^{(m+r)k_i-\mu_i} q^{(m+r)\binom{k_i}{2}}
\det_{1\leq i,j\leq r}\Big( \big(x_iq^{k_i})^{\mu_j+r-j}\Big), \notag
\end{align}
where $\la$ is parametrised as in \eqref{Eq_la-para}.
Similarly, \eqref{Eq_level1} becomes
\begin{equation}\label{Eq_WKlevel1}
\eup^{-\La_{\ell}} \ch L(\La_{\ell})
=\frac{1}{(q)_{\infty}^{r-1}}
\sum_{\substack{k_1+\dots+k_r\in\mathbb{Z} \\[1pt] k_1+\dots+k_r=0}}
\prod_{i=1}^r x_i^{k_i} q^{\binom{k_i}{2}+\chi(i\leq\ell) k_i}.
\end{equation}
For the characters of the admissible representations indexed by
$\la=\La(\mu;m/d)\in P^{m/d}$, the Weyl--Kac formula \eqref{Eq_WK2}
generalises to
\begin{align*}
\eup^{-\la} \ch L(\la)&=\frac{1}{(q)_{\infty}^{r-1}
\prod_{1\leq i<j\leq r} (-x_j) \theta(x_i/x_j;q)} \\
&\qquad\times
\sum_{\substack{k_1,\dots,k_r\in\mathbb{Z}\\[1pt] k_1+\cdots+k_r=0}}
\prod_{i=1}^r x_i^{(m+dr)k_i-\mu_i} q^{d(m+dr)\binom{k_i}{2}}
\det_{1\leq i,j\leq r}\Big( \big(x_iq^{dk_i})^{\mu_j+r-j}\Big).
\end{align*}
According to \cite[Proposition 3]{KW88} we then have the following 
branching rule.

\begin{proposition}[$\mathrm{A}_{r-1}^{(1)}$ branching formula]
\label{Prop_branching}
Let $d$ be a positive integer relatively prime to $r$, and let $\ell\in I$.
Then
\begin{equation}\label{Eq_branching}
\ch L\big((r/d-r)\La_0\big) \cdot \ch L(\La_{\ell})= 
\sum_{\substack{ \la\in P^{r/d-r+1}\\[1pt] 
\la-(r/d-r)\La_0-\La_{\ell}\in\overline{Q}}}
q^{\frac{1}{2}\|\la\|^2-\frac{1}{2}\|\La_{\ell}\|^2}
b_{\la}(q)\, \ch L(\la),
\end{equation}
where, assuming that $\la$ is parametrised as in \eqref{Eq_Lam}
with $m=d+r-dr$,
\[
b_{\la}(q)=
\frac{(q^{d+r};q^{d+r})_{\infty}^{r-1}}{(q)_{\infty}^{r-1}}
\prod_{1\leq i<j\leq r} \theta(q^{\mu_i-\mu_j+j-i};q^{d+r}).
\]
\end{proposition}

\begin{remark}
A number of comments are in order.
(i) The weight $(r/d-r)\La_0$ is the unique element in $P^{(1/d-1)r}$,
and by the Macdonald identity \eqref{Eq_Macdonald},
\[
\eup^{-(r/d-r)\La_0} \ch L\big((r/d-r)\La_0\big)=
\frac{(q^d;q^d)_{\infty}^{r-1}}{(q)_{\infty}^{r-1}}
\prod_{1\leq i<j\leq r} 
\frac{\theta(x_i/x_j;q^d)}{\theta(x_i/x_j;q)}.
\]
(ii) Up to an overall power of $q$, the branching function
$b_{\la}(q)$ is $b^{(r/d-r)\La_0\otimes\La_0}_{\la}(q)$ of
\cite{KW88}.
In \cite[Proposition 3]{KW88} $b_{\la}(q)$ is not given in product form
since it is stated there as special case of a more general branching 
rule for which the branching functions typically do not admit product forms.
The product form given here once again follows from the Macdonald 
identity.
(iii) In terms of the partition $\mu$ parametrising the elements of
$P^{r/d-r+1}$, the condition
$\la-(r/d-r)\La_0-\La_{\ell}\in\overline{Q}$ corresponds to
$\abs{\mu}\equiv\ell\pmod{r}$. 
Moreover, 
\[
\|\la\|^2-\|\La_{\ell}\|^2=
\sum_{i=1}^r \mu_i^2-\frac{\abs{\mu}^2}{r}-
\frac{\ell(r-\ell)}{r}.
\]
It thus follows that the identity \eqref{Eq_branching} boils down to
\begin{align*}
&\prod_{1\leq i<j\leq r} (-x_j)\theta(x_i/x_j;q^d) 
\sum_{\substack{k_1+\dots+k_r\in\mathbb{Z} \\[1pt] k_1+\dots+k_r=0}}
\prod_{i=1}^r x_i^{k_i} q^{\binom{k_i}{2}+\chi(i\leq\ell) k_i} \\
&\quad=
\frac{(q^{d+r};q^{d+r})_{\infty}^{r-1}}{(q^d;q^d)_{\infty}^{r-1}}
\sum_{\mu} \Bigg( 
q^{\frac{1}{2}\sum_{i=1}^r (\mu_i^2-\ell)-\frac{\abs{\mu}^2-\ell^2}{2r}}
\prod_{1\leq i<j\leq r} \theta(q^{\mu_i-\mu_j+j-i};q^{d+r}) \\
&\qquad\qquad\times
\sum_{\substack{k_1,\dots,k_r\in\mathbb{Z}\\[1pt] k_1+\cdots+k_r=0}}
\prod_{i=1}^r x_i^{(d+r)k_i-\chi(i\leq\ell)-\frac{\abs{\mu}-\ell}{r}} 
q^{d(d+r)\binom{k_i}{2}}
\det_{1\leq i,j\leq r}\Big( \big(x_iq^{dk_i})^{\mu_j+r-j}\Big)\Bigg),
\end{align*}
where the outer sum on the right is over partitions
$\mu=(\mu_1,\dots,\mu_r)$ such that $\mu_r=0$, $\mu_1\leq d$ and
$\abs{\mu}\equiv \ell \pmod{r}$.
This identity, which can be proved using standard $q$-series
methods, is true for all positive integers $d$ and $r$
and does not require $d$ and $r$ to be coprime.
For $d=1$, which fixes $\mu=(1^{\ell},0^{r-\ell})$,
the above identity simplifies to the equality between
\eqref{Eq_WK} for $\la=\La_{\ell}$ and \eqref{Eq_WKlevel1}.
\end{remark}

By comparing Proposition~\ref{Prop_branching} with \eqref{Eq_Bor} 
and \eqref{Eq_GK-AG} we conclude that
\begin{align*}
\GK_{(c_0,\dots,c_{r-1})}(1,q)&=\frac{1}{(q)_{\infty}}\,
\AG_{c_0\La_0+\dots+c_{r-1}\La_{r-1};r}(q) \\[1mm]
&=\frac{1}{(q)_{\infty}}\,
b_{(r/d-d-r+1+c_0)\La_0+c_1\La_1+\dots+c_{r-1}\La_{r-1}}(q),
\end{align*}
where $d=c_0+\dots+c_{r-1}$.

A third representation theoretic interpretation in terms of the 
$W_r$ algebra of Zamolodchikov \cite{Zamolodchikov85} and
Fateev and Lukyanov \cite{FL88} is very closely related to the above.
It is well known that the modulus-$(2k+1)$ Andrews--Gordon 
$q$-series, $\AG_{k,s}(q)$, correspond to the (normalised) characters 
of the certain non-unitary representations of the Virasoro 
algebra.
Specifically, let $\Vir$ denote the Virasoro algebra \cite{DFMS97}
with generators $L_n$ ($n\in\mathbb{Z}$), central element $c$
(not to be confused with the profile of a cylindric partition)
and commutation relations
\[
[L_m,L_n]=(m-n)L_{m+n}+\frac{c}{12}(m^3-m)\delta_{m+n,0}.
\]
Furthermore, let $\chi_{\Vir(c,h)}(q)=\Tr q^{L_0-c/24}$ be the character
of the highest weight representation of $\Vir$
with central charge and conformal weight given by
\[
c=-\frac{2(k-1)(6k-1)}{2k+1}\quad\text{and}\quad
h=-\frac{(s-1)(2k-s)}{2(2k+1)},
\]
where $1\leq s\leq k$.
Then (see e.g., \cite{CLM06,SF94}) 
\[
\chi_{\Vir(c,h)}(q)=q^{h-c/24}\,\AG_{k,s}(q).
\]
Based on a much more general character formula for $W_r$, see e.g., 
\cite{Mizoguchi91,Nakanishi90}, it was observed in \cite{ASW99}
that the above has a direct generalisation to $W_r$ as follows.
Fix a positive integer $d$, relatively prime to $r$, and let
$\mu=(\mu_1,\dots,\mu_r)$ be a partition such that $\mu_1-\mu_r\leq d$.
Let $\chi_{W_r(c,h)}(q)=\Tr q^{L_0-c/24}$ be the character
of the (non-unitary) highest weight representation of $W_r$ of central
charge
\[
c=-\frac{(d-1)(r-1)(d+r+dr)}{d+r}
\]
and conformal weight
\[
h=\frac{r}{2(d+r)}\bigg(\sum_{i=1}^r\mu_i^2-\frac{\abs{\mu}^2}{r}\bigg)
-\frac{d}{d+r}\sum_{i=1}^r \Big(\frac{r+1}{2}-i\Big)\mu_i.
\]
Then
\begin{align*}
\chi_{W_r(c,h)}(q)&=q^{h-c/24}\,
\AG_{\la;r}(q) \\
&=q^{h-c/24}\,(q)_{\infty}\GK_{(c_0,\dots,c_{r-1})}(q),
\end{align*}
where $\la$ is parametrised as in \eqref{Eq_la-mu} and
$c_i=\langle\alpha_i^{\vee},\la\rangle$.

\section{Hypergeometric preliminaries}\label{Sec_hyper}

In this section we prove two transformation formulas for sums over
$q$-binomial coefficients.
These transformations are needed in the proof of Theorem~\ref{Thm_k12} 
for $k=2$ and to transform the modulus-$8$ Andrews--Gordon identities
stated in next section into the modulus-$8$ identities discovered previously 
by Corteel, Dousse and Uncu.
Readers not particularly interested in identities for $q$-binomial
coefficients and in basic hypergeometric series may wish to skip this
section.

Throughout we use standard notation from the theory of basic hypergeometric
functions and $q$-series, see e.g., \cite{Andrews76,GR04}.
In general we view identities such as the $\mathrm{A}_2$ Andrews--Gordon
identities of Theorem~\ref{Thm_A2-vac-bothmoduli} from the point of view
of formal power series, with $q$ a formal variable.
In some of our proofs, however, we take an analytic approach, 
requiring complex $q$ such that $\abs{q}<1$.
\begin{comment}
The main two identities in this section, given by \eqref{Eq_ksum}
and \eqref{Eq_ksum-2} are polynomial identities (in $q$), so that
$q$ may simply be viewed as an indeterminate.
The same applies to Lemma~\ref{Lem_trafo}, which is a transformation
formula for terminating basic hypergeometric series, so that all
variables ($q$, $a$ and $c$) may be treated as indeterminates.
However in some of our proofs, such that the proofs of Lemmas~\ref{Lem_trafo} 
and \eqref{XXX} below, we do require an analytic point of view,
typically requiring complex variables with $\abs{q}<1$.
\end{comment}
The $q$-shifted factorials $(a)_{\infty}$ and $(a)_n$ are defined as
\[
(a)_{\infty}=(a;q)_{\infty}:=\prod_{i=0}^{\infty}(1-aq^i)
\quad\text{and}\quad
(a)_n=(a;q)_n:=\frac{(a)_{\infty}}{(aq^n)_{\infty}},
\]
where $n$ is an arbitrary integer.
In particular,
\[
(a)_n=\prod_{i=0}^{n-1}(1-aq^i)
\]
for $n$ a nonnegative integer, and
\begin{equation}\label{Eq_nul}
\frac{1}{(q)_n}=0
\end{equation}
for $n$ a negative integer.
We also adopt the usual condensed notation
\[
(a_1,\dots,a_k)_n = (a_1)_n\cdots (a_k)_n
\]
for $n\in\mathbb{Z}\cup\{\infty\}$.
For $n,m\in\mathbb{Z}$, the $q$-binomial coefficient $\qbin{n}{m}$ is 
given by
\begin{equation}\label{Eq_qbinomial}
\qbin{n}{m}=\qbin{n}{m}_q:=
\begin{cases}
\displaystyle \frac{(q)_n}{(q)_m(q)_{n-m}}
&\text{if $0\leq m\leq n$}, \\[3mm]
0 & \text{otherwise}.
\end{cases}
\end{equation}
The ${_r\phi_s}$ basic hypergeometric series is defined as \cite{GR04}
\begin{equation}\label{Eq_bhs}
\qHyp{r}{s}{a_1,\dots,a_r}{b_1,\dots,b_s}{q,z}:=
\sum_{k=0}^{\infty} \frac{(a_1,\dots,a_r)_k}{(q,b_1,\dots,b_s)_k}
\Big((-1)^k q^{\binom{k}{2}}\Big)^{s-r+1} z^k,
\end{equation}
where it is assumed that none of the $b_i$ is of the form $q^{-n}$ for 
some nonnegative integer $n$.
The series \eqref{Eq_bhs} is said to be terminating if one of the
$a_i$ is of the form $q^{-n}$ for $n$ a nonnegative integer, in which
case we may assume all the variables to be indeterminates.
If the series is non-terminating, we typically assume that 
the $a_i$, $b_j$ as well as $z$ and $q$ are complex such that
$\abs{z},\abs{q}<1$.

\medskip

We begin with a simple transformation formula for terminating basic 
hypergeometric series that we failed to find in the literature.

\begin{lemma}\label{Lem_trafo}
For $n$ a nonnegative integer,
\begin{equation}\label{Eq_been}
\sum_{k=0}^n  \frac{(q^{-n})_k (c)_{2k}}{(q,aq,c/a)_k}\, q^k
=\frac{(c)_n}{(aq^{1-n}/c)_n}\, 
\sum_{k=0}^n \frac{(q^{-n})_k(aq^{1-n}/c)_{2k}}
{(q,aq,q^{1-n}/c)_k}\, q^k.
\end{equation}
\end{lemma}

Note that the transformation \eqref{Eq_been} corresponds to the symmetry 
\begin{equation}\label{Eq_fn-symmetry}
f_n(a,c;q)=f_n(a,aq^{1-n}/c;q)
\end{equation}
for the basic hypergeometric function
\begin{subequations}
\begin{align}
f_n(a,c;q)&:=(aq,aq^{1-n}/c)_n
\sum_{k=0}^n  \frac{(q^{-n})_k (c)_{2k}}{(q,aq,c/a)_k}\,q^k \\
&\hphantom{:}=\Big({-}\frac{a}{c}\Big)^n q^{-\binom{n}{2}}
\sum_{k=0}^n (-1)^k q^{\binom{k}{2}+(1-n)k} 
(aq^{k+1},cq^k/a)_{n-k}(c)_{2k} \qbin{n}{k}.
\label{Eq_regularised}
\end{align}
\end{subequations}
Since
\[
(c)_{2k}=\big(c^{1/2},-c^{1/2},(cq)^{1/2},-(cq)^{1/2}\big),
\]
\eqref{Eq_been} corresponds to a transformation formula for a
terminating $\qhyp{5}{4}$ series in which two of the denominator 
parameters are equal to zero.

\begin{remark}
More generally we have
\begin{equation}\label{Eq_twocases}
\sum_{k=0}^n \frac{(q^{-n})_k(bc)_{2k}}{(q,abq,bc/a)_k}\,q^k
=\frac{(c)_n}{(aq^{1-n}/c)_n}\, 
\sum_{k=0}^n \frac{(q^{-n})_k(abq^{1-n}/c)_{2k}}
{(q,abq,bq^{1-n}/c)_k}\,q^k,
\end{equation}
provided one of $a,b$ is equal to $1$. 
Since our proof of the two cases is very different and we only require
the $b=1$ case, we leave the proof of \eqref{Eq_twocases} for $a=1$ to
the reader.
\end{remark}

\begin{comment}
The second transformation needed is less elegant, and can either be proved 
along the same lines as \eqref{Eq_been} (with a few additional steps) or 
be obtained from \eqref{Eq_been} by contiguity. 
It is the latter approach that is followed below.

\begin{corollary}\label{Cor_been}
For $n$ a nonnegative integer, 
\begin{align*}
\sum_{k=0}^n  \frac{(q^{-n})_k (c)_{2k}}{(q,aq,c/a)_k}\, q^{2k}
&=\frac{q^{-n}}{c}\,\frac{(c)_n}{(aq^{1-n}/c)_n} \, 
\sum_{k=0}^n \frac{(q^{-n})_k(aq^{1-n}/c)_{2k}}
{(q,aq,q^{1-n}/c)_k}\, q^{2k}  \\
&\quad-\frac{q^{-n}}{c}\,\frac{1-aq^{n+1}}{1-aq}\,
\frac{(c)_{n+1}}{(aq^{1-n}/c)_n} \, 
\sum_{k=0}^n \frac{(q^{-n})_k(aq^{1-n}/c)_{2k}}
{(q,aq^2,q^{-n}/c)_k}\, q^k.
\end{align*}
\end{corollary}
\end{comment}

The two identities required in Section~\ref{Sec_Mod8} follow from 
Lemma~\ref{Lem_trafo} through specialisation.

\begin{corollary}\label{Cor_qbinomials}
For integers $\ell,m,n$ such that $0\leq n\leq\ell$,
\begin{equation}\label{Eq_ksum}
\sum_{k=0}^n q^{(k-n)(2k+\ell+2m-n)} \qbin{n}{k}\qbin{\ell-k}{k+\ell+m-n}
=\sum_{k=0}^n q^{k(2k-\ell-2m-n)} \qbin{n}{k}\qbin{\ell-k}{k-m}
\end{equation}
and
\begin{align}\label{Eq_ksum2}
\sum_{k=0}^n & q^{(k-n)(2k+\ell+2m-n+1)}
\qbin{n}{k}\qbin{\ell-k}{k+\ell+m-n} \\
&\qquad\qquad=q^{-m-n} \sum_{k=0}^n q^{k(2k-\ell-2m-n+1)}
\qbin{n}{k}\qbin{\ell-k}{k-m} \notag \\
&\qquad\qquad\quad+\big(1-q^{\ell-n}\big)\sum_{k=0}^n q^{k(2k-\ell-2m-n-1)}
\qbin{n}{k}\qbin{\ell-k-1}{k-m-1}. \notag
\end{align}
\end{corollary}

Since $0\leq n\leq \ell$, both identities trivialise to 
$0=0$ unless $-\ell\leq m\leq n$, in which case each side is a
(non-zero) polynomial in $q$.

\begin{proof}[Proof of Lemma~\ref{Lem_trafo}] 
Since \eqref{Eq_been} is a transformation formula for terminating basic 
hypergeometric series, it suffices to give a proof for $c,q\in\mathbb{C}$ 
such that $\abs{c}<1$ and $\abs{q}<1$.

Our starting point for such a proof is the transformation
\begin{equation}\label{Eq_limit}
\qHyp{2}{1}{q^{-n},0}{c}{q,z}
=\frac{1}{(q^{1-n}/c)_n}\,
\qHyp{2}{0}{q^{-n},q/z}{\text{--}\,}{q,\frac{z}{c}}
\end{equation}
which follows by taking the $b\to 0$ limit in 
\cite[Equation (III.8)]{GR04}
\[
\qHyp{2}{1}{q^{-n},b}{c}{q,z}
=\frac{(c/b)_n}{(c)_n}\,b^n\,
\qHyp{3}{1}{q^{-n},b,q/z}{bq^{1-n}/c}{q,\frac{z}{c}}.
\]
Let $i$ be an arbitrary nonnegative integer.
Making the substitution $(c,z)\mapsto (aq^{1-i},q^{i+1})$ in
\eqref{Eq_limit}, then multiplying both sides by $(1/a)_i/(q)_i$ and 
manipulating some of the $q$-shifted factorials, gives
\[
\sum_{k=0}^n 
\frac{(q^{-n})_k(q^{-k}/a)_i}{(q,aq)_k(q)_i}\, q^{(2i+1)k}
=\frac{1}{(q^{-n}/a)_n} \,
\sum_{k=0}^{\min\{n,i\}} 
\frac{(q^{-n},q^{-n}/a)_k (q^{-(n-k)}/a)_{i-k}}
{(q)_k(q)_{i-k}}\, a^{-k} q^{ik}.
\]
Next we multiply both sides by $c^i$ and sum $i$ over the nonnegative 
integers.
After an interchange in the order of the sums and a shift $i\mapsto i+k$
on the right, this yields
\begin{align*}
\sum_{k=0}^n & \frac{(q^{-n})_k}{(q,aq)_k}\, q^k \,
\qHyp{1}{0}{q^{-k}/a}{\text{--}\,}{q,cq^{2k}} \\
&\quad=\frac{1}{(q^{-n}/a)_n} \, 
\sum_{k=0}^n \frac{(q^{-n},q^{-n}/a)_k}{(q)_k}
\Big(\frac{c}{a}\Big)^k q^{k^2}\,
\qHyp{1}{0}{q^{-(n-k)}/a}{\text{--}\,}{q,cq^k}
\end{align*}
for $\abs{c},\abs{q}<1$.
By the $q$-binomial theorem \cite[Equation (II.3)]{GR04}
\begin{equation}\label{Eq_qbt}
\qHyp{1}{0}{a}{\text{--}\,}{q,z}=\frac{(az)_{\infty}}{(z)_{\infty}}
\end{equation}
for $\abs{q},\abs{z}<1$, both $\qhyp{1}{0}$ series may be evaluated,
resulting in
\[
\sum_{k=0}^n \frac{(q^{-n})_k}{(q,aq)_k}\, q^k \,
\frac{(cq^k/a)_{\infty}}{(cq^{2k})_{\infty}}
=\frac{1}{(q^{-n}/a)_n} \, 
\sum_{k=0}^n \frac{(q^{-n},q^{-n}/a)_k}{(q)_k}
\Big(\frac{c}{a}\Big)^k q^{k^2}\,
\frac{(cq^{2k-n}/a)_{\infty}}{(cq^k)_{\infty}}.
\]
Multiplying both sides by $(c)_{\infty}/(c/a)_{\infty}$
and carrying out some simplification leads to
\[
\sum_{k=0}^n \frac{(q^{-n})_k (c)_{2k}}{(q,aq,c/a)_k}\, q^k \,
=\frac{(aq/c)_n}{(aq)_n} \, c^n \,
\sum_{k=0}^n \frac{(q^{-n},q^{-n}/a,c)_k}{(q)_k(cq^{-n}/a)_{2k}}
\Big(\frac{c}{a}\Big)^k q^{k^2}.
\]
Replacing $k\mapsto n-k$ on the right, and making a few more simplifications
yields \eqref{Eq_been}.
\end{proof}

\begin{proof}[Proof of Corollary~\ref{Cor_qbinomials}]
Let $\ell,m$ be integers and $n$ a nonnegative integer, and
specialise
\begin{equation}\label{Eq_nac}
(a,c)=\big(q^{-\ell-1},q^{m-n}\big)
\end{equation}
in the symmetry relation \eqref{Eq_fn-symmetry},
where we assume the regularised form for $f_n(a,c;q)$ given in
\eqref{Eq_regularised}.
This leads to
\begin{align}\label{Eq_lmn}
&\sum_{k=0}^n (-1)^k q^{\binom{k}{2}+(1-n)k}
(q^{k-\ell},q^{k+\ell+m-n+1})_{n-k}(q^{m-n})_{2k} \qbin{n}{k} \\
&\qquad=q^{n(\ell+2m-n)} \sum_{k=0}^n (-1)^k q^{\binom{k}{2}+(1-n)k}
(q^{k-\ell},q^{k-m+1})_{n-k}(q^{-\ell-m})_{2k} \qbin{n}{k}.
\notag
\end{align}
The reason for not directly using \eqref{Eq_been} is that the
specialisation \eqref{Eq_nac} gives
\[
\frac{1}{(aq)_k}=\frac{1}{(q^{-\ell})_k},\quad
\frac{1}{(c/a)_k}=\frac{1}{(q^{\ell+m-n+1})_k}\quad\text{and}\quad
\frac{1}{(q^{1-n}/c)_k}=\frac{1}{(q^{-m})_k},
\]
which all have the potential to lead to vanishing denominators for 
$1\leq k\leq n$.
If we impose the restrictions $n\leq \ell$ and $-\ell\leq m\leq n$
then \eqref{Eq_lmn} may be simplified to \eqref{Eq_ksum} by standard 
manipulations of $q$-shifted factorials.
As remarked previously, both sides of \eqref{Eq_ksum} trivially vanish
if $-\ell\leq m\leq n$ does not hold, so that we may again drop
this restriction.

To prove \eqref{Eq_ksum2}, we denote the left- and right-hand side of 
\eqref{Eq_ksum} by $L_{\ell,m,n}(q)$ and $R_{\ell,m,n}(q)$ respectively.
By
\[
\qbin{n}{k}=q^{-k}\qbin{n}{k}-q^{-k}(1-q^n)\qbin{n-1}{k-1}
\]
(which holds for all $k,n\in\mathbb{Z}$) it follows that
\begin{align*}
&\sum_{k=0}^n q^{(k-n)(2k+\ell+2m-n+1)}
\qbin{n}{k}\qbin{\ell-k}{k+\ell+m-n} \\
&\qquad=q^{-n}\sum_{k=0}^n q^{(k-n)(2k+\ell+2m-n)}
\qbin{n}{k}\qbin{\ell-k}{k+\ell+m-n} \\
&\quad\qquad-q^{-n}(1-q^n)\sum_{k=0}^n q^{(k-n)(2k+\ell+2m-n)}
\qbin{n-1}{k-1}\qbin{\ell-k}{k+\ell+m-n} \\
&\qquad=q^{-n}L_{\ell,m,n}(q)-q^{-n}(1-q^n)L_{\ell-1,m+1,n-1}(q) \\[2mm]
&\qquad=q^{-n}R_{\ell,m,n}(q)-q^{-n}(1-q^n)R_{\ell-1,m+1,n-1}(q) \\[2mm]
&\qquad=q^{-n} \sum_{k=0}^n q^{k(2k-\ell-2m-n)} 
\qbin{n}{k}\qbin{\ell-k}{k-m} \\
&\quad\qquad-q^{-n}(1-q^n)
\sum_{k=0}^n q^{k(2k-\ell-2m-n)} \qbin{n-1}{k}\qbin{\ell-k-1}{k-m-1}.
\end{align*}
Next we use
\[
\qbin{\ell-k}{k-m}=q^{k-m}\qbin{\ell-k}{k-m}+(1-q^{\ell-k})
\qbin{\ell-k-1}{k-m-1}
\]
in the first sum on the right and
\[
(1-q^n)\qbin{n-1}{k}=
q^{n-k}(1-q^{\ell-n})\qbin{n}{k}-(1-q^{\ell-k})\qbin{n}{k}
\]
in the second sum.
This leads to four sums on the right, two of which cancel.
The resulting identity is exactly \eqref{Eq_ksum2}.
\end{proof}

\section{The modulus-$8$ case}\label{Sec_Mod8}

In this section we focus on the $\mathrm{A}_2$ Andrews--Gordon identities
for modulus $8$, and the corresponding two-variable generating function for 
cylindric partitions of rank $3$ and level $5$.
This is also the modulus considered by Corteel, Dousse and Uncu in
\cite{CDU20}, and we make a comparison between their identities and ours.
Relating our sum sides to those in \cite{CDU20} is surprisingly intricate
and relies on the new transformation formulas \eqref{Eq_ksum} and
\eqref{Eq_ksum2}.

We begin by listing a total of $11$ modulus-$8$ identities.
The first six entries in our list are \eqref{Eq_A2-vac} and 
\eqref{Eq_ks-con} for $k=2$ as well as the 
companions of \eqref{Eq_ks-con} for $k=2$ and $s=2$ or $3$
given in Remark~\ref{Rem_alt-sums}:
\begin{subequations}
\begin{align}
\sum_{n_1,m_1,n_2=0}^{\infty} 
\frac{q^{n_1^2-n_1m_1+m_1^2+n_2^2}}{(q)_{n_1}}\,
\qbin{n_1}{n_2}\qbin{n_1+n_2}{m_1}
&=\frac{1}{(q,q,q^2,q^4,q^4,q^6,q^7,q^7;q^8)_{\infty}}, \notag \\[2mm]
\sum_{n_1,m_1,n_2=0}^{\infty} 
\frac{q^{n_1^2-n_1m_1+m_1^2+n_2^2+n_1+m_1+n_2}}{(q)_{n_1}}\, 
\qbin{n_1}{n_2}\qbin{n_1+n_2}{m_1}
&=\frac{1}{(q^2,q^3,q^3,q^4,q^4,q^5,q^5,q^6;q^8)_{\infty}}, \notag \\[2mm]
\label{Eq_pair1} 
\sum_{n_1,m_1,n_2=0}^{\infty} 
\frac{q^{n_1^2-n_1m_1+m_1^2+n_2^2+m_1+n_2}}{(q)_{n_1}}\,
\qbin{n_1}{n_2}\qbin{n_1+n_2}{m_1} &= \\
\label{Eq_pair2}
\sum_{n_1,m_1,n_2=0}^{\infty} 
\frac{q^{n_1^2-n_1m_1+m_1^2+n_2^2+n_1+n_2}}{(q)_{n_1}}\,
\qbin{n_1}{n_2}\qbin{n_1+n_2+1}{m_1} 
&=\frac{1}{(q,q^2,q^3,q^4,q^4,q^5,q^6,q^7;q^8)_{\infty}} 
\intertext{and}
\label{Eq_122}
\sum_{n_1,m_1,n_2=0}^{\infty} 
\frac{q^{n_1^2-n_1m_1+m_1^2+n_2^2+m_1}}{(q)_{n_1}}\,
\qbin{n_1}{n_2}\qbin{n_1+n_2}{m_1}& = \\ 
\sum_{n_1,m_1,n_2=0}^{\infty} 
\frac{q^{n_1^2-n_1m_1+m_1^2+n_2^2+n_1}}{(q)_{n_1}}\,
\qbin{n_1+1}{n_2}\qbin{n_1+n_2}{m_1}
&=\frac{1}{(q,q^2,q^2,q^3,q^5,q^6,q^6,q^7;q^8)_{\infty}}.
\end{align}
\end{subequations}
Next, the $k=s=2$ case of Theorem~\ref{Thm_A2-three} is
\begin{align*}
\sum_{n_1,m_1,n_2=0}^{\infty} &
\frac{q^{n_1^2-n_1m_1+m_1^2+n_2^2+n_2}}{(q)_{n_1}}\,
\qbin{n_1}{n_2}\qbin{n_1+n_2+1}{m_1} \\
&=\frac{1}{(q,q,q^3,q^3,q^5,q^5,q^7,q^7;q^8)_{\infty}}
+\frac{q}{(q^2,q^3,q^3,q^4,q^4,q^5,q^5,q^6;q^8)_{\infty}}.
\end{align*}
According to \eqref{Eq_rank3-count} with $k=1$ and $i=2$ there are
a total of five infinite products for modulus $8$, so that one
product is missing from the above.
By setting $z=1$ in the identity for $\GKt_{(3,1,1)}(z,q)$ given in 
Proposition~\ref{Prop_remainingcases} below and using
\[
\GKt_{(3,1,1)}(1;q)=\frac{1}{(q,q,q^3,q^3,q^5,q^5,q^7,q^7;q^8)_{\infty}},
\]
it follows after some standard manipulations of $q$-binomial coefficients
that 
\begin{align*}
\sum_{n_1,m_1,n_2=0}^{\infty} 
& \frac{q^{n_1^2-n_1m_1+m_1^2+n_2^2+m_1}}{(q)_{n_1}}\,
\qbin{n_1+1}{n_2}\qbin{n_1+n_2}{m_1} \\
&=\sum_{n_1,m_1,n_2=0}^{\infty} 
\frac{q^{n_1^2-n_1m_1+m_1^2+n_2^2+n_1}}{(q)_{n_1}}\,
\qbin{n_1+1}{n_2}\qbin{n_1+n_2+1}{m_1} \\
&=\frac{1}{(q,q,q^3,q^3,q^5,q^5,q^7,q^7;q^8)_{\infty}}. 
\end{align*}
Similarly, with a bit of work one may transform \eqref{Eq_122} to yield our 
final two modulus-$8$ identities:
\begin{align*}
\sum_{n_1,m_1,n_2=0}^{\infty} 
& \frac{q^{n_1^2-n_1m_1+m_1^2+n_2^2+m_1+n_2}}{(q)_{n_1}}\,
\qbin{n_1}{n_2}\qbin{n_1+n_2+1}{m_1} \\
&=\sum_{n_1,m_1,n_2=0}^{\infty}
\frac{q^{n_1^2-n_1m_1+m_1^2+n_2^2+n_1+n_2}}{(q)_{n_1}}\,
\qbin{n_1+1}{n_2}\qbin{n_1+n_2+1}{m_1} \\
&=\frac{1}{(q,q^2,q^2,q^3,q^5,q^6,q^6,q^7;q^8)_{\infty}}.
\end{align*}

\medskip
The five distinct infinite products should be compared with
the seven inequivalent profiles as counted by \eqref{Eq_number-profiles}
for $r=3$ and $d=5$.
For all seven of the corresponding two-variable generating functions 
$\GK_c(z,q)$ for cylindric partitions of rank $3$ and $5$ there is a 
manifestly positive triple-sum expression, generalising some of the
triple sums stated above.
In the next theorem we state the first four of these results,
corresponding to the $k=2$ case of Theorem~\ref{Thm_k12}.

\begin{theorem}\label{Thm_GK-mod8}
We have
\begin{align*}
\GK_{(2,2,1)}(z,q)&=\frac{1}{(zq)_{\infty}}
\sum_{n_1,m_1,n_2=0}^{\infty}
\frac{z^{n_1}q^{n_1^2-n_1m_1+m_1^2+n_2^2}}{(q)_{n_1}}\,
\qbin{n_1}{n_2}\qbin{n_1+n_2}{m_1}, \\
\GK_{(3,2,0)}(z,q)&=\frac{1}{(zq)_{\infty}}
\sum_{n_1,m_1,n_2=0}^{\infty} 
\frac{z^{n_1}q^{n_1^2-n_1m_1+m_1^2+n_2^2+m_1}}{(q)_{n_1}}\,
\qbin{n_1}{n_2}\qbin{n_1+n_2}{m_1}, \\
\GK_{(4,1,0)}(z,q)&=\frac{1}{(zq)_{\infty}}
\sum_{n_1,m_1,n_2=0}^{\infty}
\frac{z^{n_1}q^{n_1^2-n_1m_1+m_1^2+n_2^2+m_1+n_2}}{(q)_{n_1}}\,
\qbin{n_1}{n_2}\qbin{n_1+n_2}{m_1}, \\
\GK_{(5,0,0)}(z,q)&=\frac{1}{(zq)_{\infty}}
\sum_{n_1,m_1,n_2=0}^{\infty} 
\frac{z^{n_1}q^{n_1^2-n_1m_1+m_1^2+n_2^2+n_1+m_1+n_2}}{(q)_{n_1}}\,
\qbin{n_1}{n_2}\qbin{n_1+n_2}{m_1}.
\end{align*}
\end{theorem}

For the remaining three cases we have more complicated series.
To shorten the expressions below, we introduce the normalisation
\[
\GKt_c(z,q):=(zq)_{\infty}\GK_c(z,q).
\]
We also recall that $1/(q)_n=0$ when $n$ is a negative integer.

\begin{proposition}\label{Prop_remainingcases}
We have
\begin{align*}
&\GKt_{(3,1,1)}(z,q) \\
&\quad =
1+\sum_{n_1,m_1,n_2=0}^{\infty} 
\frac{z^{n_1} q^{n_1^2-n_1m_1+m_1^2+n_2^2+m_1}(1+q^{m_1-n_1+n_2+1})}
{(q)_{n_1}}\,\qbin{n_1}{n_2}\qbin{n_1+n_2-1}{m_1},
\end{align*}
\begin{align*}
&\GKt_{(4,0,1)}(z,q) \\
&\quad=1+\sum_{n_1,m_1,n_2=0}^{\infty} 
\frac{z^{n_1} q^{n_1^2-n_1m_1+m_1^2+n_2^2+n_1+m_1}
(1+q^{m_1-n_1+n_2+1})}
{(q)_{n_1}}\,\qbin{n_1}{n_2}\qbin{n_1+n_2-1}{m_1}\\
&\qquad+\sum_{n_1,m_1,n_2=0}^{\infty} 
\frac{z^{n_1}q^{n_1^2-n_1m_1+m_1^2+n_2^2+n_1+n_2-1}}
{(q)_{n_1-1}} \, \qbin{n_1-1}{n_2}\qbin{n_1+n_2-1}{m_1-1},
\end{align*}
and
\begin{align*}
&\GKt_{(3,0,2)}(z,q) \\
&\quad=zq+
\sum_{n_1,m_1,n_2=0}^{\infty} \frac{z^{n_1}q^{n_1^2-n_1m_1+m_1^2+n_2^2+n_1}}
{(q)_{n_1}} \, \qbin{n_1}{n_2}\qbin{n_1+n_2}{m_1} \\
&\qquad+\sum_{n_1,m_1,n_2=0}^{\infty} 
\frac{z^{n_1} q^{n_1^2-n_1m_1+m_1^2+n_2^2+n_1-1}
(1+q^{m_1-n_1+n_2+1})}
{(q)_{n_1-1}}\,\qbin{n_1-1}{n_2}\qbin{n_1+n_2-2}{m_1-1}\\
&\qquad+\sum_{n_1,m_1,n_2=0}^{\infty} \frac{z^{n_1}q^{n_1^2-n_1m_1+m_1^2
+n_2^2+2n_1+n_2-1}}
{(q)_{n_1-1}} \, \qbin{n_1-1}{n_2}\qbin{n_1+n_2-1}{m_1-1}.
\end{align*}
\end{proposition}

\begin{remark}
Simpler expressions for $\GKt_{(3,1,1)}(z,q)$, $\GKt_{(4,0,1)}(z,q)$
and $\GKt_{(3,0,2)}(z,q)$ may be given that are not manifestly positive.
For example,
\begin{align*}
\GKt_{(4,0,1)}(z,q) 
&=\sum_{n_1,m_1,n_2=0}^{\infty} 
\frac{z^{n_1} q^{n_1^2-n_1m_1+m_1^2+n_2^2+n_1+n_2}}
{(q)_{n_1}}\,\qbin{n_1}{n_2}\qbin{n_1+n_2+1}{m_1}\\
&\quad+(z-1)\sum_{n_1,m_1,n_2=0}^{\infty} 
\frac{z^{n_1}q^{n_1^2-n_1m_1+m_1^2+n_2^2+2n_1+m_1+n_2+1}}
{(q)_{n_1}} \, \qbin{n_1}{n_2}\qbin{n_1+n_2}{m_1}.
\end{align*}
By the equality between \eqref{Eq_pair1} and \eqref{Eq_pair2} this
immediately implies that
\[
\GK_{(4,1,0)}(1,q)=\GK_{(4,0,1)}(1,q).
\]
\end{remark}

As mentioned in the introduction, the expressions of
Corteel, Dousse and Uncu for $\GK_c(z,q)$ and the corresponding
Rogers--Ramanujan-type identities for $\GK_c(1,q)$ 
take the form of quadruple sums instead of triple sums.
We will not state the complete list of seven generating functions
from their paper, and instead focus our attention on the four cases 
related to Theorem~\ref{Thm_GK-mod8}.
For the remaining three identities, see \cite[Theorem 3.2]{CDU20}.

\begin{theorem}[Corteel, Dousse, Uncu]\label{Thm_CDU}
There holds
\begin{align*}
\GKt_{(2,2,1)}(z,q)&=\sum_{n_1,n_2,n_3,n_4=0}^{\infty}
\frac{z^{n_1} q^{n_1^2+n_2^2+n_3^2+n_4^2-n_1n_2+n_2n_4}}
{(q;q)_{n_1}}\,\qbin{n_1}{n_2}\qbin{n_1}{n_4}\qbin{n_2}{n_3}, \\
\GKt_{(4,1,0)}(z,q)&=\sum_{n_1,n_2,n_3,n_4=0}^{\infty}
\frac{z^{n_1} q^{n_1^2+n_2^2+n_3^2+n_4^2-n_1n_2+n_2n_4+n_2+n_3+n_4}}
{(q;q)_{n_1}}\,\qbin{n_1}{n_2}\qbin{n_1}{n_4}\qbin{n_2}{n_3}, \\
\GKt_{(5,0,0)}(z,q)&=\sum_{n_1,n_2,n_3,n_4=0}^{\infty}
\frac{z^{n_1} q^{n_1^2+n_2^2+n_3^2+n_4^2-n_1n_2+n_2n_4+n_1+n_2+n_3+n_4}}
{(q;q)_{n_1}}\,\qbin{n_1}{n_2}\qbin{n_1}{n_4}\qbin{n_2}{n_3}
\end{align*}
and
\begin{align*}
&\GKt_{(3,2,0)}(z,q) \\
&\quad =\sum_{n_1,n_2,n_3,n_4=0}^{\infty} 
\frac{z^{n_1} q^{n_1^2+n_2^2+n_3^2+n_4^2-n_1n_2+n_2n_4+n_1}}
{(q;q)_{n_1}}\,\qbin{n_1}{n_2}\qbin{n_1}{n_4}\qbin{n_2}{n_3} \\
&\qquad +\sum_{n_1,n_2,n_3,n_4=0}^{\infty} \bigg( 
\frac{z^{n_1} q^{n_1^2+n_2^2+n_3^2+n_4^2-n_1n_2+n_2n_4+n_2}
(1+q^{n_1+n_3}+q^{2n_1+n_2+n_3+n_4})}{(q;q)_{n_1-1}}\\
&\qquad\qquad\qquad\qquad\qquad\times
\qbin{n_2}{n_3}\qbin{n_1-1}{n_2}\qbin{n_1-1}{n_4}\bigg).
\end{align*}
\end{theorem}

To prove Theorem~\ref{Thm_GK-mod8} it suffices to transform each
of the triple-sums for $\GK_c(z,q)$ into the corresponding
quadruple sum given in Theorem~\ref{Thm_CDU}.
We will prove two stronger results as follows.

\begin{proposition}\label{Prop_threeisfour}
We have
\begin{subequations}
\begin{align}\label{Eq_34a}
&\sum_{n_1,n_2,m_1=0}^{\infty}
\frac{z^{n_1} w^{m_1+n_2} q^{n_1^2-n_1m_1+m_1^2+n_2^2}}
{(q)_{n_1}}\,\qbin{n_1}{n_2}\qbin{n_1+n_2}{m_1} \\
&\quad=\sum_{n_1,n_2,n_3,n_4=0}^{\infty}
\frac{z^{n_1} w^{n_2+n_3+n_4} q^{n_1^2+n_2^2+n_3^2+n_4^2-n_1n_2+n_2n_4}}
{(q)_{n_1}}\,\qbin{n_1}{n_2}\qbin{n_1}{n_4}\qbin{n_2}{n_3}, \notag
\end{align}
and
\begin{align}\label{Eq_34b}
&\sum_{n_1,n_2,m_1=0}^{\infty}
\frac{z^{n_1} w^{m_1+n_2} q^{n_1^2-n_1m_1+m_1^2+n_2^2+m_1}}
{(q)_{n_1}}\,\qbin{n_1}{n_2}\qbin{n_1+n_2}{m_1} \\
&\quad=\sum_{n_1,n_2,n_3,n_4=0}^{\infty} 
\frac{z^{n_1}w^{n_2+n_3+n_4}q^{n_1^2+n_2^2+n_3^2+n_4^2-n_1n_2+n_2n_4+n_1}}
{(q;q)_{n_1}}\,
\qbin{n_1}{n_2}\qbin{n_1}{n_4}\qbin{n_2}{n_3} \notag \\
&\qquad+\sum_{n_1,n_2,n_3,n_4=0}^{\infty}
\bigg( \frac{z^{n_1}w^{n_2+n_3+n_4} 
q^{n_1^2+n_2^2+n_3^2+n_4^2-n_1n_2+n_2n_4+n_2}}{(q;q)_{n_1-1}} \notag \\
&\qquad\qquad\qquad\qquad\quad\times
(1+wq^{n_1+n_3}+w^2q^{2n_1+n_2+n_3+n_4})
\qbin{n_1-1}{n_2}\qbin{n_1-1}{n_4}\qbin{n_2}{n_3}\bigg). \notag
\end{align}
\end{subequations}
\end{proposition}

Setting $w=1$ in \eqref{Eq_34a} proves the equality of the two
expressions for $\GK_{(2,2,1)}(z,q)$,
setting $w=q$ in \eqref{Eq_34a} proves the equality of the two
expressions for $\GK_{(4,1,1)}(z,q)$ and
setting $w=q$ and replacing $z\mapsto zq$ in \eqref{Eq_34a} proves 
the equality of the two expressions for $\GK_{(5,0,0)}(z,q)$.
Finally, setting $w=1$ in \eqref{Eq_34b} proves the
equality of the two expressions for $\GK_{(3,2,0)}(z,q)$.

\begin{proof}[Proof of Proposition~\ref{Prop_threeisfour}]
Equating coefficients of
\[
\frac{z^{n_1} w^{n_2} q^{n_1^2-n_1n_2+n_2^2}}{(q)_{n_1}},
\]
in \eqref{Eq_34a}, we need to show that
\begin{align*}
&\sum_{m_1=0}^{n_2} q^{(m_1-n_2)(2m_1-n_1)}
\qbin{n_1}{n_2-m_1}\qbin{n_1+n_2-m_1}{m_1} \\
&\quad=\sum_{n_3,n_4=0}^{n_1}
q^{n_3(n_1-2n_2+2n_3)+n_4(n_1-n_2+n_3+n_4)}
\qbin{n_1}{n_2-n_3-n_4}\qbin{n_1}{n_4}\qbin{n_2-n_3-n_4}{n_3}
\end{align*}
for arbitrary nonnegative integers $n_1,n_2$.
Since 
\begin{equation}\label{Eq_rewrite}
\qbin{n_1}{n_2-n_3-n_4}\qbin{n_2-n_3-n_4}{n_3}=
\qbin{n_1}{n_3}\qbin{n_1-n_3}{n_1-n_2+n_3+n_4},
\end{equation}
we can carry out the sum over $n_4$ using the $q$-Chu--Vandermonde
summation \cite[Equation (II.7)]{GR04}
\begin{equation}\label{Eq_qCV}
\sum_{k=0}^{N_1} q^{k(k+m)}\qbin{N_1}{k}\qbin{N_2}{k+m}=
\qbin{N_1+N_2}{N_1+m}
\end{equation}
with $(k,N_1,N_2,m)\mapsto (n_4,n_1,n_1-n_3,n_1-n_2+n_3)$.
Also replacing $m_1\mapsto  k-n_1+n_2$ on the left and
$n_3\mapsto k$ on the right we obtain \eqref{Eq_ksum} 
(with $(\ell,m,n)=(2n_1,n_2-2n_1,n_1)$), completing the proof of
\eqref{Eq_34a}.

To prove \eqref{Eq_34b} we proceed in the exact same manner as before,
which leaves us to show the following somewhat unwieldy polynomial
identity:
\begin{align*}
&\sum_{m_1=0}^{n_2} q^{(m_1-n_2)(2m_1-n_1)+m_1}
\qbin{n_1}{n_2-m_1}\qbin{n_1+n_2-m_1}{m_1} \\
&\quad=q^{n_1} \sum_{n_3,n_4=0}^{n_1} 
q^{\varphi_{1,0}(n_1,n_2)}
\qbin{n_1}{n_3}\qbin{n_1-n_3}{n_1-n_2+n_3+n_4}\qbin{n_1}{n_4} \\
&\qquad+(1-q^{n_1})q^{n_2}\sum_{n_3,n_4=0}^{n_1-1} 
q^{\varphi_{-1,-1}(n_1,n_2)}
\qbin{n_1-1}{n_3}\qbin{n_1-n_3-1}{n_1-n_2+n_3+n_4-1}\qbin{n_1-1}{n_4} \\
&\qquad+(1-q^{n_1})q^{2n_1-n_2}\sum_{n_3,n_4=0}^{n_1-1}
q^{\varphi_{2,0}(n_1,n_2)}
\qbin{n_1-1}{n_3}\qbin{n_1-n_3-1}{n_1-n_2+n_3+n_4}\qbin{n_1-1}{n_4} \\
&\qquad+(1-q^{n_1})q^{4n_1-2n_2} \sum_{n_3,n_4=0}^{n_1-1} 
q^{\varphi_{3,1}(n_1,n_2)}
\qbin{n_1-1}{n_3}\qbin{n_1-n_3-1}{n_1-n_2+n_3+n_4+1}\qbin{n_1-1}{n_4},
\end{align*}
where
\[
\varphi_{a,b}(n_1,n_2):=n_3(n_1-2n_2+2n_3+a)+n_4(n_1-n_2+n_3+n_4+b).
\] 
To obtain the above identity we have not just extracted coefficients
of $z^{n_1}w^{n_2} q^{n_1^2-n_1n_2+n_2}/(q)_{n_1}$ as before, but also
applied \eqref{Eq_rewrite} (and its variants obtained by
shifting $n_1\mapsto n_1-1$ or $(n_1,n_2)\mapsto (n_1-1,n_2-1)$ or
$(n_1,n_2)\mapsto (n_1-1,n_2-2)$) to rewrite the products of $q$-binomial
coefficients on the right.
In all four double sums on the right the sum over $n_4$
can be carried out by the $q$-Chu--Vandermonde summation \eqref{Eq_qCV}.
Also replacing $m_1\mapsto k-n_1+n_2$ on the left and $n_3\mapsto k$ 
on the right, this results in
\begin{align*}
&\sum_{k=0}^{n_1} q^{(k-n_1)(2k-3n_1+2n_2+1)}
\qbin{n_1}{k}\qbin{2n_1-k}{k-n_1+n_2} \\
&\qquad=q^{n_1-n_2} \sum_{k=0}^{n_1} 
q^{k(2k+n_1-2n_2+1)} \qbin{n_1}{k}\qbin{2n_1-k}{k+2n_1-n_2} \\
&\quad\qquad+(1-q^{n_1})\sum_{k=0}^{n_1-1} 
q^{k(2k+n_1-2n_2-1)} \qbin{n_1-1}{k}\qbin{2n_1-k-2}{k+2n_1-n_2-2} \\
&\quad\qquad+(1-q^{n_1})q^{2n_1-2n_2}\sum_{k=0}^{n_1-1}
q^{k(2k+n_1-2n_2+2)} \qbin{n_1-1}{k}\qbin{2n_1-k-2}{k+2n_1-n_2-1} \\
&\quad\qquad+(1-q^{n_1})q^{4n_1-3n_2} \sum_{k=0}^{n_1-1} 
q^{k(2k+n_1-2n_2+3)} \qbin{n_1-1}{k}\qbin{2n_1-k-2}{k+2n_1-n_2}.
\end{align*}
Denote the four terms on the right by $s_1,\dots,s_4$ respectively.
By 
\begin{equation}\label{Eq_qbin-rec}
q^n \qbin{n+m}{n}+\qbin{n+m}{n-1}=\qbin{n+m+1}{n}, \quad
(n,m)\in\mathbb{Z}^2\setminus\{(0,-1)\},
\end{equation}
with $(n,m)=(k+2n_1-n_2,n_2-2k-2)$
it follows that 
\[
s_3+s_4=(1-q^{n_1})q^{2n_1-2n_2}\sum_{k=0}^{n_1-1}
q^{k(2k+n_1-2n_2+2)} \qbin{n_1-1}{k}\qbin{2n_1-k-1}{k+2n_1-n_2}=:s_3'.
\]
Note that the potentially problematic case 
$(n,m)=(k+2n_1-n_2,n_2-2k-2)=(0,-1)$
does not arise since it implies $k=2n_1-1$ which would give
$\qbin{n_1-1}{k}=\qbin{n_1-1}{2n_1-1}=0$.
Next we use
\begin{align}\label{Eq_qbin22}
\qbin{n+m-1}{n-2}&+q^{n-m-1}\qbin{n+m}{n} \\
&\quad=\qbin{n+m}{n-1}+q^{n-m-1} \qbin{n+m-1}{n}, \quad
(n,m)\in\mathbb{Z}^2\setminus\{(0,0),(1,-1)\}, \notag
\end{align}
with $(n,m)=(k+2n_1-n_2,n_2-2k-1)$ to find that
\begin{align}\label{Eq_s2s3p}
s_2+s_3'&=
(1-q^{n_1})\sum_{k=0}^{n_1-1} q^{k(2k+n_1-2n_2-1)}
\qbin{n_1-1}{k}\qbin{2n_1-k-1}{k+2n_1-n_2-1} \\
&\quad+(1-q^{n_1})q^{2n_1-2n_2}
\sum_{k=0}^{n_1-1} q^{k(2k+n_1-2n_2+2)}
\qbin{n_1-1}{k}\qbin{2n_1-k-2}{k+2n_1-n_2}. \notag
\end{align}
Again the two exceptional cases in \eqref{Eq_qbin22}
do not pose a problem as they both correspond to $k=2n_1-2$.
If we now replace $k\mapsto k-1$ in the second sum on the right of
\eqref{Eq_s2s3p} we get
\begin{align*}
s_2+s_3'&=(1-q^{n_1})\sum_{k=0}^{n_1-1} 
q^{k(2k+n_1-2n_2-1)} \qbin{n_1-1}{k}\qbin{2n_1-k-1}{k+2n_1-n_2-1} \\
&\quad+(1-q^{n_1})q^{n_1}\sum_{k=1}^{n_1}
q^{k(2k+n_1-2n_2-2)} \qbin{n_1-1}{k-1}
\qbin{2n_1-k-1}{k+2n_1-n_2-1} \\
&=(1-q^{n_1})\sum_{k=0}^{n_1} 
q^{k(2k+n_1-2n_2-1)} \qbin{n_1}{k}\qbin{2n_1-k-1}{k+2n_1-n_2-1},
\end{align*}
where the second equality follows from \eqref{Eq_qbin-rec}
with $(n,m)=(n_1-k,k-1)$.
The upshot of the above manipulations is that it only remains to be
shown that
\begin{align*}
&\sum_{k=0}^{n_1} q^{(k-n_1)(2k-3n_1+2n_2+1)}
\qbin{n_1}{k}\qbin{2n_1-k}{k-n_1+n_2} \\
&\quad=q^{n_1-n_2} \sum_{k=0}^{n_1} 
q^{k(2k+n_1-2n_2+1)} \qbin{n_1}{k}\qbin{2n_1-k}{k+2n_1-n_2} \\
&\qquad +(1-q^{n_1})\sum_{k=0}^{n_1}
q^{k(2k+n_1-2n_2-1)} \qbin{n_1}{k}\qbin{2n_1-k-1}{k+2n_1-n_2-1},
\end{align*}
for integers $n_1,n_2$.
Since this is \eqref{Eq_ksum2} with $(\ell,m,n)=(2n_1,n_2-2n_1,n_1)$
we are done.
\end{proof}

\begin{proof}[Proof of Proposition~\ref{Prop_remainingcases}]
We can either proceed as in the proof of Theorem~\ref{Thm_GK-mod8}
or, more simply, use
the three functional equations \cite[Equations (3.17)--(3.19)]{CDU20}
\begin{subequations}\label{Eq_functional-eqn}
\begin{align}
\GKt_{(4,0,1)}(z,q)&=\GKt_{(3,1,1)}(zq;q)+zq\GKt_{(4,1,0)}(zq^2;q) \\
\GKt_{(3,0,2)}(z,q)&=\GKt_{(2,2,1)}(zq;q)+zq\GKt_{(3,1,1)}(zq^2;q)
+zq^2 \GKt_{(4,1,0)}(zq^3;q) \\
\GKt_{(3,2,0)}(z,q)&=\GKt_{(3,1,1)}(zq;q)+zq\GKt_{(2,2,1)}(zq^2;q) \\
&\quad +zq^2\GKt_{(3,1,1)}(zq^3;q)+zq^3\GKt_{(4,1,0)}(zq^4;q). \notag
\end{align}
\end{subequations}
Given $\GKt_{(4,1,0)}(z,q)$, $\GKt_{(3,2,0)}(z,q)$ and
$\GKt_{(2,2,1)}(z,q)$, these three equations uniquely determine
$\GKt_{(4,0,1)}(z,q)$, $\GKt_{(3,0,2)}(z,q)$ and
$\GKt_{(3,1,1)}(z,q)$.
Substituting the expressions for the six generating functions in question,
as given by Theorem~\ref{Thm_GK-mod8} and (the as yet to be proven) 
Proposition~\ref{Prop_remainingcases}, it immediately follows that
\eqref{Eq_functional-eqn} holds, thus proving the proposition.
\end{proof}

\begin{remark}
With a bit of extra work one can show that the seven generating functions
given by Theorem~\ref{Thm_GK-mod8} and Proposition~\ref{Prop_remainingcases}
satisfy the full set of functional equations obtained in \cite{CDU20}.
Together with some simple initial conditions this provides a proof of
the theorem and proposition independent of Theorem~\ref{Thm_CDU}, and hence
provides a non-computer assisted proof.
The proof of Theorem~\ref{Thm_CDU} given in \cite{CDU20} heavily uses
the Mathematica packages \texttt{qFunctions} \cite{AU19} and 
\texttt{HolonomicFunctions}~\cite{Koutschan10}.
\end{remark}

\section{Proofs of the main results}\label{Sec_Proofs}

\subsection{Proofs of Theorems~\ref{Thm_A2-vac-bothmoduli},
\ref{Thm_threecases}, \ref{Thm_A2-three}, \ref{Thm_threecases-b} and 
\ref{Thm_A2-three-b}}

In this section we prove the $\mathrm{A}_2$ Andrews--Gordon
identities claimed in the introduction and in Section~\ref{Sec_Main}.
In each case, our starting point is one of the Rogers--Ramanujan-type
identities of Andrews, Schilling and the author, which were proved in 
\cite{ASW99} using an $\mathrm{A}_2$ analogue of the Bailey lemma. 
(See also \cite{W06} for a proof based on Hall--Littlewood polynomials).
We also present conditional proofs of Conjectures~\ref{Con_A2-two} and 
\ref{Con_A2-two-b} for all $s$.

\subsubsection{Proof of Theorem~\ref{Thm_A2-vac-bothmoduli}}
We begin by defining $F^{(a)}_{n_0,m_0;k}(z,q)\in\mathbb{Q}(q)[z]$ as
\begin{equation}\label{Eq_F1}
F^{(a)}_{n_0,m_0;k}(z,q):=
\sum_{\substack{n_1,\dots,n_k\geq 0 \\[1pt] m_1,\dots, m_k\geq 0}}
\frac{z^{n_1}}{(q)_{n_k+m_k}} \prod_{i=1}^k 
q^{n_i^2-\sigma_i n_im_i+m_i^2}
\qbin{n_{i-1}}{n_i}\qbin{m_{i-1}}{m_i},
\end{equation}
where $n_0,m_0$ are nonnegative integers, $k$ is a positive integer,
$a\in\{-1,1\}$ and
\begin{equation}\label{Eq_sigma-a}
(\sigma_1,\dots,\sigma_{k-1},\sigma_k):=
(1,\dots,1,a).
\end{equation}
If we further define
\[
F^{(a)}_k(z;q):=\lim_{n_0,m_0\to\infty}
F^{(a)}_{n_0,m_0;k}(z;q),
\]
then, according to \cite[Equations (5.22)\&(5.28), $i=k$]{ASW99},
\[
F^{(a)}_{k}(1;q)
=\frac{(q^{3k+a+3};q^{3k+a+3})_{\infty}^2}{(q)_{\infty}^3}\,
\theta(q^{k+1},q^{k+1},q^{k+a+1};q^{3k+a+3}).
\]
Comparing this with \eqref{Eq_A2-vac} and \eqref{Eq_A2-vac-b}
(with $k\mapsto k+1$ in the latter), we must thus show that
\begin{align*}
F_k^{(-1)}(1;q)&=\frac{1}{(q;q)_{\infty}}
\sum_{\substack{n_1,\dots,n_k\geq 0 \\[1pt] m_1,\dots,m_{k-1}\geq 0}}
\frac{q^{n_k^2}}{(q)_{n_1}} 
\prod_{i=1}^{k-1} q^{n_i^2-n_im_i+m_i^2}
\qbin{n_i}{n_{i+1}}\qbin{n_i-n_{i+1}+m_{i+1}}{m_i},
\intertext{where $m_k:=2n_k$, and}
F_k^{(1)}(1;q)&=\frac{1}{(q;q)_{\infty}}
\sum_{\substack{n_1,\dots,n_k\geq 0 \\[1pt] m_1,\dots,m_k\geq 0}}
\frac{q^{\sum_{i=1}^k (n_i^2-n_im_i+m_i^2)}}{(q)_{n_1}}\,
\qbin{2n_k}{m_k} 
\prod_{i=1}^{k-1} \qbin{n_i}{n_{i+1}}\qbin{n_i-n_{i+1}+m_{i+1}}{m_i}.
\end{align*}
Instead we will prove the following stronger result.

\begin{proposition}\label{Prop_finiteform}
For a positive integer $k$ and nonnegative integers $n_0,m_0$,
\begin{align*}
&F_{n_0,m_0;k}^{(-1)}(z,q) \\
&\qquad=\sum_{\substack{n_1,\dots,n_k\geq 0 \\[1pt] m_1,\dots,m_{k-1}\geq 0}}
\frac{z^{n_1} q^{n_k^2}}{(q)_{m_0-n_1+m_1}}\, \qbin{n_0}{n_1}
\prod_{i=1}^{k-1} q^{n_i^2-n_im_i+m_i^2} 
\qbin{n_i}{n_{i+1}}\qbin{n_i-n_{i+1}+m_{i+1}}{m_i},
\intertext{where $m_k:=2n_k$, and}
&F_{n_0,m_0;k}^{(1)}(z,q) \\
&\qquad=\sum_{\substack{n_1,\dots,n_k\geq 0 \\[1pt] m_1,\dots,m_k\geq 0}}
\frac{z^{n_1}q^{\sum_{i=1}^k(n_i^2-n_im_i+m_i^2)}}
{(q)_{m_0-n_1+m_1}}\,\qbin{n_0}{n_1}\qbin{2n_k}{m_k} 
\prod_{i=1}^{k-1} \qbin{n_i}{n_{i+1}}\qbin{n_i-n_{i+1}+m_{i+1}}{m_i}.
\end{align*}
\end{proposition}

A crucial ingredient in the proof of Proposition~\ref{Prop_finiteform} 
is the next lemma.

\begin{lemma}\label{Lem_F-trafo}
For $k$ a positive integer, $m_0$ a nonnegative integer 
and $u=(u_1,\dots,u_{k+1})\in\mathbb{Z}^{k+1}$
such that
\begin{equation}\label{Eq_u-ineq}
u_1\leq u_2\leq \cdots\leq u_{k+1},
\end{equation}
define
\begin{equation}\label{Eq_F-def}
\mathcal{F}_{m_0;u}(q):=\sum_{m_1,\dots,m_k\geq 0}
\frac{q^{\sum_{i=1}^k m_i(m_i+u_i)}}{(q)_{m_k+u_{k+1}}} 
\prod_{i=1}^k \qbin{m_{i-1}}{m_i}.
\end{equation}
Then, for $\ell\in\{0,1,\dots,k\}$,
\begin{equation}\label{Eq_F-ell}
\mathcal{F}_{m_0;u}(q)=\sum_{m_1,\dots,m_k\geq 0}
\frac{q^{\sum_{i=1}^k m_i(m_i+u_i)}}
{(q)_{m_{\ell}+m_{\ell+1}+u_{\ell+1}}} 
\prod_{i=1}^{\ell} \qbin{m_{i-1}}{m_i}
\prod_{i=\ell+1}^k \qbin{m_{i+1}+u_{i+1}-u_i}{m_i},
\end{equation}
where $m_{k+1}:=0$.
\end{lemma}
Note that if $u_{\ell+1}$ is a negative integer then the summand of
\eqref{Eq_F-ell} vanishes unless $m_{\ell}+m_{\ell+1}\geq -u_{\ell+1}$.

\begin{proof}
We proceed by induction on $\ell$, with base case $\ell=k$ 
corresponding to \eqref{Eq_F-def}. 

For the induction step we begin by replacing $z\mapsto z/b$ in Heine's 
$\qhyp{2}{1}$ transformation \cite[Equation (III.2)]{GR04}
\[
\qHyp{2}{1}{a,b}{c}{q,z}=\frac{(c/b,bz)_{\infty}}{(c,z)_{\infty}}\,
\qHyp{2}{1}{abz/c,b}{bz}{q,\frac{c}{b}} \quad 
\text{ for $\abs{q},\abs{z},\abs{c/b}<1$},
\]
and then letting $b$ tend to infinity.
This yields the $\qhyp{1}{1}$ transformation
\[
\qHyp{1}{1}{a}{c}{q,z}=
\frac{(z)_{\infty}}{(c)_{\infty}}\,\qHyp{1}{1}{az/c}{z}{q,c}
\quad \text{ for $\abs{q}<1$},
\]
which we write in regularised form
\[
\sum_{k=0}^{\infty} (-z)^k q^{\binom{k}{2}}\,
\frac{(a)_k(cq^k)_{\infty}}{(q)_k} = 
\sum_{k=0}^{\infty} (-c)^k q^{\binom{k}{2}}\,
\frac{(az/c)_k(zq^k)_{\infty}}{(q)_k}.
\]
This allows us to specialise $(a,c,z)=(q^{-(n_2-p)},q^{n_1+1},q^{n_2+1})$,
where $n_1,n_2,p$ are arbitrary integers.
Imposing the conditions
\begin{equation}\label{Eq_n1n2p}
\min\{n_1,n_2\}\geq p,
\end{equation}
the resulting transformation can be expressed using $q$-binomial coefficients
as\footnote{The transformation \eqref{Eq_qbintrafo} also
holds if \eqref{Eq_n1n2p} is replaced by $p>n_1+n_2$.
Since both sides trivialise to zero we ignore this second case.}
\begin{equation}\label{Eq_qbintrafo}
\sum_{m\geq 0} \frac{q^{m(m+p)}}{(q)_{m+n_1}}\,\qbin{n_2-p}{m}=
\sum_{m\geq 0} \frac{q^{m(m+p)}}{(q)_{m+n_2}}\,\qbin{n_1-p}{m}.
\end{equation}

Now assume that \eqref{Eq_F-ell} holds for some fixed $1\leq\ell\leq k$.
Applying \eqref{Eq_qbintrafo} with
\[
(m,n_1,n_2,p)=(m_{\ell},m_{\ell+1}+r_{\ell+1},
m_{\ell-1}+u_{\ell},u_{\ell})
\]
results in
\begin{equation}\label{Eq_F-ellmineen}
\mathcal{F}_{m_0;u}(q)=\sum_{m_1,\dots,m_k\geq 0}
\frac{q^{\sum_{i=1}^k m_i(m_i+u_i)}}
{(q)_{m_{\ell-1}+m_{\ell}+u_{\ell}}} 
\prod_{i=1}^{\ell-1} \qbin{m_{i-1}}{m_i}
\prod_{i=\ell}^k \qbin{m_{i+1}+u_{i+1}-u_i}{m_i}.
\end{equation}
The conditions \eqref{Eq_n1n2p} translate to
\[
m_{\ell+1}+u_{\ell+1}\geq u_{\ell} \quad\text{and}\quad m_{\ell-1}\geq 0,
\]
which are both satisfied since $u_{\ell+1}\geq u_{\ell}$.
Since \eqref{Eq_F-ellmineen} is \eqref{Eq_F-ell} with $\ell$ replaced by 
$\ell-1$ our proof is done.
\end{proof}

Equipped with Lemma~\ref{Lem_F-trafo}, the proof of 
Proposition~\ref{Prop_finiteform} is not difficult.

\begin{proof}[Proof of Proposition~\ref{Prop_finiteform}]
Comparing \eqref{Eq_F1} with \eqref{Eq_F-def}, we have
\begin{equation}\label{Eq_Fn0m0}
F_{n_0,m_0;k}^{(a)}(z,q)=
\sum_{n_1\geq\cdots\geq n_k\geq 0} z^{n_1} \mathcal{F}_{m_0;u}(q)
\prod_{i=1}^k q^{n_i^2} \qbin{n_{i-1}}{n_i},
\end{equation}
where
\[
u=(-\sigma_1 n_1,\dots,-\sigma_k n_k,n_k)=(-n_1,\dots,-n_{k-1},-an_k,n_k).
\]
Since \eqref{Eq_u-ineq} holds for $n_1\geq\cdots\geq n_k\geq 0$,
we can apply Lemma~\ref{Lem_F-trafo} with $\ell=0$.
Therefore,
\[
\mathcal{F}_{m_0;u}(q)=\sum_{m_1,\dots,m_k\geq 0}
\frac{q^{\sum_{i=1}^k m_i(m_i-\sigma_i n_i)}}
{(q)_{m_0-\sigma_1 n_1+m_1}}\,\qbin{(a+1)n_k}{m_k}
\prod_{i=1}^{k-1} \qbin{n_i-\sigma_{i+1} n_{i+1}+m_{i+1}}{m_i},
\]
and thus
\begin{align*}
F_{n_0,m_0;k}^{(a)}(z,q)=
\sum_{\substack{n_1,\dots,n_k\geq 0 \\[1pt] m_1,\dots,m_k\geq 0}} \bigg( &
\frac{z^{n_1} q^{\sum_{i=1}^k (n_i^2-\sigma_i n_im_i+m_i^2)}}
{(q)_{m_0-\sigma_1 n_1+m_1}}\,
\qbin{(1+a)n_k}{m_k} \\ & \times
\prod_{i=1}^k \qbin{n_{i-1}}{n_i}
\prod_{i=1}^{k-1} \qbin{n_i-\sigma_{i+1} n_{i+1}+m_{i+1}}{m_i}\bigg).
\end{align*}
When $a=-1$ (so that $\sigma_k=-1$ and $\sigma_i=1$ for $1\leq i<k$)
the term $\qbin{(1+a)n_k}{m_k}$ in the summand forces $m_k=0$,
resulting in the first claim of the proposition.
When $a=1$ (so that $\sigma_i=1$ for all $1\leq i\leq k$) 
the second claim follows.
\end{proof}

\subsubsection{Proof of Theorems~\ref{Thm_threecases}, \ref{Thm_A2-three},
\ref{Thm_threecases-b} and \ref{Thm_A2-three-b}}
We apply the same method used to prove Theorem~\ref{Thm_A2-vac-bothmoduli}
to prove Theorems~\ref{Thm_threecases}, \ref{Thm_A2-three},
\ref{Thm_threecases-b} and \ref{Thm_A2-three-b}. 
In fact we will do more, and also give a conditional proof of 
Conjectures~\ref{Con_A2-two} and \ref{Con_A2-two-b}, 
assuming an identity that is missing from \cite{ASW99} but probably 
should have been in that paper.\footnote{The third author of 
\cite{ASW99} takes full responsibility for the omission and hopes to
prove Conjecture~\ref{Con_missing} in a future publication.}

Define $F_{n_0,m_0;k,s,t}^{(a)}(z,q)\in\mathbb{Q}(q)[z]$ as
\begin{equation}\label{Eq_F2}
F_{n_0,m_0;k,s,t}^{(a)}(z,q):=
\sum_{\substack{n_1,\dots,n_k\geq 0 \\[1pt] m_1,\dots, m_k\geq 0}}
\frac{z^{n_1}q^{\sum_{i=s}^k n_i+\sum_{i=t}^k m_i}}{(q)_{n_k+m_k+1}}
\prod_{i=1}^k q^{n_i^2-\sigma_i n_im_i+m_i^2}
\qbin{n_{i-1}}{n_i}\qbin{m_{i-1}}{m_i},
\end{equation}
where $k,s,t$ are positive integers such that $1\leq s,t\leq k+1$, 
$n_0,m_0$ are nonnegative integers, $a\in\{-1,1\}$ and the $\sigma_i$ 
for $1\leq i\leq k$ are again fixed as in \eqref{Eq_sigma-a}.
The following proposition, which complements
Proposition~\ref{Prop_finiteform}, lies at the heart of each of
the proofs given below.

\begin{proposition}\label{Prop_finiteform2}
Let $k,s,t$ be positive integers and $n_0,m_0$ nonnegative integers.
If $1\leq s\leq k+1$ and $1\leq t\leq k$, then
\begin{align}\label{Eq_mineen}
F_{n_0,m_0;k,s,t}^{(-1)}(z,q)=
\sum_{\substack{n_1,\dots,n_k\geq 0\\[1pt] m_1,\dots,m_{k-1}\geq 0}}\bigg( &
\frac{z^{n_1} q^{n_k^2+\sum_{i=s}^k n_i+\sum_{i=t}^{k-1} m_i}}
{(q)_{m_0-n_1+m_1+\delta_{t,1}}}\, \qbin{n_0}{n_1} \\
& \times\prod_{i=1}^{k-1} q^{n_i^2-n_im_i+m_i^2}
\qbin{n_i}{n_{i+1}}
\qbin{n_i-n_{i+1}+m_{i+1}+\delta_{i,t-1}}{m_i}\bigg), \notag
\end{align}
and if $1\leq s\leq k$, $t=k+1$ and $k\geq 2$, then
\begin{align}\label{Eq_mineen2}
F_{n_0,m_0;k,s,k+1}^{(-1)}(z,q)=
\sum_{\substack{n_1,\dots,n_k\geq 0 \\[1pt] m_1,\dots,m_{k-1}\geq 0}}\bigg( &
\frac{z^{n_1} q^{n_k^2-n_k+\sum_{i=s}^k n_i}}{(q)_{m_0-n_1+m_1}}\,
\qbin{n_0}{n_1} \\ 
& \times\prod_{i=1}^{k-1} q^{n_i^2-n_im_i+m_i^2}
\qbin{n_i+\delta_{i,k-1}}{n_{i+1}}\qbin{n_i-n_{i+1}+m_{i+1}}{m_i}\bigg), 
\notag
\end{align}
where $m_k:=2n_k$ in both \eqref{Eq_mineen} and \eqref{Eq_mineen2}.
Similarly, if $1\leq s\leq k+1$ and $1\leq t\leq k+1$, then
\begin{align}\label{Eq_mineen3}
F_{n_0,m_0;k,s,t}^{(1)}(z,q)& \\
=\sum_{\substack{n_1,\dots,n_k\geq 0\\[1pt] m_1,\dots,m_k\geq 0}}\bigg( &
\frac{z^{n_1} q^{\sum_{i=1}^k (n_i^2-n_im_i+m_i^2)+\sum_{i=s}^k n_i+
\sum_{i=t}^k m_i}}{(q)_{m_0-n_1+m_1+\delta_{t,1}}} 
\qbin{2n_k+\delta_{t,k+1}}{m_k} \qbin{n_0}{n_1} \notag \\ 
& \times 
\prod_{i=1}^{k-1} \qbin{n_i}{n_{i+1}} 
\qbin{n_i-n_{i+1}+m_{i+1}+\delta_{i,t-1}}{m_i}\bigg). \notag
\end{align}
\end{proposition}

\begin{proof}
Recalling \eqref{Eq_F-def}, we have 
\begin{equation}\label{Eq_F4}
F_{n_0,m_0;k,s,t}^{(a)}(z,q)=\sum_{n_1\geq\cdots\geq n_k\geq 0}
z^{n_1} q^{\sum_{i=s}^k n_i} \mathcal{F}_{m_0;u}(q)
\prod_{i=1}^k q^{n_i^2} \qbin{n_{i-1}}{n_i},
\end{equation}
where $u\in\mathbb{Z}^{k+1}$ is given by
\[
u=\big(\underbrace{-\sigma_1n_1,\dots,-\sigma_{t-1}n_{t-1}}_
{t-1 \text{ terms}},
\underbrace{1-\sigma_t n_t,\dots,1-\sigma_k n_k}_{k-t+1 \text{ terms}},
1+n_k\big).
\]
Since for $n_1\geq\cdots\geq n_k\geq 0$ the inequalities \eqref{Eq_u-ineq}
hold, we may apply Lemma~\ref{Lem_F-trafo} with $\ell=0$.
Hence
\begin{align}\label{Eq_Fut}
F_{m_0;u}(z,q)=\sum_{m_1,\dots,m_k\geq 0} \bigg( &
\frac{q^{\sum_{i=1}^k m_i(m_i-\sigma_i n_i)+\sum_{i=t}^k m_i}}
{(q)_{m_0-\sigma_1 n_1+m_1+\delta_{t,1}}} \\
& \times
\qbin{(a+1)n_k+\delta_{t,k+1}}{m_k}
\prod_{i=1}^{k-1} \qbin{n_i-\sigma_i n_{i+1}+m_{i+1}+
\delta_{i,t-1}}{m_i}\bigg). \notag
\end{align}
If $a=-1$ and $1\leq t\leq k$ then the summand vanishes unless $m_k=0$.
Substituting the resulting expression for $F_{m_0;u}(q)$ into \eqref{Eq_F4}
yields \eqref{Eq_mineen}.
If $a=1$, so that $\sigma_i=1$ for all $i$, the substitution of
\eqref{Eq_Fut} into \eqref{Eq_F4} immediately gives \eqref{Eq_mineen3}.
The case requiring more work corresponds to $a=-1$ and $t=k+1$.
Then the summand of \eqref{Eq_Fut} vanishes unless $m_k=0$ or $m_k=1$, 
so that
\begin{align*}
F_{m_0;u}(z,q)&=\sum_{m_1,\dots,m_{k-1}\geq 0}
\frac{q^{\sum_{i=1}^{k-1} m_i(m_i-n_i)}}{(q)_{m_0-n_1+m_1}}
\prod_{i=1}^{k-1}\qbin{n_i-n_{i+1}+m_{i+1}}{m_i} \\
&\quad+\sum_{m_1,\dots,m_{k-1}\geq 0}
\frac{q^{n_k+1+\sum_{i=1}^{k-1} m_i(m_i-n_i)}}
{(q)_{m_0-n_1+m_1+\delta_{k,1}}}
\prod_{i=1}^{k-1}\qbin{n_i-n_{i+1}+m_{i+1}+\delta_{i,k-1}}{m_i},
\end{align*}
where now $m_k:=2n_k$.
Substituting this into \eqref{Eq_F4} gives
\begin{align*}
F_{n_0,m_0;k,s,k+1}^{(-1)}(z,q)&=
\sum_{\substack{n_1,\dots,n_k\geq 0 \\[1pt] m_1,\dots,m_{k-1}\geq 0}}
\bigg( \frac{z^{n_1} q^{n_k^2+\sum_{i=s}^k n_i}}{(q)_{m_0-n_1+m_1}}\,
\qbin{n_0}{n_1} \\ 
& \qquad\qquad\qquad\quad\times\prod_{i=1}^{k-1} q^{n_i^2-n_im_i+m_i^2}
\qbin{n_i}{n_{i+1}}\qbin{n_i-n_{i+1}+m_{i+1}}{m_i}\bigg) \\
&\quad+
\sum_{\substack{n_1,\dots,n_k\geq 0 \\[1pt] m_1,\dots,m_{k-1}\geq 0}}
\bigg(\frac{z^{n_1} q^{n_k^2+n_k+1+\sum_{i=s}^k n_i}} 
{(q)_{m_0-n_1+m_1+\delta_{k,1}}}\,\qbin{n_0}{n_1} \\
& \qquad\qquad\qquad\qquad\times
\prod_{i=1}^{k-1} q^{n_i^2-n_im_i+m_i^2}
\qbin{n_i}{n_{i+1}}\qbin{n_i-n_{i+1}+m_{i+1}+\delta_{i,k-1}}{m_i}\bigg).
\end{align*}
Assuming $1\leq s\leq k$ and $k\geq 2$, and replacing $n_k\mapsto n_k-1$
(so that also $m_k\mapsto m_k-2$) in the second multisum on the right,
this yields
\begin{align*}
F_{n_0,m_0;k,s,k+1}^{(-1)}(z,q)&=
\sum_{\substack{n_1,\dots,n_k\geq 0 \\[1pt] m_1,\dots,m_{k-1}\geq 0}}
\bigg( \frac{z^{n_1} q^{n_k^2+\sum_{i=s}^k n_i}}{(q)_{m_0-n_1+m_1}}\,
\bigg(\qbin{n_{k-1}}{n_k}+q^{-n_k} \qbin{n_{k-1}}{n_k-1}\bigg) \\ 
& \qquad\qquad\qquad\quad\times\prod_{i=1}^{k-1} q^{n_i^2-n_im_i+m_i^2}
\qbin{n_{i-1}}{n_i}\qbin{n_i-n_{i+1}+m_{i+1}}{m_i}\bigg) \\
&=\sum_{\substack{n_1,\dots,n_k\geq 0 \\[1pt] m_1,\dots,m_{k-1}\geq 0}}
\bigg( \frac{z^{n_1} q^{n_k^2-n_k+\sum_{i=s}^k n_i}}{(q)_{m_0-n_1+m_1}}\,
\qbin{n_{k-1}+1}{n_k} \\ 
& \qquad\qquad\qquad\quad\times\prod_{i=1}^{k-1} q^{n_i^2-n_im_i+m_i^2}
\qbin{n_{i-1}}{n_i}\qbin{n_i-n_{i+1}+m_{i+1}}{m_i}\bigg),
\end{align*}
where the final equality follows from \eqref{Eq_qbin-rec}
with $(n,m)=(n_k,n_{k-1}-n_k)$. 
\end{proof}

First we use Proposition~\ref{Prop_finiteform2} to prove
Theorems~\ref{Thm_threecases} and \ref{Thm_threecases-b},
and to give a conditional proof of Conjectures~\ref{Con_A2-two}
and \ref{Con_A2-two-b}.
Our starting point is the aforementioned identity missing from \cite{ASW99}.

\begin{conjecture}\label{Con_missing}
For integers $k,s$ such that $1\leq s\leq k+1$ and $a\in\{-1,1\}$,
let $\sigma_1,\dots,\sigma_k$ be fixed as in \eqref{Eq_sigma-a}.
Then
\begin{align}\label{Eq_missing}
\sum_{\substack{n_1,\dots,n_k\geq 0 \\[1pt] m_1,\dots, m_k\geq 0}}
& \frac{q^{\sum_{i=1}^k (n_i^2-\sigma_in_im_i+m_i^2+m_i)+\sum_{i=s}^k n_i}}
{(q)_{n_1}(q)_{m_1}(q)_{n_k+m_k+1}} 
\prod_{i=1}^{k-1} \qbin{n_i}{n_{i+1}}\qbin{m_i}{m_{i+1}} \\
&=\frac{(q^{3k+a+3};q^{3k+a+3})_{\infty}^2}{(q)_{\infty}^3}\,
\theta(q,q^s,q^{s+1};q^{3k+a+3}).  \notag
\end{align}
\end{conjecture}

\begin{proposition}\label{Prop_missing}
Conjecture~\ref{Con_missing} holds for $s\in\{1,k,k+1\}$.
\end{proposition}

\begin{proof}
Let $a=-1$. 
Then the $s=1$ case is \cite[Equation (5.28), $i=1$]{ASW99},
the $s=k$ case is \cite[Equation (5.29), $\sigma=1$]{ASW99} and
the $s=k+1$ case is \cite[Equation (5.29), $\sigma=0$]{ASW99}.
Next let $a=1$.
Then the $s=1$ case is \cite[Equation (5.22), $i=1$]{ASW99},
the $s=k$ case is \cite[Equation (5.23), $\sigma=1$]{ASW99} and
the $s=k+1$ case is \cite[Equation (5.23), $\sigma=0$]{ASW99}.
\end{proof}

If we define
\[
F^{(a)}_{k,s,t}(z;q):=\lim_{n_0,m_0\to\infty}
F^{(a)}_{n_0,m_0;k,s,t}(z;q),
\]
then \eqref{Eq_missing} can be stated succinctly as
\begin{equation}\label{Eq_Fks1}
F_{k,s,1}^{(a)}(1,q)
=\frac{(q^{3k+a+3};q^{3k+a+3})_{\infty}^2}{(q)_{\infty}^3}\,
\theta(q,q^s,q^{s+1};q^{3k+a+3}).
\end{equation}
Taking the large-$n_0,m_0$ limit of
\eqref{Eq_mineen} and \eqref{Eq_mineen3} for $z=t=1$ results in
Conjecture~\ref{Con_A2-two} and Conjecture~\ref{Con_A2-two-b} 
(with $k\mapsto k+1$) respectively.
Assuming $s\in\{1,k,k+1\}$ this proves 
Theorems~\ref{Thm_threecases} and \ref{Thm_threecases-b}.

Next we apply Proposition~\ref{Prop_finiteform2} to 
also prove Theorems~\ref{Thm_A2-three} and \ref{Thm_A2-three-b}.
First we note that according to \cite[Theorem 5.5]{ASW99} and 
\cite[Theorem 5.7]{ASW99} (with $k\mapsto k+1$)
\[
F_{k,s,s}^{(a)}(1;q)
=\frac{(q^{3k+a+3};q^{3k+a+3})_{\infty}^2}{(q)_{\infty}^3}
\sum_{i=1}^s q^{s-i}\,\theta(q^i,q^i,q^{2i};q^{3k+a+3})
\]
for integers $k,s$ such that $1\leq s\leq k$.
Hence, taking the large-$n_0,m_0$ limit of \eqref{Eq_mineen} and
\eqref{Eq_mineen3} for $z=1$ and $t=s$, yields Theorem~\ref{Thm_A2-three}
and Theorem~\ref{Thm_A2-three-b} (with $k\mapsto k+1$).

\subsubsection{Proof of \eqref{Eq_ks-alt}, \eqref{Eq_kk1-alt} and
\eqref{Eq_ks-alt-b}}
In this (sub)section we prove (or conditionally prove) the identities
\eqref{Eq_ks-alt}, \eqref{Eq_kk1-alt} and \eqref{Eq_ks-alt-b}.
As alluded to in Remark~\ref{Rem_alt-sums}, the idea is it to take
advantage of the near-symmetry exhibited by \eqref{Eq_missing}.
In particular, we exploit the fact that from \eqref{Eq_F2} it follows that
\begin{equation}\label{Eq_F-symmetry}
F_{n_0,m_0;k,s,t}^{(a)}(1,q)=F_{m_0,n_0;k,t,s}^{(a)}(1,q)
\end{equation}
but that this symmetry is not at all manifest in
Proposition~\ref{Prop_finiteform2}.
In fact, we only require \eqref{Eq_F-symmetry} in the limit of large $n_0$
and $m_0$:
\[
F_{k,s,t}^{(a)}(1,q)=F_{k,t,s}^{(a)}(1,q).
\]
By \eqref{Eq_Fks1} this implies
\begin{equation}\label{Eq_Fk1s}
F_{k,1,s}^{(a)}(1,q)
=\frac{(q^{3k+a+3};q^{3k+a+3})_{\infty}^2}{(q)_{\infty}^3}\,
\theta(q,q^s,q^{s+1};q^{3k+a+3}).
\end{equation}
Letting $n_0,m_0$ tend to infinity in the $(s,t,z)\mapsto (1,s,1)$ 
case of \eqref{Eq_mineen} and equating the resulting multisum for
$F_{k,1,s}^{(-1)}(1,q)$ with \eqref{Eq_Fk1s} yields \eqref{Eq_ks-alt}.
Similarly, letting $n_0,m_0$ tend to infinity in the $(s,z)=(1,1)$ case
of \eqref{Eq_mineen2} and equating the resulting multisum for 
$F_{k,1,k+1}^{(-1)}(1,q)$ with \eqref{Eq_Fk1s} yields \eqref{Eq_kk1-alt}
for $k\geq 2$.
The missing $k=1$ case of \eqref{Eq_kk1-alt} simply corresponds to
\eqref{Eq_ks-con} for $(k,s)=(1,2)$.
Finally, letting $n_0,m_0$ tend to infinity in the $(s,t,z)\mapsto (1,s,1)$
case of \eqref{Eq_mineen3} and equating the resulting multisum for
$F_{k,1,s}^{(1)}(1,q)$ with \eqref{Eq_Fk1s} yields \eqref{Eq_ks-alt-b}.

\subsection{Proof of Theorem~\ref{Thm_k12} for $k=1$}\label{Sec_keen}

To shorten some of our expressions, throughout this section
we write $\mathscr{C}_{\la/\mu/d}$, $\GK_{\la/\mu/d}(z,q)$ 
and $\GK_{\la/\mu/d;n}(q)$ 
instead of $\mathscr{C}_{\la/\mu/d}(\mathbb{N}_0)$,
$\GK_{\la/\mu/d}(z,q;\mathbb{N}_0)$ and
$\GK_{\la/\mu/d;n}(q;\mathbb{N}_0)$.

\medskip

The $k=1$ case of Theorem~\ref{Thm_k12} is given by
\begin{equation}\label{Eq_RR-rank3-unbounded}
\GK_{(a+1,1-a,0)}(z,q)=\frac{1}{(zq)_{\infty}}\,
\sum_{n=0}^{\infty} \frac{z^n q^{n(n+a)}}{(q)_n}
\end{equation}
for $a\in\{0,1\}$.
This is an immediate consequence of the level-rank duality 
\eqref{Eq_RR-rank3} and the rank-$2$ identity \cite[page 6]{CDU20}
\begin{equation}\label{Eq_RR-rank2-unbounded}
\GK_{(a+2,1-a)}(z,q)=\frac{1}{(zq)_{\infty}}\,
\sum_{n=0}^{\infty} \frac{z^n q^{n(n+a)}}{(q)_n}.
\end{equation}
Below we present an alternative proof based on a bounded analogue
of \eqref{Eq_RR-rank3-unbounded}.

\begin{proposition}\label{Prop_RRcase}
For $L$ a nonnegative integer and $a\in\{0,1\}$,
\begin{equation}\label{Eq_RR-rank3-bounded}
\GK_{(L+1,L,L)/(1-a,0,0)/2}(z,q)= 
\frac{1}{(zq)_{3L+a}}\,\sum_{n=0}^{\infty} z^n q^{n(n+a)} 
\qbin{3L-n}{n}.
\end{equation}
\end{proposition}

Letting $L$ tend to infinity using \eqref{Eq_Gc} yields 
\eqref{Eq_RR-rank3-unbounded}.
We remark that it is an open problem to find the
bounded analogues of Theorem~\ref{Thm_k12} for $k=2$ 
and (which should be easier) Conjecture~\ref{Con_cylindric-b} for 
$k=2$, i.e., \cite[Theorem 3.2]{CW19}.

Our first step in proving Proposition~\ref{Prop_RRcase} is to again use
level-rank duality; \eqref{Eq_RR-rank3-bounded} is implied by the
following rank-$2$ identity.

\begin{proposition}\label{Prop_RRcase-rank2}
For $L$ a nonnegative integer and $a,b\in\{0,1\}$,
\begin{equation}\label{Eq_GK1}
\GK_{(L+b+1,L)/(1-a,0)/3}(z,q)=\frac{1}{(zq)_{2L+a+b}}\,
\sum_{n=0}^{\infty} z^n q^{n(n+a)} \qbin{2L+b-n}{n}.
\end{equation}
\end{proposition}

Replacing $L\mapsto 3L+b$ in \eqref{Eq_GK1} and applying \eqref{Eq_levelrank}
with $r=2$, $d=3$, $\la=(2b+1,b)$, $\la'=(b+1,b,b)$, $\mu=(1-a,0)$ and
$\mu'=(1-a,0,0)$ yields \eqref{Eq_RR-rank3-bounded}
with $L\mapsto 2L+b$.

\begin{remark}
More generally it may be shown that for integers $b,L,k,s$ such that 
$b\in\{0,1\}$, $L\geq 0$ and $1\leq s\leq k$,
\begin{align}\label{Eq_GK-rank2}
&\GK_{(L+b+k-1,L)/(s-1,0)/2k-1}(z,q) \\
&\quad=\frac{1}{(zq)_N}\,\sum_{n_1,\dots,n_{k-1}=0}^{\infty}
z^{n_1} q^{\sum_{i=s}^{k-1}n_i}\prod_{i=1}^{k-1} q^{n_i^2} 
\qbin{N-n_i-n_{i+1}-\varphi_{i-s+1}-2\sum_{j=1}^{i-1}n_j}{n_i-n_{i+1}},
\notag
\end{align}
where $n_k:=0$, $N:=2L+b+k-s$ and $\varphi_i:=\max\{0,i\}$.
Using \eqref{Eq_GF-max}, \eqref{Eq_GK-thm3} and
\[
\sum_{m=0}^{n_0} [z^m] \bigg( \sum_{n_1=0}^{\infty} \frac{z^{n_1}}{(zq)_N}\,
f_{n_1}(q)\bigg)=\sum_{n_1=0}^n \qbin{N+n_0-n_1}{n_0-n_1} f_{n_1}(q),
\]
it follows from \eqref{Eq_GK-rank2} that 
\begin{align*}
&\GK_{(L+b+k-1,L)/(s-1,0)/2k-1;n_0}(q) \\[1mm]
&\quad=\sum_{j=-n_0}^{n_0} \bigg((-1)^j 
q^{(2k+1)\binom{j}{2}+js}
\qbin{n_0+\floor{\frac{1}{2}(N+(2k-1)j-k+s)}}{n_0-j} \\
& \qquad\qquad\qquad\qquad\qquad\qquad\qquad\times
\qbin{n_0+\ceil{\frac{1}{2}(N-(2k-1)j+k-s)}}{n_0+j}\bigg) \\
&\quad=\sum_{n_1,\dots,n_{k-1}=0}^{\infty}
q^{\sum_{i=1}^{k-1} n_i^2+\sum_{i=s}^{k-1} n_i} 
\prod_{i=0}^{k-1} \qbin{2n_0+N-n_i-n_{i+1}-\varphi_{i-s+1}
-2\sum_{j=0}^{i-1}n_j}{n_i-n_{i+1}}.
\end{align*}
For $k=1$ and even $N$, this is equivalent to Burge's doubly-bounded
analogue of Euler's pentagonal number theorem \cite[page~216]{Burge93}.
Similarly, the case of odd $N$ is equivalent to Burge's second analogue
of the pentagonal number theorem \cite[page~221]{Burge93}.
For $k=2$ and even $N$, the above corresponds to identities for
the generating functions $P(N,M,2,3,1,1)$ and $P(N,M,1,4,1,1)$ 
of partition pairs (Burge's notation and terminology) \cite[page 218]{Burge93},
which may be viewed as doubly-bounded analogues of the 
Rogers--Ramanujan identities \eqref{Eq_RR}.
For $k\geq 3$ the above result corresponds to doubly-bounded analogues
of the Andrews--Gordon identities, which for even $N$ are implicit
in \cite{FLW98}.

As an even level counterpart of \eqref{Eq_GK-rank2}, it may further be shown
that for integers $L,k,s$ such that $L\geq 0$ and $1\leq s\leq k+1$,
\begin{multline}\label{Eq_GK-rank2-even}
\GK_{(L+k,L)/(s-1,0)/2k}(z,q)
=\frac{1}{(zq)_N}\,\sum_{n_1,\dots,n_k=0}^{\infty}\bigg(
z^{n_1} q^{\sum_{i=1}^k n_i^2+\sum_{i=s}^kn_i} 
\qbin{L-\sum_{j=1}^{k-1}n_j}{n_k}_{q^2} \\
\times
\prod_{i=1}^{k-1} 
\qbin{N-n_i-n_{i+1}-\varphi_{i-s+1}-2\sum_{j=1}^{i-1}n_j}{n_i-n_{i+1}}\bigg),
\end{multline}
where, now, $N:=2L+k-s+1$.
As per the above, this implies
\begin{align*}
&\GK_{(L+k,L)/(s-1,0)/2k;n_0}(q) \\[1mm]
&\quad=\sum_{j=-n_0}^{n_0} (-1)^j 
q^{2(k+1)\binom{j}{2}+js}
\qbin{n_0+kj+\frac{1}{2}(N-k+s-1)}{n_0-j} 
\qbin{n_0-kj+\frac{1}{2}(N+k-s+1)}{n_0+j} \\
&\quad=\sum_{n_1,\dots,n_k=0}^{\infty}\bigg(
q^{\sum_{i=1}^k n_i^2+\sum_{i=s}^kn_i} 
\prod_{i=0}^{k-1} 
\qbin{N+2n_0-n_i-n_{i+1}-\varphi_{i-s+1}-2\sum_{j=0}^{i-1}n_j}{n_i-n_{i+1}} \\
&\qquad\qquad\qquad\qquad\qquad\qquad\quad \times
\qbin{n_0+\frac{1}{2}(N-k+s-1)-\sum_{j=0}^{k-1}n_j}{n_k}_{q^2}\bigg).
\end{align*}
For $k=0$ (and hence $s=1$ and $N$ even) this is equivalent to an 
identity for $P(N,M,1,1,1,1)$ due to Burge \cite[page 217]{Burge93}.
For $k=1$ and $s=2$ (and hence $N$ even) the above corresponds to
Burge's identity for $P(N,M,2,2,1,1)$ \cite[page 218]{Burge93}, and
for $k\geq 2$ the above gives doubly-bounded analogues of Bressoud's
Rogers--Ramanujan-type identities for even moduli
\cite{Bressoud80,Bressoud80b}.
\end{remark}

\begin{proof}[Proof of Proposition~\ref{Prop_RRcase-rank2}]
Our proof is a bounded version of the proof of \eqref{Eq_RR-rank2-unbounded}
given in \cite{CDU20}. 
The latter is based on a system of functional equations for $\GK_c(z,q)$ 
with fixed level $d:=c_0+\cdots+c_{r-1}$, due to Corteel and 
Welsh \cite[Proposition 3.1]{CW19}.
Below we consider a bounded version of these functional equations,
and although much of what we do can be generalised to arbitrary shapes
$\la/\mu/d$, we restrict considerations to the rank-$2$ case, which is the
only case we know how to solve.

Let 
\[
\mathscr{C}_{\la/\mu/d}(m):=
\big\{\pi\in\mathscr{C}_{\la/\mu/d}: \max(\pi)=m\}
\]
and
\[
\GK_{\la/\mu/d;m}(q):=\sum_{\pi\in\mathscr{C}_{\la/\mu/d}(m)} q^{\abs{\pi}}.
\]
Then
\[
\GK_{\la/\mu/d}(z,q)=\sum_{m=0}^{\infty} z^m \GK_{\la/\mu/d;m}(q).
\]

For integers $d,i,j$ such that $0\leq j\leq i\leq d+j$ and
$d,i\geq 1$, let $\pi\in\mathscr{C}_{(i,j)/(1,0)/d}(m)$.
By adding a box with filling $n\geq m$ in the first row of $\pi$ as in
\smallskip 

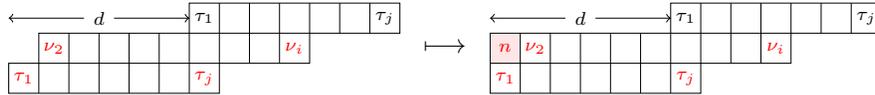
\begin{center}
\begin{tikzpicture}[scale=0.4,line width=0.3pt]
\draw (0,0)--(7,0)--(7,1)--(10,1)--(10,2)--(13,2)--(13,3)--(6,3)--(6,2)--(1,2)--(1,1)--(0,1)--cycle;
\foreach \i in {1,...,1}{\draw (\i,0)--(\i,1);}
\foreach \i in {2,...,6}{\draw (\i,0)--(\i,2);}
\foreach \i in {7,...,9}{\draw (\i,1)--(\i,3);}
\foreach \i in {10,...,12}{\draw (\i,2)--(\i,3);}
\draw (1,1)--(7,1);
\draw (6,2)--(10,2);
\draw[red] (1.5,1.5) node {$\sc \nu_2$};
\draw[red] (9.5,1.5) node {$\sc \nu_i$};
\draw[red] (0.5,0.5) node {$\sc \tau_1$};
\draw[red] (6.5,0.5) node {$\sc \tau_j$};
\draw (6.5,2.5) node {$\sc \tau_1$};
\draw (12.5,2.5) node {$\sc \tau_j$};
\draw (3,2.5) node {$\sc d$};
\draw[<-] (0,2.5)--(2.5,2.5);
\draw[->] (3.5,2.5)--(6,2.5);
\draw (14.5,1.5) node {$\longmapsto$};
\begin{scope}[xshift=16cm]
\filldraw[white,fill=red!10,very thin] (0,1) rectangle (1,2);
\draw (0,0)--(7,0)--(7,1)--(10,1)--(10,2)--(13,2)--(13,3)--(6,3)--(6,2)--(0,2)--cycle;
\foreach \i in {1,...,6}{\draw (\i,0)--(\i,2);}
\foreach \i in {7,...,9}{\draw (\i,1)--(\i,3);}
\foreach \i in {10,...,12}{\draw (\i,2)--(\i,3);}
\draw (0,1)--(7,1);
\draw (6,2)--(10,2);
\draw[red] (0.5,1.5) node {$\sc n$};
\draw[red] (1.5,1.5) node {$\sc \nu_2$};
\draw[red] (9.5,1.5) node {$\sc \nu_i$};
\draw[red] (0.5,0.5) node {$\sc \tau_1$};
\draw[red] (6.5,0.5) node {$\sc \tau_j$};
\draw (6.5,2.5) node {$\sc \tau_1$};
\draw (12.5,2.5) node {$\sc \tau_j$};
\draw (3,2.5) node {$\sc d$};
\draw[<-] (0,2.5)--(2.5,2.5);
\draw[->] (3.5,2.5)--(6,2.5);
\end{scope}
\end{tikzpicture}
\end{center}

\noindent
(where $m=\max\{\nu_2,\tau_1\}$)
we obtain a cylindric partition in $\mathscr{C}_{(i,j)/(0,0)/d}(n)$.
Since every cylindric partition in $\mathscr{C}_{(i,j)/(0,0)/d}(n)$ 
is uniquely obtained in this way from a cylindric partition in 
$\mathscr{C}_{(i,j)/(1,0)/d}(m)$ for some $m$, it follows that
\[
\GK_{(i,j)/(0,0)/d;n}(q)
=q^n \sum_{m=0}^n \GK_{(i,j)/(1,0)/d;m}(q).
\]
Multiplying both sides by $z^n$ and summing over $n$, this gives
\begin{equation}\label{Eq_FE-1}
\GK_{(i,j)/(0,0)/d}(z,q)=\frac{\GK_{(i,j)/(1,0)/d}(zq;q)}{1-zq}.
\end{equation}

Next, for integers $d,i,j,s$ such that $0\leq j\leq i\leq d+j$ and
$d,i\geq s\geq 2$, let $\pi\in\mathscr{C}_{(i,j)/(s,0)/d}(m)$.
Adding a box with filling $n$ in the first row of $\pi$ as in
\begin{equation}\label{Eq_fig1}
\raisebox{-5mm}{
\begin{tikzpicture}[scale=0.4,line width=0.3pt]
\draw (0,0)--(7,0)--(7,1)--(10,1)--(10,2)--(13,2)--(13,3)--(6,3)--(6,2)--(4,2)--(4,1)--(0,1)--cycle;
\foreach \i in {1,...,4}{\draw (\i,0)--(\i,1);}
\foreach \i in {5,...,6}{\draw (\i,0)--(\i,2);}
\foreach \i in {7,...,9}{\draw (\i,1)--(\i,3);}
\foreach \i in {10,...,12}{\draw (\i,2)--(\i,3);}
\draw (4,1)--(7,1);
\draw (6,2)--(10,2);
\draw[red] (4.5,1.5) node {$\hspace{1pt}\sc \nu_{\hspace{-1pt} s'}$};
\draw[red] (9.5,1.5) node {$\sc \nu_i$};
\draw[red] (0.5,0.5) node {$\sc \tau_1$};
\draw[red] (3.5,0.5) node {$\sc \tau_s$};
\draw[red] (6.5,0.5) node {$\sc \tau_j$};
\draw (6.5,2.5) node {$\sc \tau_1$};
\draw (12.5,2.5) node {$\sc \tau_j$};
\draw (3,2.5) node {$\sc d$};
\draw[<-] (0,2.5)--(2.5,2.5);
\draw[->] (3.5,2.5)--(6,2.5);
\draw (14.5,1.5) node {$\longmapsto$};
\begin{scope}[xshift=16cm]
\filldraw[white,fill=red!10,very thin] (3,1) rectangle (4,2);
\draw (0,0)--(7,0)--(7,1)--(10,1)--(10,2)--(13,2)--(13,3)--(6,3)--(6,2)--(3,2)--(3,1)--(0,1)--cycle;
\foreach \i in {1,...,3}{\draw (\i,0)--(\i,1);}
\foreach \i in {4,...,6}{\draw (\i,0)--(\i,2);}
\foreach \i in {7,...,9}{\draw (\i,1)--(\i,3);}
\foreach \i in {10,...,12}{\draw (\i,2)--(\i,3);}
\draw (3,1)--(7,1);
\draw (6,2)--(10,2);
\draw[red] (3.5,1.5) node {$\sc n$};
\draw[red] (4.5,1.5) node {$\hspace{1pt}\sc \nu_{\hspace{-1pt}s'}$};
\draw[red] (9.5,1.5) node {$\sc \nu_i$};
\draw[red] (0.5,0.5) node {$\sc \tau_1$};
\draw[red] (3.5,0.5) node {$\sc \tau_s$};
\draw[red] (6.5,0.5) node {$\sc \tau_j$};
\draw (6.5,2.5) node {$\sc \tau_1$};
\draw (12.5,2.5) node {$\sc \tau_j$};
\draw (3,2.5) node {$\sc d$};
\draw[<-] (0,2.5)--(2.5,2.5);
\draw[->] (3.5,2.5)--(6,2.5);
\end{scope}
\end{tikzpicture}}
\end{equation}

\noindent
where $s':=s+1$ and $m=\max\{\nu_{s'},\tau_1\}$,
yields a cylindric partition in $\mathscr{C}_{(i,j)/(s-1,0)/d}(n)$.
Similarly, adding a box with filling $n$ in the second row 
of a cylindric partition in
$\mathscr{C}_{(i,j)/(s-1,1)/d}(m)$ (with $j\geq 1$) as in
\begin{equation}\label{Eq_fig2}
\raisebox{-5mm}{
\begin{tikzpicture}[scale=0.4,line width=0.3pt]
\draw[white](0,0)--(1,0);
\draw (1,0)--(7,0)--(7,1)--(10,1)--(10,2)--(13,2)--(13,3)--(7,3)--(7,2)--(3,2)--(3,1)--(1,1)--cycle;
\foreach \i in {2,...,3}{\draw (\i,0)--(\i,1);}
\foreach \i in {4,...,6}{\draw (\i,0)--(\i,2);}
\foreach \i in {7,...,9}{\draw (\i,1)--(\i,3);}
\foreach \i in {10,...,12}{\draw (\i,2)--(\i,3);}
\draw (3,1)--(7,1);
\draw (7,2)--(10,2);
\draw[red] (3.5,1.5) node {$\sc \nu_s$};
\draw[red] (9.5,1.5) node {$\sc \nu_i$};
\draw[red] (1.5,0.5) node {$\sc \tau_2$};
\draw[red] (3.5,0.5) node {$\sc \tau_s$};
\draw[red] (6.5,0.5) node {$\sc \tau_j$};
\draw (7.5,2.5) node {$\sc \tau_2$};
\draw (12.5,2.5) node {$\sc \tau_j$};
\draw (4,2.5) node {$\sc d$};
\draw[<-] (1,2.5)--(3.5,2.5);
\draw[->] (4.5,2.5)--(7,2.5);
\draw (14.5,1.5) node {$\longmapsto$};
\begin{scope}[xshift=16cm]
\filldraw[white,fill=red!10,very thin] (0,0) rectangle (1,1);
\draw (0,0)--(7,0)--(7,1)--(10,1)--(10,2)--(13,2)--(13,3)--(6,3)--(6,2)--(3,2)--(3,1)--(0,1)--cycle;
\foreach \i in {1,...,3}{\draw (\i,0)--(\i,1);}
\foreach \i in {4,...,6}{\draw (\i,0)--(\i,2);}
\foreach \i in {7,...,9}{\draw (\i,1)--(\i,3);}
\foreach \i in {10,...,12}{\draw (\i,2)--(\i,3);}
\draw (3,1)--(7,1);
\draw (6,2)--(10,2);
\draw[red] (3.5,1.5) node {$\sc \nu_s$};
\draw[red] (9.5,1.5) node {$\sc \nu_i$};
\draw[red] (0.5,0.5) node {$\sc n$};
\draw[red] (1.5,0.5) node {$\sc \tau_2$};
\draw[red] (3.5,0.5) node {$\sc \tau_s$};
\draw[red] (6.5,0.5) node {$\sc \tau_j$};
\draw (6.5,2.5) node {$\sc n$};
\draw (7.5,2.5) node {$\sc \tau_2$};
\draw (12.5,2.5) node {$\sc \tau_j$};
\draw (3,2.5) node {$\sc d$};
\draw[<-] (0,2.5)--(2.5,2.5);
\draw[->] (3.5,2.5)--(6,2.5);
\end{scope}
\end{tikzpicture}}
\end{equation}

\noindent
(where $m=\nu_s$) 
once again yields a cylindric partition in 
$\pi\in\mathscr{C}_{(i,j)/(s-1,0)/d}(n)$.
Clearly, those cylindric partitions in 
$\mathscr{C}_{(i,j)/(s-1,0)/d}(n)$ of the form
\begin{equation}\label{Eq_fig3}
\raisebox{-5mm}{
\begin{tikzpicture}[scale=0.4,line width=0.3pt]
\filldraw[white,fill=red!10,very thin] (0,0) rectangle (1,1);
\filldraw[white,fill=red!10,very thin] (3,1) rectangle (4,2);
\draw (0,0)--(7,0)--(7,1)--(10,1)--(10,2)--(13,2)--(13,3)--(6,3)--(6,2)--(3,2)--(3,1)--(0,1)--cycle;
\foreach \i in {1,...,3}{\draw (\i,0)--(\i,1);}
\foreach \i in {4,...,6}{\draw (\i,0)--(\i,2);}
\foreach \i in {7,...,9}{\draw (\i,1)--(\i,3);}
\foreach \i in {10,...,12}{\draw (\i,2)--(\i,3);}
\draw (3,1)--(7,1);
\draw (6,2)--(10,2);
\draw[red] (3.5,1.5) node {$\sc n$};
\draw[red] (4.5,1.5) node {$\hspace{1pt}\sc \nu_{\hspace{-1pt}s'}$};
\draw[red] (9.5,1.5) node {$\sc \nu_i$};
\draw[red] (0.5,0.5) node {$\sc n$};
\draw[red] (1.5,0.5) node {$\sc \tau_2$};
\draw[red] (3.5,0.5) node {$\sc \tau_s$};
\draw[red] (6.5,0.5) node {$\sc \tau_j$};
\draw (6.5,2.5) node {$\sc n$};
\draw (7.5,2.5) node {$\sc \tau_2$};
\draw (12.5,2.5) node {$\sc \tau_j$};
\draw (3,2.5) node {$\sc d$};
\draw[<-] (0,2.5)--(2.5,2.5);
\draw[->] (3.5,2.5)--(6,2.5);
\end{tikzpicture}}
\end{equation}
will be generated twice (take $\tau_1=n$ in \eqref{Eq_fig1} and 
$\nu_s=n$ in \eqref{Eq_fig2}), and a correction is required to
avoid double counting.
Hence
\begin{align}\label{Eq_case2}
\GK_{(i,j)/(s-1,0)/d;n}(q)
&=q^n \sum_{m=0}^n \GK_{(i,j)/(s,0)/d;m}(q)+
q^n \sum_{m=0}^n \GK_{(i,j)/(s-1,1)/d;m}(q) \\ & \quad -
q^{2n} \sum_{m=0}^n \GK_{(i,j)/(s,1)/d;m}(q). \notag
\end{align}
Indeed, by adding a box with filling $n\geq m$ to both rows of a cylindric
partition

\smallskip

\begin{center}
\begin{tikzpicture}[scale=0.4,line width=0.3pt]
\draw (1,0)--(7,0)--(7,1)--(10,1)--(10,2)--(13,2)--(13,3)--(7,3)--(7,2)--(4,2)--(4,1)--(1,1)--cycle;
\foreach \i in {2,...,4}{\draw (\i,0)--(\i,1);}
\foreach \i in {5,...,6}{\draw (\i,0)--(\i,2);}
\foreach \i in {7,...,7}{\draw (\i,1)--(\i,2);}
\foreach \i in {8,...,9}{\draw (\i,1)--(\i,3);}
\foreach \i in {10,...,12}{\draw (\i,2)--(\i,3);}
\draw (3,1)--(7,1);
\draw (6,2)--(10,2);
\draw[red] (4.5,1.5) node {$\hspace{1pt}\sc \nu_{\hspace{-1pt}s'}$};
\draw[red] (9.5,1.5) node {$\sc \nu_i$};
\draw[red] (1.5,0.5) node {$\sc \tau_2$};
\draw[red] (3.5,0.5) node {$\sc \tau_s$};
\draw[red] (6.5,0.5) node {$\sc \tau_j$};
\draw (7.5,2.5) node {$\sc \tau_2$};
\draw (12.5,2.5) node {$\sc \tau_j$};
\draw (4,2.5) node {$\sc d$};
\draw[<-] (1,2.5)--(3.5,2.5);
\draw[->] (4.5,2.5)--(7,2.5);
\draw (18,1.5) node {($m:=\max\{\nu_{s'},\tau_s\}$)};
\end{tikzpicture}
\end{center}

\noindent
in $\mathscr{C}_{(i,j)/(s,1)/d}(m)$
results in a cylindric partition in $\mathscr{C}_{(i,j)/(s-1,0)/d}(n)$ 
of the form \eqref{Eq_fig3}.
Multiplying both sides of \eqref{Eq_case2} by $z^n$ and summing
over $n$ gives
\begin{align*}
\GK_{(i,j)/(s-1,0)/d}(z,q)
&=\frac{\GK_{(i,j)/(s,0)/d}(zq,q)}{1-zq} \\ 
&\quad + \frac{\GK_{(i,j)/(s-1,1)/d}(zq,q)}{1-zq}
- \frac{\GK_{(i,j)/(s,1)/d}(zq^2,q)}{1-zq^2}.
\end{align*}
Here we recall that by \eqref{Eq_zero} the last two terms on the right 
vanish unless $j\geq 1$.
Applying the translation symmetry \eqref{Eq_trans} to the second and 
third term on the right yields
\begin{align}\label{Eq_generals}
\GK_{(i,j)/(s-1,0)/d}(z,q)
&=\frac{\GK_{(i,j)/(s,0)/d}(zq,q)}{1-zq} \\ 
&\quad + \frac{\GK_{(i-1,j-1)/(s-2,0)/d}(zq,q)}{1-zq} 
- \frac{\GK_{(i-1,j-1)/(s-1,0)/d}(zq^2,q)}{1-zq^2}. \notag
\end{align}
The integer $s$ in \eqref{Eq_generals} is restricted to $2\leq s\leq d$, but by the
cyclic symmetry we can in fact limit the range of $s$ to 
$2\leq s\leq\floor{d/2}+1$. 
To this end we take \eqref{Eq_generals} with $s=\floor{d/2}+1$ and
apply the cyclic symmetry \eqref{Eq_cyc} followed by the
translation symmetry \eqref{Eq_trans} to the first term on the right.
This leads to
\begin{multline}\label{Eq_sboundary}
\GK_{(i,j)/(\floor{d/2},0)/d}(z,q)
=\frac{\GK_{(j+\ceil{d/2}-1,i-\floor{d/2}-1)/(\ceil{d/2}-1,0)/d}(zq,q)}{1-zq} \\
+\frac{\GK_{(i-1,j-1)/(\floor{d/2}-1,0)/d}(zq,q)}{1-zq} 
-\frac{\GK_{(i-1,j-1)/(\floor{d/2},0)/d}(zq^2,q)}{1-zq^2}.
\end{multline}
Depending on the parity of $d$, \eqref{Eq_FE-1}, \eqref{Eq_generals} for
$2\leq s\leq \floor{d/2}+1$ and \eqref{Eq_sboundary} yield a closed
system of equations. 
In particular, for $d=2k-1$ and $(i,j)=(L+k+b-1,L)$,
with $b\in\{0,1\}$ and $L$ a nonnegative integer, the following
system of equations (subject to the initial conditions 
$\GK_{(b+k-2,-1)/(s-1,0)/2k-1}(z,q)=0$ for $1\leq s\leq k$ and 
$\GK_{(k-1,0)/(k-1,0)/2k-1}(z,q)=1$) uniquely determines 
$\GK_{(L+k+b-1,L)/(s-1,0)/2k-1}(z,q)$ for all $b\in\{0,1\}$,
$1\leq s\leq k$ and $L\geq 0$:
\[
\GK_{(L+k+b-1,L)/(0,0)/2k-1}(z,q)=
\frac{\GK_{(L+k+b-1,L)/(1,0)/2k-1}(zq;q)}{1-zq}
\]
and
\begin{multline*}
\GK_{(L+k+b-1,L)/(s-1,0)/2k-1}(z,q)
=\frac{\GK_{(L+k+b-1,L)/(s,0)/2k-1}(q)(zq,q)}{1-zq} \\ 
+\frac{\GK_{(L+k+b-2,L-1)/(s-2,0)/2k-1}(zq,q)}{1-zq}
-\frac{\GK_{(L+k+b-2,L-1)/(s-1,0)/2k-1}(zq^2,q)}{1-zq^2},
\end{multline*}
for $2\leq s\leq k-1$, and
\begin{multline*}
\GK_{(L+k+b-1,L)/(k-1,0)/2k-1}(z,q)
=\frac{\GK_{(L+k-1,L+b-1)/(k-1,0)/2k-1}(zq,q)}{1-zq} \\
+\frac{\GK_{(L+k+b-2,L-1)/(k-2,0)/2k-1}(zq,q)}{1-zq} 
-\frac{\GK_{(L+k+b-2,L-1)/(k-1,0)/2k-1}(zq^2,q)}{1-zq^2},
\end{multline*}
where in the last equation $L\geq 1-b$.
One can show, using telescopic expansions for $q$-binomial
coefficients in the spirit of \cite{Berkovich94,Schilling96},
that the above system of equations is solved by \eqref{Eq_GK-rank2}.
Restricting to the $k=2$ case, we get the following bounded analogues of
\cite[Equations~(2.4)\&(2.5)]{CDU20}:
\begin{equation}\label{Eq_FE-simple}
\GK_{(L+b+1,L)/(0,0)/3}(z,q)=
\frac{\GK_{(L+b+1,L)/(1,0)/3}(zq;q)}{1-zq},
\end{equation}
for $L\geq 0$, and
\begin{align*}
\GK_{(L+b+1,L)/(1,0)/3}(z,q)
&=\frac{\GK_{(L+1,L+b-1)/(1,0)/3}(zq,q)}{1-zq}
+ \frac{\GK_{(L+b,L-1)/(0,0)/3}(zq,q)}{1-zq} \\
&\quad - \frac{\GK_{(L+b,L-1)/(1,0)/3}(zq^2,q)}{1-zq^2},
\end{align*}
for $L\geq 1-b$.
Using the first equation to eliminate 
$\GK_{(L+b,L-1)/(0,0)/3}(zq,q)$ from the second equation,
we further obtain the a bounded analogue of 
\cite[Equation~(2.6)]{CDU20}:
\begin{align}\label{Eq_Fun}
&\GK_{(L+b+1,L)/(1,0)/3}(z,q) \\
&\qquad=\frac{\GK_{(L+1,L+b-1)/(1,0)/3}(zq,q)}{1-zq}
+ \frac{zq\GK_{(L+b,L-1)/(1,0)/3}(zq^2,q)}{(1-zq)(1-zq^2)}. \notag
\end{align}
Together with the initial conditions
\begin{equation}\label{Eq_initial}
\GK_{(1,-1)/(1,0)/3}(z,q)=0
\quad\text{and}\quad
\GK_{(1,0)/(1,0)/3}(z,q)=1
\end{equation}
this uniquely determines
$\GK_{(L+b+1,L)/(1,0)/3}(z,q)$ for all $L\geq 0$ and $b\in\{0,1\}$.
Since for negative integers values of $L$ the right-hand side of
\eqref{Eq_GK1} trivially vanishes, 
the claim of Proposition~\ref{Prop_RRcase-rank2} holds for all integers $L$.
Assuming this extended range of $L$, 
\eqref{Eq_GK1} for $a=0$ satisfies the initial 
conditions \eqref{Eq_initial}.
Moreover, substituting the $a=0$ case of \eqref{Eq_GK1}
into \eqref{Eq_Fun} and clearing common denominators yields
(after replacing $2L+1-b$ by $m$)
\[
\sum_{n=0}^{\infty} z^n q^{n^2} \qbin{m-n+1}{n}=
\sum_{n=0}^{\infty} z^n q^{n^2+n} \qbin{m-n}{n}+
\sum_{n=0}^{\infty} z^{n+1} q^{(n+1)^2} \qbin{m-n-1}{n}
\]
for $m$ a nonnegative integer.
This polynomial identity is readily proved by substituting 
$n\mapsto n-1$ in the second sum on the right and by then applying 
the $q$-binomial recurrence
\[
q^n \qbin{m-n}{n}+\qbin{m-n}{n-1}=\qbin{m-n+1}{n}
\]
for $m,n\geq 0$.

To complete the proof of Proposition~\ref{Prop_RRcase-rank2}
we note that the $a=1$ case of the proposition is a direct
consequence of the $a=0$ case and \eqref{Eq_FE-simple}.

\medskip
For completeness we mention that the system of equations
\eqref{Eq_FE-1}, \eqref{Eq_generals} and \eqref{Eq_sboundary} for
$d=2k$ is closed for $(i,j)=(L+k,L)$ for $L$ a nonnegative integer, 
eliminating the need for the parameter $b$. 
In this even-$d$ case we thus get
\[
\GK_{(L+k,L)/(0,0)/2k}(z,q)=\frac{\GK_{(L+k,L)/(1,0)/2k}(zq;q)}{1-zq}
\]
and
\begin{multline*}
\GK_{(L+k,L)/(s-1,0)/2k}(z,q)
=\frac{\GK_{(L+k,L)/(s,0)/2k}(q)(zq,q)}{1-zq} \\ 
+ \frac{\GK_{(L+k-1,L-1)/(s-2,0)/2k}(zq,q)}{1-zq} 
- \frac{\GK_{(L+k-1,L-1)/(s-1,0)/2k}(zq^2,q)}{1-zq^2},
\end{multline*}
for $2\leq s\leq k$, and
\begin{multline*}
\GK_{(L+k,L)/(k,0)/2k}(z,q)
=\frac{\GK_{(L+k-1,L-1)/(k-1,0)/2k}(zq,q)}{1-zq} \\
+\frac{\GK_{(L+k-1,L-1)/(k-1,0)/2k}(zq,q)}{1-zq} 
-\frac{\GK_{(L+k-1,L-1)/(k,0)/d}(zq^2,q)}{1-zq^2},
\end{multline*}
where in the last equation $L\geq 1$.
It is again not difficult to show that this system of equations
is solved by \eqref{Eq_GK-rank2-even}.
\end{proof}

\section{Towards the Andrews--Gordon identities for $\mathrm{A}_{r-1}$}
\label{Sec_Ar}

An important open problem is to find the $\mathrm{A}_{r-1}$ 
(or $\mathrm{A}_{r-1}^{(1)}$) Rogers--Ramanujan and Andrews--Gordon
identities for arbitrary rank $r$.
By level-rank duality, a number of low-level results
for $\mathrm{A}_{r-1}$ are implied by those for $r=2$ and $r=3$.
Specifically, the classical or $\mathrm{A}_1$ Andrews--Gordon identities
\eqref{Eq_AG} for modulus $2k+1$ imply identities for 
$\mathrm{A}_{2k-2}$ at level $2$,
and the $\mathrm{A}_2$ Andrews--Gordon identities for modulus 
$3k+\sigma+1$ ($\sigma=0,1$)
imply identities for $\mathrm{A}_{3k+\sigma-3}$ at level $3$.
However, to be able to see the full $\mathrm{A}_{r-1}$-structure of the 
sum-side of an arbitrary-level $\mathrm{A}_{r-1}$ Andrews--Gordon identity
one needs at least $r-1$ summation variables, and none of the available 
low-level identities suffices.

At the opposite end of the spectrum are the $q$-series identities
corresponding to the infinite-level limit of the $\mathrm{A}_{r-1}$
Andrews--Gordon identities.
Before considering the case of arbitrary rank, we briefly discuss this
limit for $r=2$ and $r=3$.

In the case of $r=2$ we take the large-$k$ limit of the Andrews--Gordon
identity \eqref{Eq_AG} for $s=k$ and $s=1$.
The resulting pair of identities are the $z=1$ and $z=q$ instances of
\begin{equation}\label{Eq_A1-infinitelevel}
\sum_{n_1,n_2,\dots\geq 0}
\frac{1}{(q)_{n_1}}\, \prod_{i\geq 1} z^{n_i} q^{n_i^2} 
\qbin{n_i}{n_{i+1}}=\frac{1}{(zq)_{\infty}}.
\end{equation}
This is the $n_0\to\infty$ limit of the rational function identity
\begin{equation}\label{Eq_A1-rational}
\sum_{n_1,n_2,\ldots=0}^{n_0} \, \prod_{i\geq 1} z^{n_i} q^{n_i^2} 
\qbin{n_{i-1}}{n_i}=\frac{1}{(zq)_{n_0}},
\end{equation}
where $n_0$ is a nonnegative integer.
The easiest way to understand this last identity is 
to note that the terminating form of the $\qhyp{1}{1}$
summation \cite[Equation (II.5)]{GR04} may be written as
\begin{equation}\label{Eq_1phi1}
\sum_{k=0}^n \frac{a^k q^{k^2}}{(aq)_k}\,
\qbin{n}{k}=\frac{1}{(aq)_n}, \quad n\in\mathbb{N}_0.
\end{equation}
Hence, if
\[
\Phi_n(z;q):=\frac{1}{(zq)_n}
\]
for $n$ a nonnegative integer, then
\begin{equation}\label{Eq_invariance_A1}
\sum_{n=0}^{n_0} 
z^n q^{n^2}\qbin{n_0}{n} \Phi_n(z;q)=\Phi_{n_0}(z;q).
\end{equation}
Iterating \eqref{Eq_invariance_A1} yields
\[
\sum_{n_1,\dots,n_k=0}^{n_0} 
\bigg( \prod_{i=1}^k z^{n_i} q^{n_i^2} \qbin{n_{i-1}}{n_i}\bigg)
\Phi_{n_k}(z;q)=\Phi_{n_0}(z;q),
\]
from which \eqref{Eq_A1-rational} follows by letting $k$ tend to infinity.

Next let $r=3$.
Then the large-$k$ limit of Theorem~\ref{Thm_A2-vac-bothmoduli}, 
Theorem~\ref{Thm_threecases} for $s=1$ or $s=k+1$ and 
Theorem~\ref{Thm_threecases-b} for $s=1$ or $s=k$ all follow from
\begin{equation}\label{Eq_A2-infinitelevel}
\sum_{\substack{n_1,n_2,\dots\geq 0 \\[1pt] m_1,m_2,\dots\geq 0}} 
\frac{1}{(q)_{n_1}}\,
\prod_{i\geq 1} z^{n_i-m_i} w^{m_i} q^{n_i^2-m_in_i+m_i^2} 
\qbin{n_i}{n_{i+1}}\qbin{n_i-n_{i+1}+m_{i+1}}{m_i}
=\frac{1}{(zq)_{\infty}(wq)_{\infty}}.
\end{equation}
Again this is the $n_0\to\infty$ limit of a rational function identity:
\begin{align}\label{Eq_A2-rational}
\sum_{\substack{n_1,n_2,\dots\geq 0 \\[1pt] m_1,m_2,\dots\geq 0}}\,
&\prod_{i\geq 1} z^{n_i-m_i} w^{m_i} q^{n_i^2-m_in_i+m_i^2} 
\qbin{n_{i-1}}{n_i}\qbin{n_{i-1}-n_i+m_i}{m_{i-1}} \\ 
&=\frac{1}{(zq)_{n_0-m_0}(wq)_{n_0}}\,\qbin{n_0}{m_0}
\notag
\end{align}
for $n_0,m_0$ nonnegative integers.
Defining
\[
\Phi_{n,m}(z,w;q):=
\frac{1}{(zq)_{n-m}(wq)_n}\,\qbin{n}{m},
\]
the identity \eqref{Eq_A2-rational} is a consequence of the 
following $\mathrm{A}_2$-analogue of \eqref{Eq_invariance_A1}:
\begin{equation}\label{Eq_invariance_A2}
\sum_{0\leq m\leq n\leq n_0}
z^{n-m} w^m q^{n^2-mn+m^2} \qbin{n_0}{n}\qbin{n_0-n+m}{m_0} 
\Phi_{n,m}(z,w;q)=\Phi_{n_0,m_0}(z,w;q).
\end{equation}

It is not difficult to generalise \eqref{Eq_invariance_A1} and
\eqref{Eq_invariance_A2} to $\mathrm{A}_{r-1}$ and to then use this to
prove the $\mathrm{A}_{r-1}$ analogues of \eqref{Eq_A1-infinitelevel} 
and \eqref{Eq_A2-infinitelevel}.
First we prove an $\mathrm{A}_{r-1}$ rational function identity.

\begin{proposition}
\label{Prop_Ar-invariance}
Let $n_1,\dots,n_{r-1}$ be integers such that
\begin{equation}\label{Eq_order-n}
0\leq n_1\leq n_2\leq \cdots\leq n_{r-1}.
\end{equation}
Then
\begin{equation}\label{Eq_Ar}
\sum_{m_1,\dots,m_{r-1}\geq 0}\,
\prod_{i=1}^{r-1} \frac{z_i^{m_i-m_{i+1}} 
q^{m_i^2-m_im_{i+1}}}{(z_iq)_{m_1-m_{i+1}}}\,
\qbin{n_i-m_1+m_i}{m_i-m_{i+1}}
=\prod_{i=1}^{r-1} \frac{1}{(z_iq)_{n_i}},
\end{equation}
where $m_r:=0$.
\end{proposition}

As follows from the proof, the condition \eqref{Eq_order-n} is necessary.
We also note that if $C=(C_{ij})_{i,j=1}^{r-1}$ is the Cartan matrix of 
$\mathrm{A}_{r-1}$, i.e.,
\[
C_{ij}=2\delta_{i,j}-\delta_{i,j+1}-\delta_{i,j-1},
\]
then 
\[
\sum_{i=1}^{r-1} (m_i^2-m_im_{i+1})=
\frac{1}{2}\sum_{i,j=1}^{r-1} m_i C_{ij} m_j,
\]
where, as in the proposition, $m_r:=0$.
Finally, if 
$\varphi^{\mathrm{A}_{r-1}}_{n_1,\dots,n_{r-1}}(z_1,\dots,z_{r-1};q)$
denotes either side of \eqref{Eq_Ar}, then
\[
\varphi^{\mathrm{A}_r}_{n_1,\dots,n_r}
(z_1,\dots,z_{j-1},0,z_{j+1},\dots,z_r;q)=
\varphi^{\mathrm{A}_{r-1}}_{n_1,\dots,n_{j-1},n_{j+1},\dots,n_r}
(z_1,\dots,z_{j-1},z_{j+1},\dots,z_r;q).
\]

\begin{proof}[Proof of Proposition~\eqref{Prop_Ar-invariance}]
Assume that \eqref{Eq_order-n} holds, and denote the left-hand side
of \eqref{Eq_Ar} by 
$\varphi^{\mathrm{A}_{r-1}}_{n_1,\dots,n_{r-1}}(z_1,\dots,z_{r-1};q)$.
Clearly, 
\begin{equation}\label{Eq_phi-A1}
\varphi^{\mathrm{A}_1}_{n_1}(z_1;q)=
\sum_{m_1=0}^{n_1} \frac{z^{m_1} q^{m_1^2}}
{(z_1q)_{m_1}}\,\qbin{n_1}{m_1}
=\frac{1}{(z_1q)_{n_1}}
\end{equation}
by \eqref{Eq_1phi1}.
Now assume that $r\geq 2$ and replace 
$m_i\mapsto m_i+m_r$ for all $1\leq i\leq r-1$ in the
expression for $\varphi^{\mathrm{A}_r}_{n_1,\dots,n_r}$.
Then
\begin{align}\label{Eq_towards-recurrence}
\varphi^{\mathrm{A}_r}_{n_1,\dots,n_r}(z_1,\dots,z_r;q)
=\sum_{m_1,\dots,m_{r-1}\geq 0} &\bigg(
\prod_{i=1}^{r-1} \frac{z_i^{m_i-m_{i+1}} 
q^{m_i^2-m_im_{i+1}}}{(z_iq)_{m_1-m_{i+1}}}\,
\qbin{n_i-m_1+m_i}{m_i-m_{i+1}} \\
&\quad \times
\frac{1}{(z_rq)_{m_1}}\,
\phi_{n_r-m_1}^{\mathrm{A}_1}(z_rq^{m_1};q) \bigg), \notag
\end{align}
where, now, $m_r:=0$.
We would like to use \eqref{Eq_phi-A1} (or, equivalently \eqref{Eq_1phi1})
to simplify the second line on the right of \eqref{Eq_towards-recurrence}
to
\[
\frac{1}{(z_rq)_{m_1}(z_rq^{m_1+1})_{n_r-m_1}}=\frac{1}{(z_rq)_{n_r}},
\]
resulting in
\begin{equation}\label{Eq_rec}
\varphi^{\mathrm{A}_r}_{n_1,\dots,n_r}(z_1,\dots,z_r;q) 
=\frac{1}{(zq)_{n_r}}\,
\varphi^{\mathrm{A}_{r-1}}_{n_1,\dots,n_{r-1}}(z_1,\dots,z_{r-1};q).
\end{equation}
Caution is required, however, since \eqref{Eq_1phi1} is false if
$n$ is a negative integer, with the left-but not the right-hand side
vanishing for such $n$.
We thus need to verify that the summand in \eqref{Eq_towards-recurrence}
still vanishes for $m_1>n_r$ after replacing the second line by
$1/(z_rq)_{n_r}$.
To this end we note that the $q$-binomial coefficient corresponding to the
$i=r-1$ term in the product is given by 
\[
\qbin{n_{r-1}-m_1+m_{r-1}}{m_{r-1}}.
\]
The summand thus vanishes if $m_1>n_{r-1}$.
Since $n_{r-1}\leq n_r$ this implies the desired vanishing for
$m_1>n_r$, so that the recursion relation \eqref{Eq_rec} is valid.
By induction this completes the proof.
\end{proof}

We are now ready to state the $\mathrm{A}_{r-1}$-analogue
of \eqref{Eq_invariance_A1} and \eqref{Eq_invariance_A2}.
Let $\boldsymbol{n}=(n_1,\dots,n_{r-1})\in\mathbb{N}_0^{r-1}$,
and define the rational function
\[
\Phi^{\mathrm{A}_{r-1}}_{\boldsymbol{n}}(z_1,\dots,z_{r-1};q):=
\prod_{i=1}^{r-1} \frac{1}{(z_iq)_{n_1-n_{i+1}}}\,\qbin{n_i}{n_{i+1}},
\]
where $n_r:=0$.
Further define the polynomial
\[
\mathcal{K}_{\boldsymbol{n},\boldsymbol{m}}(z_1,\dots,z_{r-1};q):=
\qbin{n_1}{m_1} 
\prod_{i=1}^{r-1} z_i^{m_i-m_{i+1}} q^{m_i^2-m_im_{i+1}}
\qbin{n_1-n_{i+1}-m_1+m_i}{n_i-n_{i+1}},
\]
where $n_r=m_r:=0$.

\begin{theorem}\label{Thm_Ar-invariance}
For $\boldsymbol{n}\in\mathbb{N}_0^{r-1}$,
\begin{equation}\label{Eq_invariance_Ar}
\sum_{\boldsymbol{m}}
\mathcal{K}_{\boldsymbol{n},\boldsymbol{m}}(z_1,\dots,z_{r-1};q)
\Phi^{\mathrm{A}_{r-1}}_{\boldsymbol{m}}(z_1,\dots,z_{r-1};q)
=\Phi^{\mathrm{A}_{r-1}}_{\boldsymbol{n}}(z_1,\dots,z_{r-1};q),
\end{equation}
where $\boldsymbol{m}$ is summed over $\mathbb{N}_0^{r-1}$.
\end{theorem}

\begin{proof}
In \eqref{Eq_Ar} we carry out the simultaneous substitutions
\[
(n_1,n_2,\dots,n_{r-1})\mapsto(n_1-n_2,n_1-n_3,\dots,n_1-n_r),
\]
where $n_r:=0$. 
Then
\[
\sum_{m_1,\dots,m_{r-1}\geq 0}\,
\prod_{i=1}^{r-1} \frac{z_i^{m_i-m_{i+1}} 
q^{m_i^2-m_im_{i+1}}}{(z_iq)_{m_1-m_{i+1}}}
\qbin{n_1-n_{i+1}-m_1+m_i}{m_i-m_{i+1}}
=\prod_{i=1}^{r-1} \frac{1}{(z_iq)_{n_1-n_{i+1}}},
\]
where $n_r=m_r:=0$ and 
$n_1\geq n_2\geq\cdots\geq n_{r-1}\geq 0$.
Multiplying both sides by
\begin{equation}\label{Eq_factor}
\prod_{i=1}^{r-1}\qbin{n_i}{n_{i+1}}
\end{equation}
and using that
\[
\prod_{i=1}^{r-1} \qbin{n_i}{n_{i+1}}
\qbin{n_1-n_{i+1}-m_1+m_i}{m_i-m_{i+1}}
=\qbin{n_1}{m_1} \prod_{i=1}^{r-1} 
\qbin{m_i}{m_{i+1}}\qbin{n_1-n_{i+1}-m_1+m_i}{n_i-n_{i+1}}
\]
yields \eqref{Eq_invariance_Ar}.
Since \eqref{Eq_factor} vanishes unless 
$n_1\geq n_2\geq\cdots\geq n_{r-1}\geq 0$ holds, the claim
is true for all $\boldsymbol{n}\in\mathbb{N}_0^{r-1}$.
\end{proof}

Iterating \eqref{Eq_invariance_Ar} immediately leads to
\[
\sum_{\boldsymbol{n}^{(1)},\dots,\boldsymbol{n}^{(k)}}
\bigg( \prod_{i=1}^k
\mathcal{K}_{\boldsymbol{n}^{(i-1)},\boldsymbol{n}^{(i)}}
(z_1,\dots,z_{r-1};q)\bigg)
\Phi^{\mathrm{A}_{r-1}}_{\boldsymbol{n}^{(k)}}(z_1,\dots,z_{r-1};q)
=\Phi^{\mathrm{A}_{r-1}}_{\boldsymbol{n}^{(0)}}(z_1,\dots,z_{r-1};q)
\]
for $k$ a positive integer.
Taking the large-$k$ limit this yields
\[
\sum_{\boldsymbol{n}^{(1)},\boldsymbol{n}^{(2)},\dots}\:
\prod_{i\geq 1} \mathcal{K}_{\boldsymbol{n}^{(i-1)},\boldsymbol{n}^{(i)}}
(z_1,\dots,z_{r-1};q)
=\Phi^{\mathrm{A}_{r-1}}_{\boldsymbol{n}^{(0)}}(z_1,\dots,z_{r-1};q),
\]
generalising \eqref{Eq_A1-rational} and \eqref{Eq_A2-rational}.
As a final step we let $n^{(0)}_1$ tend to infinity, resulting in our
final theorem.
\begin{theorem}\label{Thm_infinite}
We have
\begin{align*}
\sum \frac{1}{(q)_{n^{(1)}_1}} \prod_{i\geq 1} 
\Bigg( & \prod_{j=1}^{r-1} z_j^{n^{(i)}_j-n^{(i)}_{j+1}} 
q^{n^{(i)}_j(n^{(i)}_j-n^{(i)}_{j+1})} \\
& \times \qbinbig{n^{(i)}_1}{n^{(i+1)}_1} 
\prod_{j=2}^{r-1} 
\qbinbig{n^{(i)}_1-n^{(i)}_{j+1}-n^{(i+1)}_1+n^{(i+1)}_j}
{n^{(i)}_j-n^{(i)}_{j+1}} \Bigg)
=\prod_{j=1}^{r-1} \frac{1}{(z_jq)_{\infty}},
\end{align*}
where the sum is over nonnegative integers $n^{(i)}_j$
for all $i\geq 1$ and $1\leq j\leq {r-1}$, and where
$n_r^{(1)}=n_r^{(2)}=\dots:=0$.
\end{theorem}

For $z_1=\dots=z_{r-1}=1$ this should correspond to
\[
\lim_{k\to\infty} \AG_{k\La_0+\dots+k\La_{r-2}+(k-1)\La_{r-1};r}(q)=
\lim_{k\to\infty} \AG_{k\La_0+(k-1)\La_1+\dots+(k-1)\La_{r-1};r}(q).
\]
For example, if for $r=4$ we define
\begin{align*}
S_k(q)&:=\AG_{k\La_0+k\La_1+k\La_2+(k-1)\La_3;4}(q) \\
&\hphantom{:}=\frac{(q^{4k+3};q^{4k+3})_{\infty}^3}{(q)_{\infty}^3}\,
\theta(q^k,q^{k+1},q^{k+1},q^{k+1},q^{2k+1},q^{2k+1};q^{4k+3})
\end{align*}
and let $S_{\infty}(q):=\lim_{k\to\infty} S_k(q)$,
then we have the two extremal cases
\[
S_1(q)=
\sum_{n,m=0}^{\infty} \frac{q^{n^2-nm+m^2}}{(q)_n}\,
\qbin{2n}{m}
\]
and 
\begin{align*}
S_{\infty}(q)=
\sum_{\substack{n_1,n_2,\dots\geq 0 \\[1pt] 
m_1,m_2,\dots\geq 0 \\[1pt] 
p_1,p_2,\dots\geq 0}} 
\frac{1}{(q)_{n_1}} \prod_{i\geq 1} 
\bigg( & q^{n_i^2-n_im_i+m_i^2-m_ip_i+p_i^2} 
\qbin{n_i}{n_{i+1}}  \\[-2mm] & \times
\qbin{n_i-p_i-n_{i+1}+m_{i+1}}{m_i-p_i} 
\qbin{n_i-n_{i+1}+p_{i+1}}{p_i} \bigg).
\end{align*}
It is not yet clear to us how to interpolate between these two results,
but one further piece of the puzzle follows from the generalisation of
Conjecture~\ref{Con_CDU-Q} to arbitrary rank $r$.
Recall that by Kummer's theorem $(d+r)/\gcd(d,r)$ divides $\binom{d+r}{r}$.

\begin{conjecture}
Fix $d,r\geq 1$ and let $c=(c_0,\dots,c_{r-1})$ be a profile of level $d$, 
i.e., $c_0+\dots+c_{r-1}=d$. 
Then the formal power series $Q_{n,c}(q)$ defined in \eqref{Eq_Qnc} is 
a polynomial in $q$ with nonnegative coefficients, such that
\[
Q_{n,c}(1)=\bigg(\frac{\gcd(d,r)}{d+r}\,\binom{d+r}{r}-\gcd(d,r)\bigg)^n.
\]
\end{conjecture}

This Conjecture was also proposed by Corteel \cite{Corteel}, and the
polynomiality and value at $1$ have once again been established by
Welsh \cite{Welsh21}.
We in particular have
\[
S_2(q)=\sum_{n=0}^{\infty} \frac{Q_{n,(2,2,2,1)}(q)}{(q)_{n}},
\]
where $Q_{n,(2,2,2,1)}(1)=29^n$.
For small $n$ the polynomial $Q_{n,(2,2,2,1)}(q)$ can be computed explicitly,
and the first few polynomials are given by
\begin{align*}
Q_{0,(2,2,2,1)}(q)&=1, \\
Q_{1,(2,2,2,1)}(q)&=q(3+5q+5q^2+5q^3+4q^4+3q^5+2q^6+q^7+q^8), \\
Q_{2,(2,2,2,1)}(q)&=q^3(3+12q+20q^2+32q^3+39q^4+49q^5+52q^6+57q^7+56q^8
+58q^9+53q^{10} \\
&\qquad +53q^{11}+48q^{12}+46q^{13}+39q^{14}+38q^{15}+31q^{16}
+29q^{17}+24q^{18}+21q^{19} \\
& \qquad +16q^{20}+16q^{21}+11q^{22}+10q^{23}+7q^{24}+6q^{25}
+4q^{26}+4q^{27}+2q^{28} \\
& \qquad +2q^{29}+q^{30}+q^{31}+q^{33}) \\
&=3q^3+12q+\cdots+q^{34}+q^{36}.
\end{align*}
This rules out that
$Q_{n_1,(2,2,2,1)}(q)$ is of the form
\begin{align*}
Q_{n_1,(2,2,2,1)}(q)=\sum_{n_2,m_1,m_2,p_1\geq 0} \bigg(&
q^{n_1^2-n_1m_1+m_1^2-m_1p_1+p_1^2+n_2^2-n_2m_2+m_2^2} \\ &
\times
\qbin{n_1}{n_2} \qbin{2n_2}{m_2} \qbin{n_1-n_2-p_1+m_2}{m_1-p_1}
\qbin{\dots}{p_1}\bigg),
\end{align*}
since, regardless of the precise form of $\qbin{\dots}{p_1}$,
\[
n_1^2-n_1m_1+m_1^2-m_1p_1+p_1^2+n_2^2-n_2m_2+m_2^2\Big|_{n_1=2}\neq 36
\]
for any $n_2,m_1,m_2,p_1\in\mathbb{N}_0$ such that $n_2\leq 2$,
$m_2\leq 2n_2$, $m_1\geq p_1$ and $2-n_2+m_2\geq m_1$.

\medskip

\subsection*{Note added}
Recently Kanade and Russell \cite[Conjecture~5.1]{KR22} further 
generalised Conjecture~\ref{Con_missing}, proposing a multisum expression
for 
\[
\frac{(q^K;q^K)_{\infty}^2}{(q)_{\infty}^3}\,
\theta(q^r,q^s,q^{r+s};q^K)
\]
for all $1\leq r,s\leq \lfloor(K+1)/3\rfloor$.
Their conjecture also implies generalisations of 
Conjectures~\ref{Con_cylindric} and \ref{Con_cylindric-b}, extending
both to a much larger set of profiles of level $3k-1$ and $3k-2$
respectively while also allowing for profiles of level $3k$.

\subsection*{Acknowledgements}
Helpful discussions with Sylvie Corteel, Ali Uncu and Trevor Welsh are
gratefully acknowledged.

\bibliographystyle{alpha}

\end{document}